\documentclass{amsart}
\usepackage{enumitem}
\usepackage{subcaption}

\usepackage{mathtools}
\usepackage[foot]{amsaddr}

\newtheorem{theorem}{Theorem}[section]
\newtheorem{lemma}[theorem]{Lemma}
\newtheorem{corollary}[theorem]{Corollary}
\theoremstyle{definition}
\newtheorem{definition}[theorem]{Definition}

\theoremstyle{remark}
\newtheorem{remark}[theorem]{Remark}

\begin{document}
	\title[$N$-Body Dielectric Spheres Problem. Part I.]{An Integral Equation Formulation of the $N$-Body Dielectric Spheres Problem. Part I: Numerical Analysis}
	%
	\author{$\text{Muhammad Hassan}^{\dagger}$}
	\address[$\dagger$]{Center for Computational Engineering Science, Department of Mathematics, RWTH Aachen University, Germany}
	\email{hassan@mathcces.rwth-aachen.de}
	
	\author{$\text{Benjamin Stamm}^{\dagger}$}
	\email{stamm@mathcces.rwth-aachen.de}
	\date{\today}
	\begin{abstract} 
		In this article, we analyse an integral equation of the second kind that represents the solution of $N$ interacting dielectric spherical particles undergoing mutual polarisation. A traditional analysis can not quantify the scaling of the stability constants- and thus the approximation error- with respect to the number $N$ of involved dielectric spheres. We develop a new a priori error analysis that demonstrates $N$-independent stability of the continuous and discrete formulations of the integral equation. Consequently, we obtain convergence rates that are independent of $N$.
	\end{abstract}
	%
	%
	\maketitle
	\section{Introduction}
	
	The so-called \emph{$N$-body problem} is a general term used to describe a vast category of physical problems involving the interaction of a large number of objects. Such problems arise in a variety of contexts in fields as diverse as quantum mechanics, molecular dynamics, astrophysics and electrostatics. The origin of the $N$-body problem lies in the Principia Mathematica wherein Newton considered the motion of celestial bodies. \cite{newton1934principia}. Starting with the work of {Henri Poincar\'e} \cite{poincare}, which incidentally led to the development of chaos theory, a significant amount of evidence has been accumulated that obtaining an analytic solution to the $N$-body problem in a tractable manner is not possible (see, e.g., \cite{DongWang, Sundman}). As a consequence, there has been a great deal of interest in developing numerical methods that can efficiently compute approximate solutions to the $N$-body problem. An important benchmark to assess the quality of any such numerical method has been its ability to obtain a linear scaling, i.e., given a system composed of $N$ interacting objects, to achieve time and computational complexity of order $\mathcal{O}(N)$. Attempts to achieve this benchmark have led to the development of extremely efficient numerical algorithms such as fast multipole (FMM) and particle mesh methods, which have been applied very successfully to a variety of $N$-body problems (see, e.g., \cite{greengard2, greengard1} for an explanation of the FMM and \cite{hockney} for particle mesh methods). \vspace{2mm}
	
	In the discipline of chemical physics, the interactions between charged particles in concentrated colloidal solutions (see, e.g., \cite{Colloid}) and Coulombic crystals (see, e.g., \cite{Crystal}), or the phenomena of electrostatic self-assembly (see, e.g., \cite{Self}) and super lattices (see, e.g., \cite{Lattice}) are all examples of $N$-body problems in electrostatics, and an accurate description of the electrostatic forces between the interacting particles is necessary in order to understand the physics underlying each of these phenomena. Until quite recently however, the state-of-the-art for the computation of the electrostatic forces between a large number of charged particles was quite under-developed. Most results in the literature relied on so-called image charge methods or multipole expansion approaches (see, e.g., \cite{image1,image2, image3} for the former, \cite{multipole1,lotan2006, multipole2} for the latter and \cite{barros2014efficient, freed2014perturbative} for other approches). The key deficiency of such numerical methods is that they have often not been formulated in a manner which allows a systematic numerical analysis of the algorithm. Recently, in \cite{lindgren2018}, the authors proposed a computational method based on a Galerkin discretisation of an integral equation formulation of the second kind for the induced surface charges resulting on a large number of dielectric spheres of varying radii and dielectric constants, embedded in a homogenous dielectric medium and undergoing mutual polarisation. Numerical experiments indicated that this algorithm displayed some interesting behaviour:
	\begin{enumerate}
		\item {For a fixed number of degrees of freedom per sphere}, the average error on each sphere remained bounded when increasing the number of dielectric spheres;
		\item For a fixed number of dielectric spheres, the total error decayed exponentially when increasing the degrees of freedom per sphere;
		\item Through the use of the FMM, the numerical method achieved computational complexity that scaled linearly with respect to the number of dielectric spheres.
	\end{enumerate}
	
	\vspace{2mm}
	
	Points (1) and (2) deal with the accuracy of the method and point (3) gives a measure of the computational scalability of the numerical algorithm. Taken together, these numerical observations suggest that the method proposed in \cite{lindgren2018} is \emph{linear scaling in accuracy}, i.e., in order to obtain an approximate solution with fixed average (the total error scaled by $N$) or relative error, the computational cost of the algorithm scales as $\mathcal{O}(N)$. Consequently, the integral equation-based approach proposed by Lindgren at al. is a significant advance in the state-of-the-art for the computation of the electrostatic interactions between a large number of charged particles undergoing mutual polarisation. \vspace{2mm}

	It is now natural to ask if one can provide a rigorous mathematical justification for the behaviour exhibited by the numerical method in points (1)-(3). More precisely, can one prove that the average or relative error is bounded independent of the number of objects in the problem? And that the computational complexity of the numerical method proposed in \cite{lindgren2018} scales linearly with respect to the number of objects in the problem? The current article is the first in a series of two and focuses on the numerical analysis of the algorithm introduced in \cite{lindgren2018} in order to provide a mathematically sound answer to the first question. More specifically, we prove that
	
	\begin{enumerate}[label=(\alph*)]
		\item {For any fixed geometrical configuration of non-intersecting spherical dielectric particles,} the integral equation formulation of the second kind proposed in \cite{lindgren2018} that describes the induced surface charges resulting on these dielectric spheres undergoing mutual polarisation is well-posed;
		\item {For any fixed geometrical configuration of non-intersecting spherical dielectric particles,} the Galerkin discretisation of this second-kind integral equation is also well-posed;
		\item { For any fixed geometrical configuration of non-intersecting spherical dielectric particles,} there exists an upper bound on the relative error of the approximate solution that does not \underline{explicitly} depend on the number $N$ of dielectric spheres in the system. {Consequently, we can deduce $N$-independent error estimates for any family of geometrical configurations that satisfies certain geometrical assumptions which are described in detail later;}
		\item { For any fixed geometrical configuration of non-intersecting spherical dielectric particles,} given certain assumptions on the regularity of the exact solution, the total error of the approximate solution decays exponentially as the degrees of freedom per sphere are increased.
	\end{enumerate}
	
	A detailed complexity analysis of this numerical method which provides a mathematically sound answer to the second question is the subject of the contribution \cite{Hassan2}. \vspace{2mm}
	
	{$N$-body problems have been widely studied in the literature in the context of electromagnetic or acoustic scattering by a large number of obstacles (see, e.g., \cite{Scattering_New4, Scattering_New2, Scattering_New3, Scattering_New5, Scattering_New6, Scattering_New7, Scattering_New1}). Such scattering problems are significantly more complicated to analyse than the electrostatic interaction problem we consider here because the underlying differential operator in wave phenomena is indefinite, which causes many technical difficulties. Consequently, it is already a significant challenge to design a computationally efficient numerical algorithm that is stable with respect to a large regime of wave numbers and obstacle sizes and placements, let alone perform a comprehensive numerical analysis of the algorithm and derive rigorous convergence rates. Thus, articles such as~\cite{Scattering_New4, Scattering_New6, Scattering_New7} quoted above focus mostly on the design of efficient computational methods and use numerical tests to validate their proposals. On the other hand while articles such as \cite{Costabel1985, Scattering_New5} do establish convergence rates with respect to the degrees of freedom, these rates are not shown to be independent of the number of obstacles $N$. Incidentally, several of the articles mentioned above propose algorithms that are based on integral equations of the first kind (see \cite{Scattering_New2, Scattering_New3, Costabel1985, Scattering_New1} quoted above). This framework, while suitable for numerical analysis, suffers from a well-known problem: Galerkin discretisations of integral equations of the first kind typically lead to dense, ill-conditioned solution matrices which causes slow convergence of the iterative solvers. As a remedy, several of these articles have proposed efficient preconditioners (see e.g., the article \cite{Scattering_New2}) but the introduction of preconditioning techniques further complicates questions of scalability. This computational deficiency is precisely why Lindgren et al.\cite{lindgren2018} proposed and why we analyse an integral equation formulation of the \textit{second~kind}.} \vspace{2mm}
	
	There is an abundant literature on integral equations of the second kind (see, e.g., the books \cite{Wendland1, Schwab}, or the articles \cite{2ndkind2, 2ndkind1, Elschner1, Elschner2, 2ndkind3, Steinbach1}). In particular, the well-posedness theory of second kind integral equations is well established, and it is understood that Galerkin discretisations of second kind integral equations typically leads to well conditioned solution matrices. As a consequence, second kind integral equations have been constructed for the solutions of a variety of problems. More recently, such formulations have also been proposed for problems very similar to the $N$-body dielectric sphere problem including, for instance, acoustic and electromagnetic scattering by composite structures (see, e.g., \cite{Xavier1, Xavier2, Rokhlin, mullerEM}), and multi subdomain diffusion \cite{XavierDiff}. The key mathematical deficiency of such second kind integral formulations is that stability estimates- and thus also error estimates- are often difficult to obtain except in certain special cases.  \vspace{2mm}
	
	Therefore, obtaining stability and error estimates for our problem using the existing well-posedness analysis in the literature is not straight forward. To make matters worse, most integral equations are applied in situations where the size of the domain is \emph{fixed} so the existing analysis in the literature focuses on establishing the existence of stability and continuity constants of the boundary integral operators that are independent of the degrees of freedom, such as the mesh width or the boundary element size. Since the stability and continuity constants appear in the error estimates, it is crucial to establish that these constants are explicitly independent of \emph{the number of objects in the problem setting}. Unfortunately, this is not a priori clear and in some cases is not even true for the classical well-posedness analysis. Consequently, in order to prove points (c)-(d), we have had to introduce \emph{a new well-posedness analysis} for establishing points (a)-(b). All these issues are discussed in more detail in Section 2.5. \vspace{2mm}
	
	The remainder of this article is organised as follows. In Section 2, we describe the problem setting, state and discuss our main results, and consider the limitations of the existing classical analysis of second kind integral equations in the literature. Section 3 then contains numerical experiments that validate our theoretical results. In Section 4, we state intermediate lemmas and the proofs of our main results. Finally, in Section 5, we present our conclusion and discuss future directions of research.
	
	\section{Problem Setting and Main Results}\label{sec:2}
	Throughout this article, we will use standard results and notation from the theory of integral equations. We follow the notation of, and use as the primary reference, the book of Sauter and Schwab on boundary elements methods \cite{Schwab}.
	
	\subsection{Setting and Notation}\label{sec:2a}
	{To begin with we would like to describe precisely the types of geometrical situations we will consider in this article. As indicated in the introduction, we are interested in studying geometrical configurations that are the unions of an arbitrary number $N$ of non-intersecting open balls with varying radii in three dimensions. However, in order to be completely rigorous in our claim of $N$-independent error estimates, we must impose certain assumptions on the types of geometries we consider. To this end, let $\mathcal{I}$ denote a countable indexing set. We consider a so-called family of geometries~$\{\Omega_{\mathcal{F}}\}_{\mathcal{F} \in \mathcal{I}}$. Each element $\Omega_{\mathcal{F}} \subset \mathbb{R}^3$ in this family is the (set) union of a fixed number of non-intersecting open balls of varying locations and radii with associated dielectric constants, and therefore represents a particular physical geometric situation. It is easy to see that each element $\Omega_{\mathcal{F}}$ of this family of geometries is uniquely determined by the following four parameters:
		\begin{itemize}
			\item A non-zero number $N_{\mathcal{F}} \in \mathbb{N}$, which represents the total number of dielectric spherical particles that compose the geometry $\Omega_{\mathcal{F}}$;
			\item A collection of points $\{\bold{x}^{\mathcal{F}}_i\}_{i=1}^{N_{\mathcal{F}}} \in \mathbb{R}^3$, which represent the centres of the spherical particles composing the geometry $\Omega_{\mathcal{F}}$;
			\item A collection of positive real numbers $\{r_i^{\mathcal{F}}\}_{i=1}^{N_{\mathcal{F}}} \in \mathbb{R}$, which represent the radii of the spherical particles composing the geometry $\Omega_{\mathcal{F}}$;
			\item A collection of positive real numbers $\{\kappa^{\mathcal{F}}_i\}_{i=0}^N \in \mathbb{R}$. Here, $\kappa^{\mathcal{F}}_0$ denotes the dielectric constant of the external medium while $\{\kappa^{\mathcal{F}}_i\}_{i=1}^N$ represent the dielectric constants of each dielectric sphere.
		\end{itemize}
		
		Indeed, using the first three parameters we can define the open balls $\Omega^\mathcal{F}_i := \mathcal{B}_{r_i}(\bold{x}_i) \subset \mathbb{R}^3$, $i \in \{1, \ldots, N_{\mathcal{F}}\}$ which represent the spherical dielectric particles composing the geometry $\Omega_{\mathcal{F}}$, i.e., $\Omega_{\mathcal{F}}= \cup_{i=1}^{N_\mathcal{F}} \Omega_i^{\mathcal{F}}$. Moreover, the fourth parameter $\{\kappa^{\mathcal{F}}_i\}_{i=0}^N$ denotes the dielectric constants associated with this geometry.

		\vspace{3mm}
		We now impose the following three important assumptions on the above parameters:\\
		
		\begin{enumerate}
			
			\item[\textbf{A1:}] \textbf{[Uniformly bounded radii]} There exist constants $r^{\infty}_->0$ and $r^{\infty}_+>0$ such that 
			\begin{align*}
			\inf_{\mathcal{F} \in \mathcal{I}}\, \min_{i=1, \ldots, N_{\mathcal{F}}} r^{\mathcal{F}}_i > r^{\infty}_- \quad \text{and} \quad \sup_{\mathcal{F} \in \mathcal{I}}\, \max_{i=1, \ldots, N_{\mathcal{F}}} r^{\mathcal{F}}_i < r^{\infty}_+.
			\end{align*}
			
			\item[\textbf{A2:}] \textbf{[Uniformly bounded minimal separation]} There exists a constant $\epsilon^{\infty} > 0$ such that 
			\begin{align*}
			\inf_{\mathcal{F} \in \mathcal{I}}\, \min_{\substack{i, j=1, \ldots, N_{\mathcal{F} } \\ i \neq j}} \big(\vert \bold{x}_i^{\mathcal{F}} -\bold{x}_j^{\mathcal{F}}\vert - r^{\mathcal{F}}_i -r^{\mathcal{F}}_j\big)> \epsilon^{\infty}.
			\end{align*}
			
			\item[\textbf{A3:}] \textbf{[Uniformly bounded dielectric constants]} There exist constants $\kappa^{\infty}_->0$ and $\kappa^{\infty}_+>0$ such that 
			\begin{align*}
			\inf_{\mathcal{F} \in \mathcal{I}} \,\min_{i=1, \ldots, N_{\mathcal{F}}} \kappa^{\mathcal{F}}> \kappa^{\infty}_- \quad \text{and} \quad \sup_{\mathcal{F} \in \mathcal{I}} \, \max_{i=1, \ldots, N_{\mathcal{F}}} \kappa^{\mathcal{F}} < \kappa^{\infty}_+.
			\end{align*}
			
		\end{enumerate}
		In other words we assume that the family of geometries $\{\Omega_{\mathcal{F}}\}_{\mathcal{F} \in \mathcal{I}}$ we consider in this article describe physical situations where the radii of the dielectric spherical particles, the minimum inter-sphere separation distance and the dielectric constants are all uniformly bounded. These assumptions are necessary because, as we will show, the error estimates we derive, while explicitly independent of the number of dielectric particles $N_\mathcal{F}$, do depend on other geometrical parameters, and we would thus like to avoid situations where these geometric parameters degrade with increasing $N_\mathcal{F}$. We remark that from a practical perspective, these assumptions do not greatly limit the scope of our results. Indeed, in many physical applications one typically considers non-metallic dielectric particles which neither have vanishing or exploding dielectric constants nor vanishing or exploding radii (see, e.g., \cite{lindgren_dynamic, lee2015direct, Crystal, Self, soh2014charging}). \\ 
		
		In the remainder of this article, we will consider a fixed geometry from the family of geometries $\{\Omega_{\mathcal{F}}\}_{\mathcal{F} \in \mathcal{I}}$ satisfying the assumptions \textbf{A1)-A3)}. To avoid bulky notation we will drop the superscript and subscript $\mathcal{F}$ and denote this geometry by $\Omega^-$. The geometry is constructed as follows: Let $N \in \mathbb{N}$, let $\{\bold{x}_i\}_{i=1}^N \in \mathbb{R}^3$ be a collection of points in $\mathbb{R}^3$ and let $\{r_i\}_{i=1}^N \in \mathbb{R}$ be a collection of positive real numbers, and for each $i \in \{1, \ldots, N\}$ let $\Omega_i := \mathcal{B}_{r_i}(\bold{x}_i) \subset \mathbb{R}^3$ be the open ball of radius $r_i >0$ centred at the point $\bold{x}_i$. Then $\Omega^- \subset \mathbb{R}^3$ is defined as $\Omega^-:= \cup_{i=1}^N \Omega_i$. Furthermore, we define $\Omega^+:= \mathbb{R}^3 \setminus \overline{\Omega^-}$, and we write $\partial \Omega$ for the boundary of $\Omega^-$ and {$\eta(\bold{x})$ for the unit normal vector at $\bold{x} \in \partial \Omega$ pointing towards the exterior of $\Omega^-$}.}\\

	Next, let $\{\kappa_i\}_{i=0}^N \in \mathbb{R}$ be a collection of positive real numbers and let the function $\kappa \colon \partial \Omega \rightarrow \mathbb{R}$ be defined as $\kappa (\bold{x}):= \kappa_i ~~~\text{ for } \bold{x} \in \partial \Omega_i.$ Thus, $\kappa$ is a piecewise constant function that takes constant positive values on the boundary of each open ball $\partial \Omega_i, ~ i=1, \ldots, N$. Physically, this function represents the dielectric constant of each of these open balls while the constant $\kappa_0$ represents the dielectric constant of the medium. We observe that by definition for each $i \in \{1, \ldots, N\}$, either $\frac{\kappa-\kappa_0}{\kappa_0}\vert_{\partial \Omega_i} \geq 0$ or $\frac{\kappa-\kappa_0}{\kappa_0}\vert_{\partial \Omega_i} \in (-1, 0]$. \\
	
	Following standard practice, we write $H^1(\Omega^-):= \left\{u \in L^2(\Omega^-) \colon \nabla u \in L^2(\Omega^-)\right\}$ with the norm $\Vert u \Vert^2_{H^1(\Omega^-)}:= \sum_{i=1}^N\Vert u \Vert^2_{L^2(\Omega_i)} + \Vert \nabla u \Vert^2_{L^2(\Omega_i)}$. Moreover, inspired by the definition in \cite[Section 2.9.2.4]{Schwab} we define the weighted Sobolev space $H^1(\Omega^+)$ as the completion of $C^{\infty}_{\text{comp}}(\Omega^+)$ with respect to the norm $\Vert u \Vert^2_{H^1(\Omega^+)}:= \int_{\Omega^+} \frac{\vert v(\bold{x})\vert^2}{1 + \vert \bold{x}\vert^2}\, d\bold{x}+ \Vert \nabla v\Vert^2_{L^2(\Omega^+)}$. Note that functions that satisfy the decay conditions associated with exterior Laplace problems will belong to this space.\\


	Next, we denote by $H^{\frac{1}{2}}(\partial \Omega)$ the Sobolev space of order $\frac{1}{2}$ equipped with the Sobolev-Slobodeckij norm $\Vert \lambda \Vert^2_{H^{\frac{1}{2}}(\partial \Omega)}:=\sum_{i=1}^N \Vert \lambda\Vert^2_{L^2(\partial \Omega_i)} + \int_{\partial \Omega_i} \int_{\partial \Omega_i} \frac{\vert\lambda(\bold{x})-\lambda(\bold{y})\vert^2}{\vert \bold{x} - \bold{y}\vert^3 } \, d\bold{x} d\bold{y}$. Notice that we have chosen to define $\Vert \cdot \Vert^2_{H^{\frac{1}{2}}(\partial \Omega)}$ as a sum of local norms on each sphere. Moreover, we define $H^{-\frac{1}{2}}(\partial \Omega):=\left(H^{\frac{1}{2}}(\partial \Omega)\right)^*$ and we equip this Sobolev space with the canonical dual norm 
	\begin{align*}
	\Vert \sigma \Vert_{H^{-\frac{1}{2}}(\partial \Omega)}:= \sup_{0 \neq \psi \in {H}^{\frac{1}{2}}(\partial \Omega)} \frac{\langle \sigma, \psi \rangle_{H^{-\frac{1}{2}}(\partial \Omega) \times H^{\frac{1}{2}}(\partial \Omega)}}{\Vert \psi \Vert_{H^{\frac{1}{2}} (\partial \Omega)}} \qquad \forall \sigma \in H^{-\frac{1}{2}}(\partial \Omega).
	\end{align*}
	We remark that using the Lebesgue space $L^2(\partial \Omega)$ as a pivot space for $H^{\frac{1}{2}}(\partial \Omega)$ and $H^{-\frac{1}{2}}(\partial \Omega)$, we obtain that the duality pairing $\langle \cdot, \cdot \rangle_{H^{-\frac{1}{2}}(\partial \Omega) \times H^{\frac{1}{2}}(\partial \Omega)}$ reduces to the usual $L^2$ inner product $( \cdot, \cdot )_{L^2(\partial \Omega)}$ for sufficiently regular functions (see, e.g., \cite[Chapter 2]{Schwab}).	{For the sake of brevity, when there is no possibility of confusion, we will use the notation $\langle \cdot, \cdot \rangle_{\partial \Omega}$ to denote the duality pairing $\langle \cdot, \cdot \rangle_{H^{-\frac{1}{2}}(\partial \Omega) \times H^{\frac{1}{2}}(\partial \Omega)}$.} \\
	
	We introduce $\gamma^- \colon H^1(\Omega^{-}) \rightarrow H^{\frac{1}{2}}(\partial \Omega)$ and $\gamma^+ \colon H^1(\Omega^+) \rightarrow H^{\frac{1}{2}}(\partial \Omega)$ as the continuous, linear and surjective interior and exterior Dirichlet trace operators respectively (see, for example, \cite[Theorem 2.6.8, Theorem 2.6.11]{Schwab} or \cite[Theorem 3.38]{McLean}). Moreover, for $s \in \{+, -\}$, we define the closed subspace $\mathbb{H}(\Omega^s):= 
	\{u \in H^{1}(\Omega^s) \colon \Delta u =0 \text{ in } \Omega^s\},$
	and we write $\gamma^-_N \colon \mathbb{H}(\Omega^-) \rightarrow H^{-\frac{1}{2}}(\partial \Omega)$ and $\gamma^+_N \colon \mathbb{H}(\Omega^+) \rightarrow H^{-\frac{1}{2}}(\partial \Omega)$ for the interior and exterior Neumann trace operator respectively (see \cite[Theorem 2.8.3]{Schwab} for precise conventions). Note that both the interior and exterior Dirichlet and Neumann trace operators can be defined analogously for functions of appropriate regularity defined on $\Omega^- \cup \Omega^+$ or $\mathbb{R}^3$. {In addition, we introduce the so-called (interior) Dirichlet-to-Neumann map $\text{DtN} \colon H^{\frac{1}{2}}(\partial \Omega) \rightarrow H^{-\frac{1}{2}}(\partial \Omega)$ as follows: Given any $\lambda \in H^{\frac{1}{2}}(\partial \Omega)$, let $u_{\lambda} \in \mathbb{H}(\Omega^-)$ denote the unique harmonic function in $H^1(\Omega^-)$ such that $\gamma^- u_{\lambda} = \lambda$. Then we define $\text{DtN}\lambda:= \gamma^-_{N} u_{\lambda}$.} Note that local Dirichlet-to-Neumann maps can be defined analogously on each sphere $\partial \Omega_i, ~i=1, \ldots, N$.\\


	{Next, for each $\nu \in H^{-\frac{1}{2}}(\partial \Omega),~ \lambda \in H^{\frac{1}{2}}(\partial \Omega)$ and all $\bold{x} \in \mathbb{R}^3 \setminus \partial \Omega$ we define the functions
		\begin{align*}
		\mathcal{S}(\nu)(\bold{x})&:=\int_{\partial \Omega} \frac{\nu(\bold{y})}{4\pi\vert \bold{x}- \bold{y}\vert}\, d \bold{y},\\
		\mathcal{D}(\lambda)(\bold{x})&:=\int_{\partial \Omega} \lambda(\bold{y}) \eta(\bold{y}) \cdot\nabla_{\bold{y}}\frac{1}{4\pi\vert \bold{x}- \bold{y}\vert}\, d \bold{y}.
		\end{align*}
		
		The mappings $\mathcal{S}$ and $\mathcal{D}$ are known as the single layer and double layer potentials respectively. It can be shown (see, e.g., \cite[Chapter 2]{Schwab}) that $\mathcal{S}$ is a bounded linear operator from $H^{-\frac{1}{2}}(\partial \Omega)$ to $H^{1}_{\text{loc}}\left(\mathbb{R}^3\right)$ and $\mathcal{D}$ is a bounded linear operator from $H^{\frac{1}{2}}(\partial \Omega)$ to $H^{1}_{\text{loc}}\left(\mathbb{R}^3 \setminus \partial \Omega\right)$, and both $\mathcal{S}$ and $\mathcal{D}$ map into the space of harmonic functions on the complement $\mathbb{R}^3 \setminus \partial \Omega$ of the boundary.}\\
	
	{As the final step, we define the following linear bounded boundary integral operators:
		\begin{align*}
		\mathcal{V}&:= \hphantom{-}\big(\gamma^- \circ \mathcal{S}\big) \hspace{0.0cm}\colon \hspace{0.0cm}H^{-\frac{1}{2}}(\partial \Omega) \rightarrow H^{\frac{1}{2}}(\partial \Omega), \hspace{1.0cm} \mathcal{K}^{\hphantom{*}}:= \Big(\gamma^- \circ\mathcal{D} + \frac{1}{2}I\Big)\colon H^{\frac{1}{2}}(\partial \Omega) \rightarrow H^{\frac{1}{2}}(\partial \Omega),\\
		\mathcal{W}&:= -\big(\gamma_N^- \circ\mathcal{D}\big) \hspace{0.0cm}\colon \hspace{0.0cm}H^{\frac{1}{2}}(\partial \Omega) \rightarrow H^{-\frac{1}{2}}(\partial \Omega),
		\hspace{1cm}
		\mathcal{K}^*:= \Big(\gamma_N^- \circ\mathcal{S} - \frac{1}{2}I\Big)\colon H^{-\frac{1}{2}}(\partial \Omega) \rightarrow H^{-\frac{1}{2}}(\partial \Omega).
		\end{align*}
		Here $I$ denotes the identity operator on the relevant trace space. The mapping $\mathcal{V}$ is known as the single layer boundary operator, the mapping $\mathcal{K}$ is known as the double layer boundary operator, the mapping $\mathcal{K}^*$ is known as the adjoint double layer boundary operator and the mapping $\mathcal{W}$ is known as the hypersingular boundary operator. Detailed definitions and a discussion of the properties of these boundary integral operators can be found in \cite[Chapter 3]{Schwab}. We state three properties in particular that will be used in the sequel.}\\
	
	{{\textbf{Property 1:} \cite[Theorem 3.5.3]{Schwab} The single layer boundary operator $\mathcal{V} \colon H^{-\frac{1}{2}}(\partial \Omega) \rightarrow H^{\frac{1}{2}}(\partial \Omega)$ is Hermitian and coercive, i.e., there exists a constant $c_{\mathcal{V}} > 0$ such that for all functions $\sigma \in H^{-\frac{1}{2}}(\partial \Omega)$ it holds that}
		\begin{align*}
		\langle \sigma, \mathcal{V}\sigma \rangle_{\partial \Omega} \geq c_{\mathcal{V}} \Vert \sigma\Vert^2_{H^{-\frac{1}{2}}(\partial \Omega)}.
		\end{align*}
		
		This implies in particular that the inverse $\mathcal{V}^{-1} \colon H^{\frac{1}{2}}(\partial \Omega) \rightarrow H^{-\frac{1}{2}}(\partial \Omega)$ is also a Hermitian, coercive and bounded linear operator. Consequently, $\mathcal{V}$ induces a norm $\Vert \cdot \Vert_{\mathcal{V}}$ and associated inner product on $H^{-\frac{1}{2}}(\partial \Omega)$ and the inverse $\mathcal{V}^{-1}$ induces a norm $\Vert \cdot \Vert_{\mathcal{V}^{-1}}$ and associated inner product on $H^{\frac{1}{2}}(\partial \Omega)$.} {We emphasise here that while the coercivity constant $c_{\mathcal{V}}$ of the single layer boundary operator a priori depends on the geometry $\Omega^-$, the independence of $c_{\mathcal{V}}$ with respect to the number of open balls $N$ in the system is a key point in the present analysis and will be the subject of further discussion in Section \ref{sec:4.1} (see, in particular, Lemmas \ref{lem:Single} and \ref{lem:Hassan2})}.\\
	
	{	{\textbf{Property 2:}\cite[Theorem 3.5.3]{Schwab}} The hypersingular boundary operator $\mathcal{W} \colon H^{\frac{1}{2}}(\partial \Omega) \rightarrow H^{\frac{1}{2}}(\partial \Omega)$ is Hermitian, non-negative and coercive on a subspace of $H^{\frac{1}{2}}(\partial \Omega)$, i.e., there exists a constant $c_{\mathcal{W}} > 0$ such that for all functions $\lambda \in H^{\frac{1}{2}}(\partial \Omega)$ with $\sum_{i=1}^N\left\vert\int_{\partial \Omega_i} \lambda(\bold{x})\, d \bold{x}\right\vert=0$, it holds that
		\begin{align*}
		\langle \mathcal{W}\lambda, \lambda \rangle_{\partial \Omega} \geq c_{\mathcal{W}} \Vert \lambda\Vert^2_{H^{\frac{1}{2}}(\partial \Omega)}.
		\end{align*}
		
		{\textbf{Property 3:}\cite[Theorem 3.8.7]{Schwab}} The coercivity constants of the single layer and hypersingular boundary operators satisfy $c_{\mathcal{V}} c_{\mathcal{W}} \leq \frac{1}{4}$. Therefore the constant
		\begin{align*}
		c_{\mathcal{K}}:= \frac{1}{2} + \sqrt{\frac{1}{4}- c_{\mathcal{V}}c_{\mathcal{W}} },
		\end{align*}
		is well-defined and $c_{\mathcal{K}} \in \big[\frac{1}{2}, 1\big)$.}\\

	We are now ready to state the problem we wish to analyse.\\
	
	\subsection{Abstract Dielectric Electrostatic Interaction Problem}~
	Let $K$ denote the Coulomb constant and let $\sigma_f \in H^{-\frac{1}{2}}(\partial \Omega)$ be arbitrary. For each $s\in \{+, -\}$ find a function $\Phi^s \in \mathbb{H}(\Omega^s)$ with the property that
	\begin{equation}\label{eq:3.2}
	\begin{split}
	\gamma^+ \Phi^+-\gamma^- \Phi^-&= 0 \hspace{1.61cm}\qquad ~~~~~~~\text{in } H^{\frac{1}{2}}(\partial \Omega),\\
	\kappa \gamma_N^- \Phi^-- \kappa_0 \gamma_N^+ \Phi^+&= 4 \pi K \sigma_f \qquad \hspace{5mm}~\text{in } H^{-\frac{1}{2}}(\partial \Omega),\\
	|\Phi^+(\bold{x})| &\leq C \vert \bold{x}\vert^{-1} \hspace{1.5mm}\qquad \text{for } |\bold{x}| \to \infty.
	\end{split}
	\end{equation}
	
	\begin{remark}
		We may assume without loss of generality that $K=1$. This is, for instance, true if one picks the CGS system of units.
	\end{remark}
	
	{\begin{remark}\label{rem:sigma_f}
			In the physics literature, $\sigma_f \in H^{-\frac{1}{2}}(\partial \Omega)$ is called the free charge and is a known quantity. The unknown function $\Phi^s \in \mathbb{H}(\Omega^s)$ is the electric potential that results after the polarisation of the free charge residing on the surface of the dielectric spheres $\partial \Omega_i, ~ i=1, \ldots, N$.
	\end{remark}}
	
	\begin{remark}
		The operator equation \eqref{eq:3.2} is very similar to the abstract transmission problem for second order elliptic PDEs. A detailed overview of the transmission problem can, for example, be found in \cite[Chapter 2.9]{Schwab}.
	\end{remark}
	
	From a practical perspective, the main difficulty in solving the transmission problem \eqref{eq:3.2} is the fact that this problem is posed on the unbounded domain $\mathbb{R}^3$. The usual approach in the literature to circumventing this difficulty is to appeal to the theory of integral equations and reformulate an operator equation posed on some domain $ \Omega^- \cup \Omega^+$, such as Equation \eqref{eq:3.2}, as a so-called \emph{boundary integral equation} (BIE) posed on the interface~$\partial \Omega$ (see, for example, \cite{McLean} or \cite{Schwab}). \\

	\noindent{\textbf{Integral Equation Formulation for the Induced Charges}}~
	
	Let $\sigma_f \in H^{-\frac{1}{2}}(\partial \Omega)$. Find $\nu \in H^{-\frac{1}{2}}(\partial \Omega)$ with the property that
	\begin{align}\label{eq:3.3a}
	\nu - \frac{\kappa_0-\kappa}{\kappa_0} (\text{DtN}\mathcal{V})\nu= \frac{4\pi}{\kappa_0}\sigma_f.
	\end{align}

	\begin{remark}
		From a physical point of view, the unknown $\nu \in H^{-\frac{1}{2}}(\partial \Omega)$ in the integral equation \eqref{eq:3.3a} is the \emph{induced} surface charge on each sphere $\partial \Omega_i, ~ i=1, \ldots, N$. 
	\end{remark}
	
	\begin{remark}\label{rem:trivial}
		Consider the setting of the integral equation \eqref{eq:3.3a}. Suppose there is some open ball $\Omega_j,~j~\in~\{1, \ldots, N\}$ such that $\kappa=\kappa_0$ on $\partial \Omega_j$. Then it follows that the induced surface charge $\nu_j$ on sphere $\partial \Omega_j$ is simply given by $\nu_j = \frac{4\pi}{\kappa_0} \sigma_{f, j}$ where $\sigma_{f, j}$ denotes the free charge on $\partial \Omega_j$. Consequently, throughout the remainder of this article, we will assume that $\kappa \neq \kappa_0$ on all the spheres. Note that physically, the situation $\kappa=\kappa_0$ on $\partial \Omega_j$ corresponds to no polarisation on the sphere $\partial \Omega_j$.
	\end{remark}
	
	The boundary integral equation \eqref{eq:3.3a} can be derived from the transmission problem \eqref{eq:3.2} using a single layer ansatz. Indeed, we have the following lemma:
	
	\begin{lemma}\label{lem:equivalence}
		Let $\bold{\Phi}:= (\Phi^-, \Phi^+) \in \mathbb{H}(\Omega^-) \times \mathbb{H}(\Omega^+)$ be a solution to the transmission problem \eqref{eq:3.2}. Then $\nu:= \mathcal{V}^{-1}\big(\gamma^- \Phi^-\big)$ is a solution to the BIE \eqref{eq:3.3a}. Conversely, let $\nu \in H^{-\frac{1}{2}}(\partial \Omega)$ be a solution to the BIE \eqref{eq:3.3a}. Then $ (\Phi^-, \Phi^+):=\big(\mathcal{S}\nu \vert_{\Omega^-}, \mathcal{S}\nu \vert_{\Omega^+}\big)$ is a solution to the transmission problem \eqref{eq:3.2}.
	\end{lemma}
	
	The proof of Lemma \ref{lem:equivalence} can be found in Appendix \ref{sec:Appendix_C}.
	
	\begin{remark}
		We have claimed in the introduction that the BIE \eqref{eq:3.3a} is essentially an integral equation of the second kind. This assertion is discussed in more detail in Section \ref{sec:3}.
	\end{remark}
	
	A key quantity of interest in physical applications is the total electrostatic energy associated with the free charge $\sigma_f \in H^{-\frac{1}{2}}(\partial \Omega)$ and the resulting induced surface charge $\nu \in H^{-\frac{1}{2}}(\partial \Omega)$. 
	
	\begin{definition}[Energy Functional and Total Electrostatic Energy]\label{def:Energy}
		Let $\sigma_f \in H^{-\frac{1}{2}}(\partial \Omega)$. Then we define the electrostatic energy functional $\mathcal{E}_{\sigma_f} \colon H^{-\frac{1}{2}}(\partial \Omega) \rightarrow \mathbb{R}$ as the bounded linear mapping with the property that for all $\sigma \in H^{-\frac{1}{2}}(\partial \Omega)$ it holds that
		\begin{equation}
		\mathcal{E}_{\sigma_f}(\sigma):=  \frac{1}{2}\langle  \sigma_f, \mathcal{V}\sigma\rangle_{\partial \Omega}=\frac{1}{2}\langle \sigma, \mathcal{V}\sigma_f\rangle_{\partial \Omega},
		\end{equation}
		and we define the total electrostatic energy of the system as $\mathcal{E}_{\sigma_f}(\nu)$ where $\nu \in H^{-\frac{1}{2}}(\partial \Omega)$ is the solution to the integral equation \eqref{eq:3.3a}. 	
	\end{definition}

	For clarity of exposition, we now define the relevant boundary integral operator.
	\begin{definition}\label{def:A}
		We define the linear operator $\mathcal{A} \colon {H}^{\frac{1}{2}}(\partial \Omega) \rightarrow {H}^{\frac{1}{2}}(\partial \Omega)$ as the mapping with the property that for all $\lambda \in {H}^{\frac{1}{2}}(\partial \Omega)$ it holds that
		\begin{align*}
		\mathcal{A} \lambda:= \lambda - \mathcal{V} \text{DtN}\Big(\frac{\kappa_0-\kappa}{\kappa_0} \lambda\Big).
		\end{align*}
		
		In addition, we denote by $\mathcal{A}^* \colon {H}^{-\frac{1}{2}}(\partial \Omega) \rightarrow {H}^{-\frac{1}{2}}(\partial \Omega)$ the adjoint operator of $\mathcal{A}$.
	\end{definition}

	The BIE \eqref{eq:3.3a} now has a straightforward weak formulation. \\
	
	\noindent{\textbf{Weak Formulation of the Integral Equation \eqref{eq:3.3a}}}~
	
	Let $\sigma_f \in H^{-\frac{1}{2}}(\partial \Omega)$ and let $\mathcal{A}^* \colon H^{-\frac{1}{2}}(\partial \Omega) \rightarrow H^{-\frac{1}{2}}(\partial \Omega)$ denote the adjoint of the operator $\mathcal{A}$ given by Definition \ref{def:A}. Find $\nu \in H^{-\frac{1}{2}}(\partial \Omega)$ such that for all $\lambda \in H^{\frac{1}{2}}(\partial \Omega)$ it holds that
	\begin{align}\label{eq:weak1a}
	\left\langle\mathcal{A}^* \nu, \lambda\right\rangle_{\partial \Omega}= \frac{4\pi}{\kappa_0}\left \langle\sigma_f, \lambda \right \rangle_{\partial \Omega}.
	\end{align}
	
	Next, we state the Galerkin discretisation of the boundary integral equation \eqref{eq:3.3a}.
	
	\subsection{Galerkin Discretisation}
	We first define the relevant approximation spaces. In the sequel, we will denote by $\mathbb{N}_0$ the set of non-negative integers.
	
	\begin{definition}[Spherical Harmonics]
		{ Let $\ell \in \mathbb{N}_0$ and $m \in \{-\ell, \ldots, \ell\}$ be integers. Then we define the function $\mathcal{Y}_\ell^m \colon \mathbb{S}^2 \rightarrow \mathbb{R}$ as 
			\begin{align*}
			\mathcal{Y}_\ell^m(\theta, \phi):=\begin{cases}
			(-1)^m \sqrt{2} \sqrt{\frac{2\ell+1}{4\pi}\frac{(\ell-|m|)!}{(\ell+|m|)!}}P_{\ell}^{|m|}\big(\cos(\theta)\big)\sin\big(|m|\phi\big), \quad& \text{if } m < 0,\\
			\sqrt{\frac{2\ell+1}{4\pi}}P_{\ell}^m\big(\cos(\theta)\big), \quad& \text{if } m =0,\\
			(-1)^m \sqrt{2} \sqrt{\frac{2\ell+1}{4\pi}\frac{(\ell-m)!}{(\ell+m)!}}P_{\ell}^m\big(\cos(\theta)\big)\cos\big(m\phi\big), \quad& \text{if } m > 0,
			\end{cases}
			\end{align*}
			where $P_{\ell}^m$ denotes the associated Legendre polynomial of degree $\ell$ and order $m$. The function $\mathcal{Y}_{\ell}^m$ is known as the real-valued $L^2$-orthonormal spherical harmonic of degree $\ell$ and order $m$.}
	\end{definition}
	
	\begin{definition}[Approximation Space on a Sphere]\label{def:6.6}
		Let $\mathcal{O}_{\bold{x}_0} \subset \mathbb{R}^3$ be an open ball of radius $r > 0$ centred at the point $\bold{x}_0 \in \mathbb{R}^3$ and let $\ell_{\max} \in \mathbb{N}$. We define the finite-dimensional Hilbert space $W^{\ell_{\max}}(\partial \mathcal{O}_{\bold{x}_0}) \subset {H}^{\frac{1}{2}}(\partial\mathcal{O}_{\bold{x}_0}) \subset  {H}^{-\frac{1}{2}}(\partial\mathcal{O}_{\bold{x}_0})$ as the vector space
		\begin{align*}
		W^{\ell_{\max}}(\partial \mathcal{O}_{\bold{x}_0}):= \Big\{u \colon \partial \mathcal{O}_{\bold{x}_0} \rightarrow \mathbb{R} &\text{ such that } u(\bold{x})= \sum_{{\ell}=0}^{\ell_{\max}} \sum_{m=-\ell}^{m=+\ell} [u]_\ell^m \mathcal{Y}_{\ell}^m\left(\frac{\bold{x}-\bold{x}_0}{\vert \bold{x}-\bold{x}_0\vert}\right) \\
		&\text{where all }[u]_{\ell}^m \in \mathbb{R}\Big\},
		\end{align*}
		equipped with the inner product
		\begin{align}\label{eq:SH1}
		(u, v)_{W^{\ell_{\max}}(\partial \mathcal{O}_{\bold{x}_0})}:=r^2 [u]_0^0 [v]_0^0 + r^2 \sum_{\ell=1}^{\ell_{\max}} \sum_{m=-\ell}^{m=+\ell} \frac{\ell}{r} [u]_{\ell}^m [v]_{\ell}^m \qquad \forall u, v \in W^{\ell_{\max}}(\partial \mathcal{O}_{\bold{x}_0}).
		\end{align}	
	\end{definition}
	
	It is now straightforward to extend the Hilbert space defined in Definition \ref{def:6.6} to the domain $\partial \Omega$.
	
	\begin{definition}[Global Approximation Space]\label{def:6.7}
		We define the finite-dimensional Hilbert space $W^{\ell_{\max}} \subset H^{\frac{1}{2}}(\partial \Omega) \subset H^{-\frac{1}{2}}(\partial \Omega)$ as the vector space
		\begin{align*}
		W^{\ell_{\max}}:= \Big\{u \colon \partial\Omega \rightarrow \mathbb{R} \text{ such that } \forall i \in \{1, \ldots, N\} \colon u\vert_{\partial \Omega_i} \in W^{\ell_{\max}}(\partial \Omega_i)\Big\},
		\end{align*}
		equipped with the inner product
		\begin{align}\label{eq:SH00}
		(u, v)_{W^{\ell_{\max}}}:= \sum_{i=1}^N \left(u, v\right)_{W^{\ell_{\max}}(\partial \Omega_i)} \qquad \forall u, v \in W^{\ell_{\max}}.
		\end{align}
	\end{definition}

	\noindent{\textbf{Galerkin Discretisation of the Integral Equation \eqref{eq:3.3a}}}~
	
	Let $\sigma_f \in {H}^{-\frac{1}{2}}(\partial \Omega)$ and let $\ell_{\max} \in \mathbb{N}$. Find $\nu_{\ell_{\max}} \in W^{\ell_{\max}}$ such that for all $\psi_{\ell_{\max}} \in W^{\ell_{\max}}$ it holds that
	\begin{align}\label{eq:Galerkina}
	(\mathcal{A}^*\nu_{\ell_{\max}}, \psi_{\ell_{\max}})_{L^2(\partial \Omega)}= \frac{4\pi}{\kappa_0}(\sigma_f, \psi_{\ell_{\max}})_{L^2(\partial \Omega)}.
	\end{align}

	\subsection{Main Results}
	We begin this section by fixing some additional notation and introducing a new norm and inner product on the space $H^{\frac{1}{2}}(\partial \Omega)$ that will aid our subsequent analysis.\\
	
	\noindent \textbf{Notation}: We define $\mathcal{C}(\partial \Omega)$ as the set of functions given by
	\begin{align*}
	\mathcal{C}(\partial \Omega):= \left\{u \colon \partial \Omega \rightarrow \mathbb{R} \colon \forall i =1, \ldots, N \text{ the restriction }u|_{\partial \Omega_i} \text{ is a constant function}\right\},
	\end{align*}
	and we observe that $\mathcal{C}(\partial \Omega)$ is a closed subspace of dimension $N$ of $H^{\frac{1}{2}}(\partial \Omega)$ under the $L^2(\partial \Omega)$ norm (since the Slobodeckij semi-norm of constant functions is zero).\\
	
	\noindent\textbf{Notation:} {We define the function spaces $\breve{H}^{\frac{1}{2}}(\partial \Omega)$ and $\breve{H}^{-\frac{1}{2}}(\partial \Omega)$ as
		\begin{align*}
		\breve{H}^{\frac{1}{2}}(\partial \Omega)&:=\left\{u \in H^{\frac{1}{2}}(\partial \Omega) \colon (u, v)_{L^2(\partial \Omega)}=0 \hspace{1mm}\forall v \in \mathcal{C}(\partial \Omega)\right\},\\
		\breve{H}^{-\frac{1}{2}}(\partial \Omega)&:=\left\{\phi \in H^{-\frac{1}{2}}(\partial \Omega) \colon \langle \phi, v\rangle_{\partial \Omega}=0 ~\forall v \in \mathcal{C}(\partial \Omega)\right\},
		\end{align*}
		and we observe that both sets are Banach spaces under the Sobolev-Slobodeckij norms introduced earlier. Intuitively, the spaces $\breve{H}^{\frac{1}{2}}(\partial \Omega)$ and $\breve{H}^{-\frac{1}{2}}(\partial \Omega)$ are trace spaces that do not contain any piecewise constant functions. }The following simple lemma follows from these definitions.
	
	\begin{lemma}\label{lem:decomp}
		There exist complementary decompositions (in the sense of Brezis \cite[Section 2.4]{Brezis}) of the spaces $H^{\frac{1}{2}}(\partial \Omega)$ and $H^{-\frac{1}{2}}(\partial \Omega)$ given by
		\begin{align}\label{eq:Appendix1}
		H^{\frac{1}{2}}(\partial \Omega)&=\breve{H}^{\frac{1}{2}}(\partial \Omega) \oplus \mathcal{C}(\partial \Omega),\\ \nonumber
		H^{-\frac{1}{2}}(\partial \Omega)&=\breve{H}^{-\frac{1}{2}}(\partial \Omega) \oplus \mathcal{C}(\partial \Omega).
		\end{align}
		
		{ Moreover, the projection operators $\mathbb{P}^{\perp}_0\colon H^{\frac{1}{2}}(\partial \Omega) \rightarrow \breve{H}^{\frac{1}{2}}(\partial \Omega)$ and $\mathbb{P}_0\colon H^{\frac{1}{2}}(\partial \Omega) \rightarrow \mathcal{C}(\partial \Omega)$, $\mathbb{Q}^{\perp}_0\colon H^{-\frac{1}{2}}(\partial \Omega) \rightarrow \breve{H}^{-\frac{1}{2}}(\partial \Omega)$, and $\mathbb{Q}_0\colon H^{-\frac{1}{2}}(\partial \Omega) \rightarrow \mathcal{C}(\partial \Omega)$ associated with these complementary decompositions are all bounded.}
	\end{lemma}

	
	The complementary decomposition introduced through Lemma \ref{lem:decomp} is at the heart of our well-posedness analysis as will become clear in Section \ref{sec:Proofs}. 
	
	{\begin{remark} \label{rem:Review_3}
			Consider the complementary decomposition introduced through Lemma \ref{lem:decomp}. It is a simple exercise to show that for all $\lambda \in H^{\frac{1}{2}}(\partial \Omega)$ and all $\sigma \in H^{-\frac{1}{2}}(\partial \Omega)$ the following relations hold:
			\begin{align*}
			\left \langle \mathbb{Q}_0\sigma,\, \mathbb{P}_0^{\perp}\lambda\right \rangle_{\partial \Omega}=0 \quad \text{ and } \quad \left \langle \mathbb{Q}^{\perp}_0\sigma,\, \mathbb{P}_0\lambda\right \rangle_{\partial \Omega}=0.
			\end{align*}
	\end{remark}}
	
	In order to take full advantage of this decomposition of $H^{\frac{1}{2}}(\partial \Omega)$, it is necessary to introduce a new norm on $H^{\frac{1}{2}}(\partial \Omega)$. 
	
	\begin{definition}\label{def:NewNorm}
		We define on $H^{\frac{1}{2}}(\partial \Omega)$ a new norm $||| \cdot ||| \colon H^{\frac{1}{2}}(\partial \Omega) \rightarrow \mathbb{R}$ given by
		\begin{align*}
		\forall \lambda \in H^{\frac{1}{2}}(\partial \Omega)\colon ~ |||\lambda|||^2:= \left\Vert \mathbb{P}_0 \lambda\right\Vert^2_{L^2(\partial \Omega)}+ \left\langle \text{DtN}\lambda, \lambda\right \rangle_{\partial \Omega}.
		\end{align*}
		
	\end{definition}
	
	\begin{remark}\label{rem:New}
		We claim that the norm $||| \cdot |||$ is equivalent to the $\Vert \cdot \Vert_{H^{\frac{1}{2}}(\partial \Omega)}$ norm introduced in Section \ref{sec:2} (see Appendix \ref{sec:Appendix_A} for a proof). {Consequently, there exists a constant $c_{\rm equiv} >1$ such that for all $\lambda \in H^{\frac{1}{2}}(\partial \Omega)$ it holds that
			$\frac{1}{c_{\rm equiv}} ||| \lambda ||| \leq \Vert  \lambda \Vert_{H^{\frac{1}{2}}(\partial \Omega)} \leq c_{\rm equiv} ||| \lambda |||$.
			It is important to note that the equivalence constant $c_{\rm equiv}$ is independent of $N$.}
	\end{remark}	
	
	Henceforth, we adopt the convention that the Hilbert space ${H}^{\frac{1}{2}}(\partial \Omega)$ is equipped with the $||| \cdot |||$ norm defined through Definition \ref{def:NewNorm}. The main advantage of using the new $||| \cdot |||$ norm is that it preserves the structure of the complementary decomposition of ${H}^{\frac{1}{2}}(\partial \Omega)$. Indeed, for any function $\lambda \in {H}^{\frac{1}{2}}(\partial \Omega)$, we have 
	\begin{align*}
	|||\lambda|||^2 &= \left\Vert \mathbb{P}_0 \lambda\right\Vert^2_{L^2(\partial \Omega)}+ \left\langle \text{DtN}\lambda, \lambda\right \rangle_{\partial \Omega}= ||| \mathbb{P}_0 \lambda |||^2+||| \mathbb{P}_0^{\perp} \lambda |||^2.
	\end{align*}
	
	\begin{remark}
		We remark that under this convention, due to the equivalence of norms, the definitions of the dual space ${H}^{-\frac{1}{2}}(\partial \Omega)$ and the associated duality pairing $\langle \cdot, \cdot \rangle_{\partial \Omega}$ remain unchanged. Thus, we can define a new dual norm $||| \cdot |||^* \colon H^{-\frac{1}{2}}(\partial \Omega) \rightarrow \mathbb{R}$ as the mapping with the property that for all $\sigma \in H^{-\frac{1}{2}}(\partial \Omega)$ it holds that
		\begin{align*}
		||| \sigma |||^*:= \sup_{0\neq \psi \in H^{\frac{1}{2}}(\partial \Omega) } \frac{\left \langle \sigma, \psi \right \rangle_{\partial \Omega } }{||| \psi |||},
		\end{align*}
		and we observe that the new $||| \cdot |||^*$ dual norm on ${H}^{-\frac{1}{2}}(\partial \Omega)$ is equivalent to the canonical dual norm $\Vert \cdot \Vert_{H^{-\frac{1}{2}}(\partial \Omega)}$ {with equivalence constant that is once again independent of $N$}. 
	\end{remark}
	
	\begin{remark}\label{rem:isometry}
		It is a simple exercise to prove that the Dirichlet-to-Neumann map $\text{DtN} \colon \breve{H}^{\frac{1}{2}}(\partial \Omega) \rightarrow \breve{H}^{-\frac{1}{2}}(\partial \Omega)$ is invertible and satisfies for all $\tilde{\lambda} \in \breve{H}^{\frac{1}{2}}(\partial \Omega)$
		\begin{align*}
		||| \text{DtN} \tilde{\lambda} |||^* = ||| \tilde{\lambda} |||.
		\end{align*}
		This fact will be used often in the sequel. 
	\end{remark}

	Next, we define the higher regularity spaces and norms that are used in the error estimates.
	\begin{definition}\label{def:7.1}
		Let $s\geq 0$ be a real number and let $\mathcal{O}_{\bold{x}_0} \subset \mathbb{R}^3$ be an open ball of radius $r > 0$ centred at the point $\bold{x}_0 \in \mathbb{R}^3$. Then we define constructively the fractional Sobolev space ${H}^{s}(\partial\mathcal{O}_{\bold{x}})$ as the set
		\begin{align*}
		{H}^{s}(\partial\mathcal{O}_{\bold{x}_0}):= \Big\{u \colon \partial \mathcal{O}_{\bold{x}_0} \rightarrow \mathbb{R} \text{ such that } u(\bold{x})&= \sum_{{\ell}=0}^\infty \sum_{m=-\ell}^{m=+\ell} [u]_{\ell}^m\mathcal{Y}_{\ell}^m\left(\frac{\bold{x}-\bold{x}_0}{\vert \bold{x}-\bold{x}_0\vert}\right) \\
		\text{where all }[u]_{\ell}^m \in \mathbb{R} \text{ satisfy }& \sum_{\ell=1}^\infty\sum_{m=-\ell}^{m=+\ell}\left(\frac{l}{r}\right)^{2s}([u]^m_{\ell})^2< \infty\Big\},
		\end{align*}
		equipped with the inner product
		\begin{align}\label{eq:frac1}
		(u, v)_{{H}^{s}(\partial\mathcal{O}_{\bold{x}_0})}:= r^2[u]_0^0 \, [v]_0^0+r^2\sum_{\ell=1}^\infty\sum_{m=-\ell}^{m=+\ell}\left(\frac{\ell}{r}\right)^{2s}[u]_\ell^m [v]_\ell^m \qquad \forall u, v \in H^s(\partial \mathcal{O}_{\bold{x}_0}).
		\end{align}
		{Additionally, we write $||| \cdot |||_{{H}^{s}(\partial\mathcal{O}_{\bold{x}_0})}$ to denote the norm induced by the inner-product $(\cdot, \cdot)_{{H}^{s}(\partial\mathcal{O}_{\bold{x}_0})}$.}
		
	\end{definition}

	\begin{remark}\label{rem:7.1}
		Definition \ref{def:7.1} is an intrinsic definition of the fractional Sobolev space ${H}^{s}(\partial \mathcal{O}_{\bold{x}_0})$, which coincides with the definition of these fractional Sobolev spaces involving the Sobolev-Slobodeckij inner product (see, e.g., \cite{Hitch}). The equivalence follows from the fact that the spherical harmonics are eigenvectors of the self-adjoint Laplace-Beltrami operator $\Delta_{\partial \mathcal{O}_{\bold{x}_0}}$ as discussed in, for example, \cite[Chapter 1 Section 7]{LionsMagenes}. 
	\end{remark}

	Definition \ref{def:7.1} suggests a natural intrinsic definition of the fractional Sobolev spaces on $\partial \Omega$.
	
	\begin{definition}\label{def:7.2}
		Let $s\geq 0$ be a real number. Then we define the Hilbert space ${H}^{s}(\partial \Omega)$ as the set
		\begin{align*}
		{H}^{s}(\partial \Omega):= \Big\{u \colon \partial\Omega \rightarrow \mathbb{R} \text{ such that } \forall i \in {1, \ldots, N} \colon u\vert_{\partial \Omega_i} \in {H}^{s}(\partial \Omega_i)\Big\},
		\end{align*}	
		equipped with the inner product
		\begin{align}\label{eq:frac20}
		(u, v)_{{H}^{s}(\partial \Omega)}:= \sum_{i=1}^N \left(u, v\right)_{{H}^{s}(\partial \Omega_i)} \qquad \forall u, v \in H^s(\partial \Omega).
		\end{align}
		{Additionally, we write $||| \cdot |||_{{H}^{s}(\partial \Omega)}$ to denote the norm induced by the inner-product $(\cdot, \cdot)_{{H}^{s}(\partial \Omega)}$.}
	\end{definition}
	
	\begin{remark}
		A direct calculation shows that the norm $||| \cdot |||_{{H}^{\frac{1}{2}}(\partial \Omega)}$ coincides with the $||| \cdot |||$ norm defined through Definition \ref{def:NewNorm}. Moreover, the $||| \cdot |||_{{H}^{\frac{1}{2}}(\partial \Omega)}$ norm  coincides with the $\Vert\cdot \Vert_{W^{\ell_{\max}}}$ norm on the space~$W^{\ell_{\max}}$.
	\end{remark}
	
	We are now ready to state our main results.
	
	\begin{theorem}[Error Estimates]\label{thm:1}~
		
		\noindent Let $s \geq 0$ be a real number, let $\ell_{\max} \in \mathbb{N}$, let $\sigma_f \in {H}^{s}(\partial \Omega)$, let $\mathcal{E}_{\sigma_f} \colon H^{-\frac{1}{2}}(\partial \Omega) \rightarrow \mathbb{R}$ be the electrostatic energy functional defined through Definition \ref{def:Energy}, let $\nu \in {H}^{-\frac{1}{2}}(\partial \Omega)$ be the unique solution to the weak formulation \eqref{eq:weak1a} with right hand side given by $\sigma_f$, let $\nu_{\ell_{\max}} \in W^{\ell_{\max}}$ be the unique solution to the Galerkin discretisation defined through Equation \eqref{eq:Galerkina}, and let $\mathbb{Q}_0^{\perp} \colon H^{-\frac{1}{2}}(\partial \Omega) \rightarrow \breve{H}^{-\frac{1}{2}}(\partial \Omega)$ denote the projection operator defined through Lemma \ref{lem:decomp}. Then there exists a constant $C_{\text{main}}>0$ that depends on the radii of the open balls, the dielectric constants and the minimal inter-sphere separation distance but is independent of both $s$ and the number of open balls $N$ such that {
			\begin{align*}
			|||\nu-\nu_{\ell_{\max}}|||^* &\leq C_{\rm main}\left(\frac{\max r_j}{\ell_{\max}+1}\right)^{s+\frac{1}{2}} \left(||| \mathbb{Q}_{0}^{\perp}\nu|||_{{H}^{s}(\partial \Omega)} + \frac{8\pi}{\kappa_0} ||| \mathbb{Q}_{0}^{\perp}\sigma_f|||_{H^s(\partial \Omega)}\right)\\
			\big|\mathcal{E}_{\sigma_f}(\nu)-\mathcal{E}_{\sigma_f}(\nu_{\ell_{\max}})\big| &\leq C_{\rm main}\left(\frac{\max r_i}{\ell_{\max}+1}\right)^{s+\frac{1}{2}}|||\mathcal{V}\sigma_f||| \left( ||| \mathbb{Q}_{0}^{\perp}\nu|||_{{H}^{s}(\partial \Omega)} + \frac{8\pi}{\kappa_0} ||| \mathbb{Q}_{0}^{\perp}\sigma_f|||_{H^s(\partial \Omega)}\right).
			\end{align*}}
	\end{theorem}

	Theorem \ref{thm:1} is a standard a priori error estimate for the approximate induced surface charge and approximate electrostatic energy obtained by solving the Galerkin discretisation \eqref{eq:Galerkina}. We emphasise that the most important aspect of this error estimate is that the convergence rate pre-factor $C_{\rm main}$ is \emph{explicitly independent of the number of objects $N$}. { Consequently, for any geometry in the family of geometries~$\{\Omega_{\mathcal{F}}\}_{\mathcal{F} \in \mathcal{I}}$ satisfying assumptions {\textbf{ A1)-A3)}}, the following holds: Given a fixed number of degrees of freedom $\ell_{\max}$ per sphere, the relative error in the induced surface charge and in the total electrostatic energy normalised by the free-charge electrostatic energy \emph{does not increase} as $N$ increases. This implies in particular that for any configuration in the family of geometries~$\{\Omega_{\mathcal{F}}\}_{\mathcal{F} \in \mathcal{I}}$, in order to guarantee the same relative accuracy in the induced surface charge, one does not need to increase the number of degrees of freedom per sphere as $N_{\mathcal{F}}$ increases.}
	
	\begin{theorem}[Exponential Convergence]\label{thm:2}~
		\noindent Let $\ell_{\max} \in \mathbb{N}$, let $C_{\text{main}}$ denote the convergence rate pre-factor from Theorem \ref{thm:1}, {let $\sigma_f \in C^\infty(\partial \Omega)$ be analytic on $\partial \Omega$}, let $\mathcal{E}_{\sigma_f} \colon H^{-\frac{1}{2}}(\partial \Omega) \rightarrow \mathbb{R}$ be the electrostatic energy functional defined through Definition \ref{def:Energy}, let $\nu \in {H}^{-\frac{1}{2}}(\partial \Omega)$ be the unique solution to the weak formulation \eqref{eq:weak1a} with right hand side given by $\sigma_f$, and let $\nu_{\ell_{\max}} \in W^{\ell_{\max}}$ be the unique solution to the Galerkin discretisation defined through Equation \eqref{eq:Galerkina}. For $\ell_{\max}$ sufficiently large, if $\nu$ is analytic on $\partial \Omega$ then there exists a constant $C_{\nu, \sigma_f} > 0$ depending on the exact solution $\nu$ and the free charge $\sigma_f$ such that
			\begin{align*}
			\frac{1}{\sqrt{N}}|||\nu-\nu_{\ell_{\max}}|||^* &\leq\sqrt{8 \pi \max r^2_j}(2\max r_j)^{\frac{1}{4}}C_{\nu, \sigma_f}C_{\rm main}\exp\left(-\frac{1}{4C_{\nu, \sigma_f}}\frac{\ell_{\max}+1}{\max r_j} +\frac{1}{2}\right),\\
			\frac{1}{\sqrt{N}}\big|\mathcal{E}_{\sigma_f}(\nu)-\mathcal{E}_{\sigma_f}(\nu_{\ell_{\max}})\big| &\leq\sqrt{8 \pi \max r^2_j}  (2\max r_j)^{\frac{1}{4}}C_{\nu, \sigma_f}C_{\rm main}||| \mathcal{V}\sigma_f||| \exp\left(-\frac{1}{4C_{\nu, \sigma_f}}\frac{\ell_{\max}+1}{\max r_j} +\frac{1}{2}\right).
			\end{align*}
	\end{theorem}

	Definition \ref{def:6.7} of the approximation space implies that the numerical method defined by Equation \eqref{eq:Galerkina} is essentially a spectral Galerkin method, which are well-known to demonstrate exponential convergence for sufficiently smooth solution functions. Theorem \ref{thm:2} provides a proof of this intuitive result. We emphasise that the hypotheses of Theorem \ref{thm:2} are analogous to the hypotheses typically assumed by the discontinuous Galerkin finite element community for $hp$ finite elements (see, e.g., \cite{Schwab-Suli, houston2000stabilized, houston2001hp}).
	
	{We conclude this section by emphasising that, taken together, Theorems \ref{thm:1} and \ref{thm:2} establish that the accuracy of our numerical algorithm is robust with respect to the number of open balls $N$ for any family of geometries satisfying the assumptions {\textbf{A1)-A3)}}. Of course, in order to prove that the numerical method is \emph{linear scaling in accuracy}, we would have to prove in addition that for a fixed number of degrees of freedom per sphere, the computational cost of solving the linear system obtained from the Galerkin discretisation \eqref{eq:Galerkina} scales as $\mathcal{O}(N)$. Numerical evidence (see Section \ref{sec:Numerics} and also \cite{lindgren2018}) suggests that this is indeed the case. As mentioned in the introduction however, the current article is concerned with numerical analysis. A detailed complexity analysis of this numerical method is the subject of a second article \cite{Hassan2}.}

	
	\subsection{Existing Literature and Limitations}\label{sec:3}~
	Let us first establish our earlier claim that the boundary integral equations \eqref{eq:3.3a} is, essentially, an integral equation of the second kind. %
	
	\begin{lemma}\label{lem:2kind}
		Assume the setting of Section \ref{sec:2a}. The boundary integral equation \eqref{eq:3.3a} can be written as an integral equation of the second kind.
	\end{lemma}
	\begin{proof}
		Consider the BIE \eqref{eq:3.3a}. Standard results on boundary integral operators (see, e.g., \cite[Section 3.7]{Schwab}) imply that
		\begin{align*} 
		\text{DtN}\mathcal{V}= \frac{1}{2}I+\mathcal{K}^*,
		\end{align*}
		where $I \colon H^{-\frac{1}{2}}(\partial \Omega) \rightarrow H^{-\frac{1}{2}}(\partial \Omega)$ is the identity operator.
		
		The boundary integral equation \eqref{eq:3.3a} then implies that
		\begin{align*}
		\frac{4\pi}{\kappa_0}\sigma_f&=\nu - \frac{\kappa_0-\kappa}{\kappa_0} (\text{DtN}\mathcal{V})\nu=\nu - \frac{\kappa_0-\kappa}{\kappa_0} \Big(\frac{1}{2}I+\mathcal{K}^*\Big)\nu=\frac{\kappa_0 + \kappa}{2\kappa_0} \nu -\frac{\kappa_0-\kappa}{\kappa_0}\mathcal{K}^*\nu.
		\end{align*}
		
		Consequently, we obtain that
		\begin{align}
		\frac{4\pi}{\kappa_0+\kappa}\sigma_f&=\frac{1}{2}\nu - \frac{\kappa_0-\kappa}{\kappa_0+\kappa}\mathcal{K}^*\nu=\Big(\frac{1}{2}I - \frac{\kappa_0-\kappa}{\kappa_0+\kappa}\mathcal{K}^* \Big)\nu. \label{eq:3.4}
		\end{align}
		This completes the proof.
	\end{proof}
	
	Lemma \ref{lem:2kind} suggests that we might appeal to the classical well-posedness analysis of second kind integral equations in order to establish that the weak formulation \eqref{eq:weak1a} is well-posed. Broadly speaking, there are two popular approaches in the literature to establishing the well-posedness of second kind integral equations. 
	
	The traditional approach is based on recognising that the boundary integral operator $\mathcal{K} \colon L^2(\partial \Omega) \rightarrow L^2(\partial \Omega)$ is compact if $\Omega$ is a $C^1$ domain (which is indeed the case for the current problem). It follows that the BIE \eqref{eq:3.4} can be viewed as an operator equation on $L^2(\partial \Omega)$ involving a Fredholm operator of index 0, and well-posedness can be established by proving that the underlying operator is injective. This approach was first developed by E. B. Fabes, M. Jodeit, Jr., and N. M. Rivi\`ere in 1978 \cite{fabes1978potential}. In the general case when the domain $\Omega^-$ is only Lipschitz, the operator $\mathcal{K}$ is no longer compact on $L^2(\partial \Omega)$ but invertibility of the operator $\frac{1}{2}I - \frac{\kappa_0-\kappa}{\kappa_0+\kappa}\mathcal{K}$ on $L^2(\partial \Omega)$ can still be established as proven by Gregory Verchota in 1984 \cite{verchota1984layer}. These results can then be extended to the Sobolev spaces $H^s(\partial \Omega)$ (see, e.g., the work of Johannes Elschner \cite{Elschner1}).
	
	The primary issue with the above approaches is the following: Both analyses establish the invertibility of the underlying boundary integral operator \emph{indirectly}, by showing that the operator is injective. Thus, we are unable to obtain closed form expressions for the stability constants which means that we are unable to determine whether or not these constants are independent of $N$. 
	
	A second, more recent approach due to Steinbach and Wendland \cite{Wendland1, Steinbach1} (see also the book of Sauter and Schwab \cite{Schwab}) is based on variational techniques. This approach can be used to establish that the operator $\frac{1}{2}I -\frac{\kappa_0-\kappa}{\kappa_0+\kappa}\mathcal{K}$ is both bounded below and a contraction on $H^{\frac{1}{2}}(\partial \Omega)$ with respect to the inner product induced by the inverse single layer boundary operator $\mathcal{V}^{-1}$. This approach is based on the classical work of C. Neumann from the early $20^{\text{th}}$ century. Martin Costabel has published a fascinating article on the historical development of C. Neumann's work which also contains the core idea of the proof \cite{costabel2007some}.
	
	There are three fundamental issues with this variational approach. First, the lower bound constant for the operator $\frac{1}{2}I -\frac{\kappa_0-\kappa}{\kappa_0+\kappa}\mathcal{K}$ depends-- amongst others-- on the coercivity constant of the hypersingular boundary operator, and it is a priori unclear how this coercivity constant behaves as the number of objects $N$ is increased. Second, the analysis takes place in the Sobolev space $H^{\frac{1}{2}}(\partial \Omega)$ equipped with the inner-product induced by the inverse single layer boundary operator $\mathcal{V}^{-1}$, and this inner-product is completely non-local. Consequently, in order to qualitatively compare the relative error for different values of $N$, it becomes necessary to introduce norm equivalence constants and switch to the $H^{\frac{1}{2}}(\partial \Omega)$ norm. Unfortunately, these equivalence constants involve the continuity constant of $\mathcal{V}$, which increases as the number of objects $N$ increases. Finally, given our choice of approximation space, the Galerkin discretisation does not automatically inherit inf-sup stability from the infinite-dimensional case.

	In view of the preceding discussion, we felt it necessary to introduce a new well-posedness analysis for the weak formulation \eqref{eq:weak1a} and the Galerkin discretisation \eqref{eq:Galerkina}. The details of our analysis are presented in Section 4 but we remark briefly that we adopt an indirect approach and take advantage of the complementary decomposition of the space $H^{\frac{1}{2}}(\partial \Omega)$ introduced in Lemma \ref{lem:decomp}. We will show that this decomposition leads to a splitting of the weak formulation and Galerkin discretisation which then allows us to obtain suitable continuity and inf-sup constants that are indeed explicitly independent of the number of objects $N$.
	
	\section{Numerical Results}\label{sec:Numerics}

	The goal of this section is to briefly provide numerical evidence in support of our main results Theorems \ref{thm:1} and \ref{thm:2}. Our numerical experiments will therefore show that
	\begin{itemize}
		\item For a fixed number of degrees of freedom per sphere and geometries satisfying the assumptions {\textbf{A1)-A3)}}, the average error in the induced surface charge remains bounded as the number of open balls $N$ in the system is increased.
		\item For a fixed number of open balls $N$ in the system, the average error in the induced surface charge converges exponentially as the number of degrees of freedom per sphere is increased.
	\end{itemize}

	In addition, in order to anticipate future work on computational aspects of the numerical algorithm, we also provide numerical evidence indicating that {for a fixed number of degrees of freedom per sphere and geometries satisfying the assumptions {\textbf{A1)-A3)}}, the number of GMRES iterations required to solve the linear system arising from the Galerkin discretisation \eqref{eq:Galerkina} remains bounded as the number of open balls $N$ in the system is increased. Since we use the fast multipole method (FMM) in order to compute matrix vector products, these numerical results suggest that the computational cost of solving the underlying linear system scales as $\mathcal{O}(N)$.
		
			\begin{figure}[h!]
			\centering
			\begin{subfigure}{0.46\textwidth}
				\centering
				\includegraphics[width=0.95\textwidth]{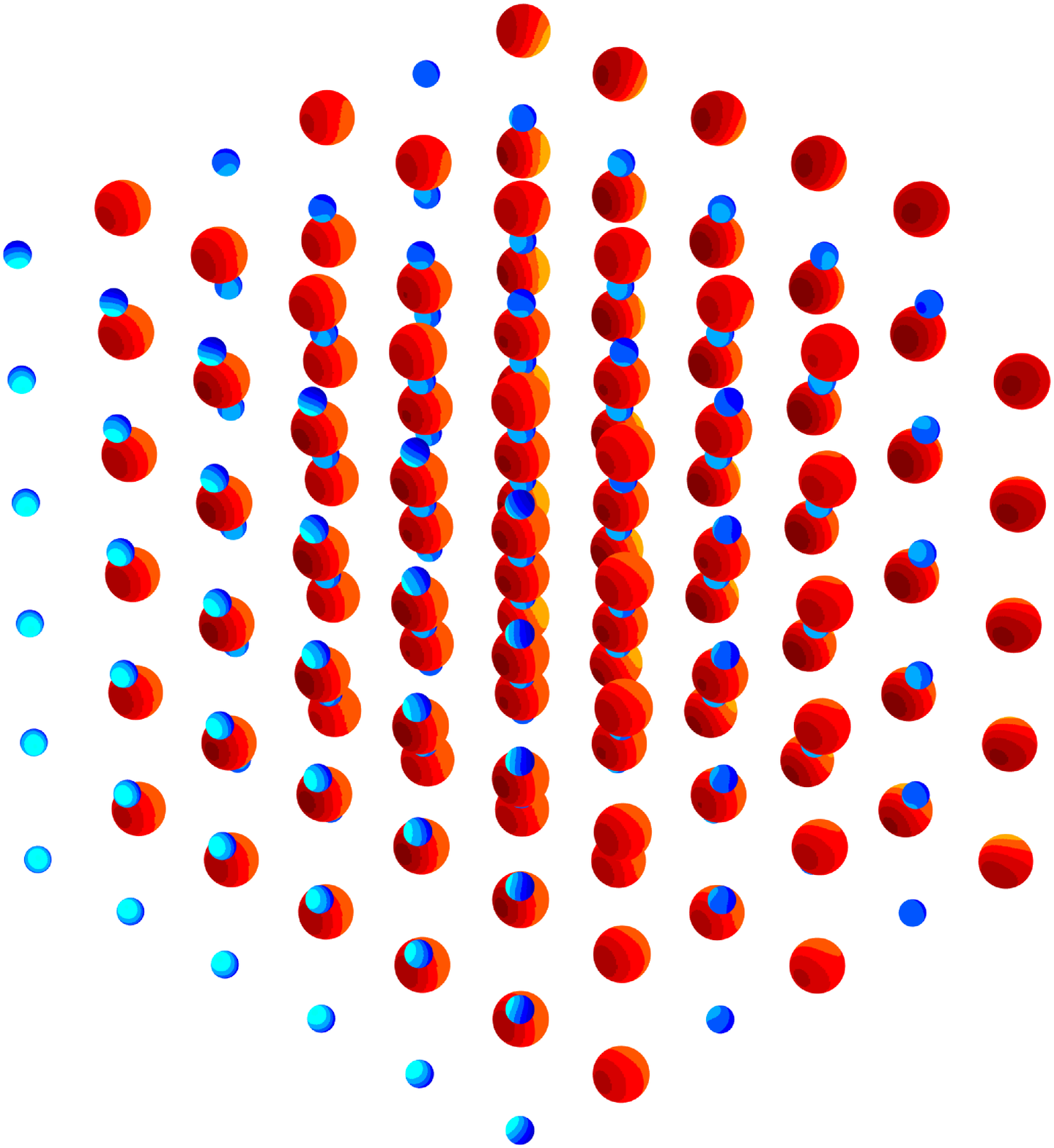} 
				\caption{Dielectric spheres arranged on a three dimensional, regular cubic lattice with edge length 10.}
				\label{fig:1}
			\end{subfigure}\hfill
			\begin{subfigure}{0.46\textwidth}
				\centering
				\includegraphics[width=0.95\textwidth]{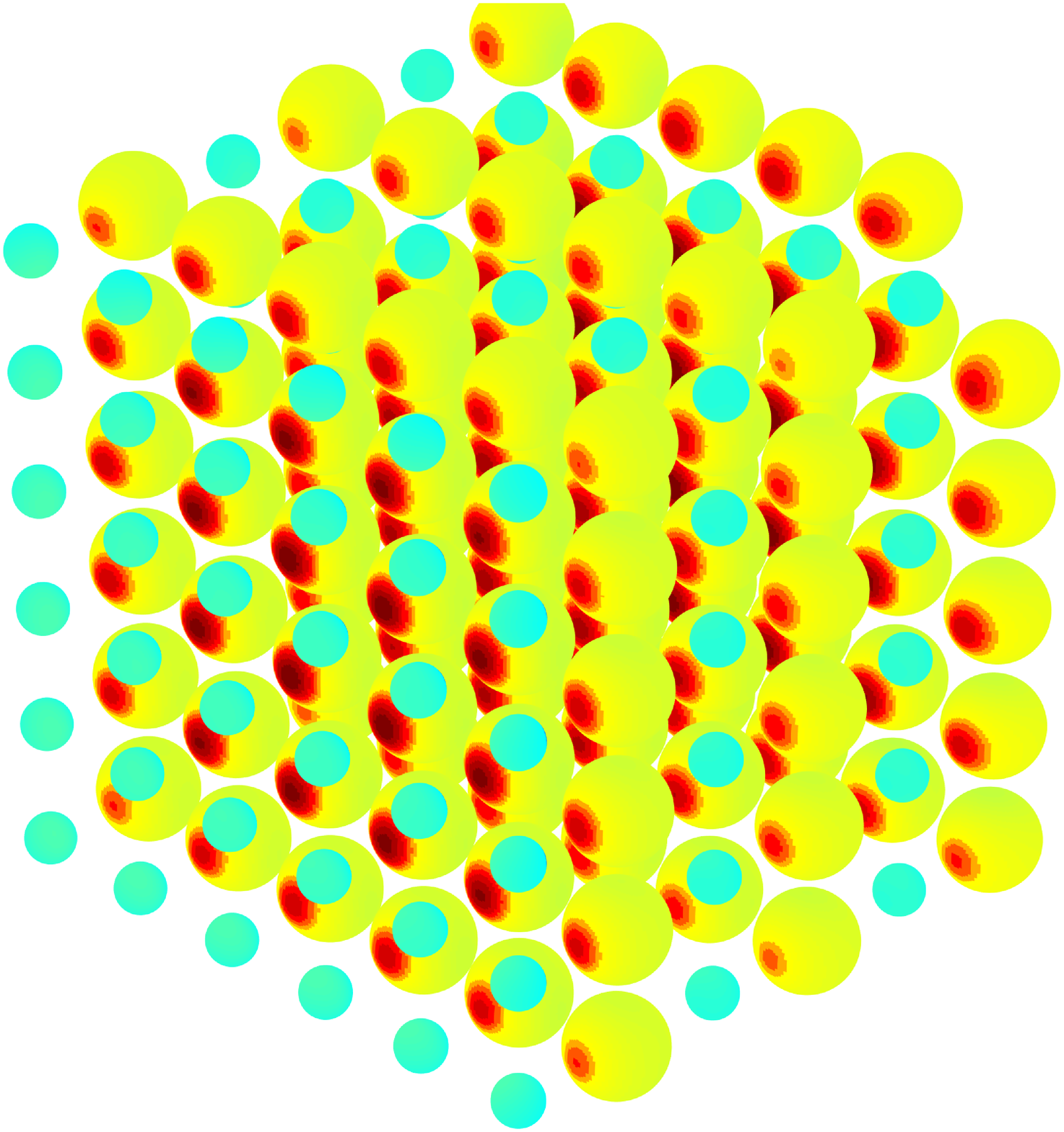} 
				\caption{Dielectric spheres arranged on a three dimensional, regular cubic lattice with edge length 5.}
				\label{fig:2}
			\end{subfigure}
			\caption{The geometric settings for both sets of numerical experiments.}
		\end{figure}

		{We consider the following geometric setting: The external medium is assumed to be vacuum which has a dielectric constant $\kappa_0 =1$. Two types of dielectric spheres are considered, one with radius 1, dielectric constant 10, and net negative free charge, and the other with radius 2, dielectric constant 5 and net positive free charge. Moreover, in order to include the effect of the minimal inter-sphere separation distance, we consider two sets of numerical experiments. The first involves the dielectric spheres arranged on a three dimensional, regular cubic lattice with edge length 10 and the other involves a similar lattice with a smaller edge length of 5 as displayed in Figures \ref{fig:1} and \ref{fig:2} respectively. All numerical simulations were run using a relative tolerance of $10^{-14}$.}
		
		\begin{figure}[h!]
		\centering
		\begin{subfigure}{0.45\textwidth}
			\centering
			\includegraphics[width=0.95\textwidth]{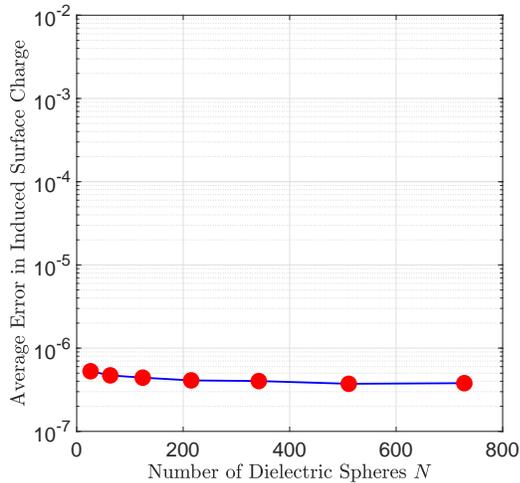} 
			\caption{Results for the cubic lattice with edge length 10.}
			\label{fig:3}
		\end{subfigure}\hfill
		\begin{subfigure}{0.45\textwidth}
			\centering
			\includegraphics[width=0.95\textwidth]{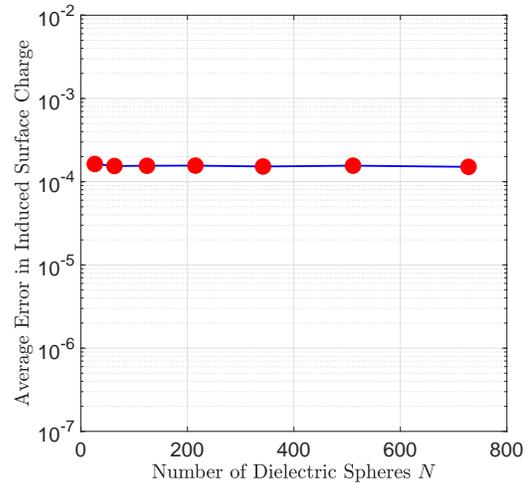} 
			\caption{Results for the cubic lattice with edge length 5.}
			\label{fig:4}
		\end{subfigure}
		\caption{Log-lin plot of the average error in the induced surface charge versus the number of dielectric spheres $N$. These numerical results support the conclusions of Theorem \ref{thm:1}.}
	\end{figure}
		
		\begin{figure}[h!]
		\centering
		\begin{subfigure}{0.46\textwidth}
			\centering
			\includegraphics[width=0.95\textwidth]{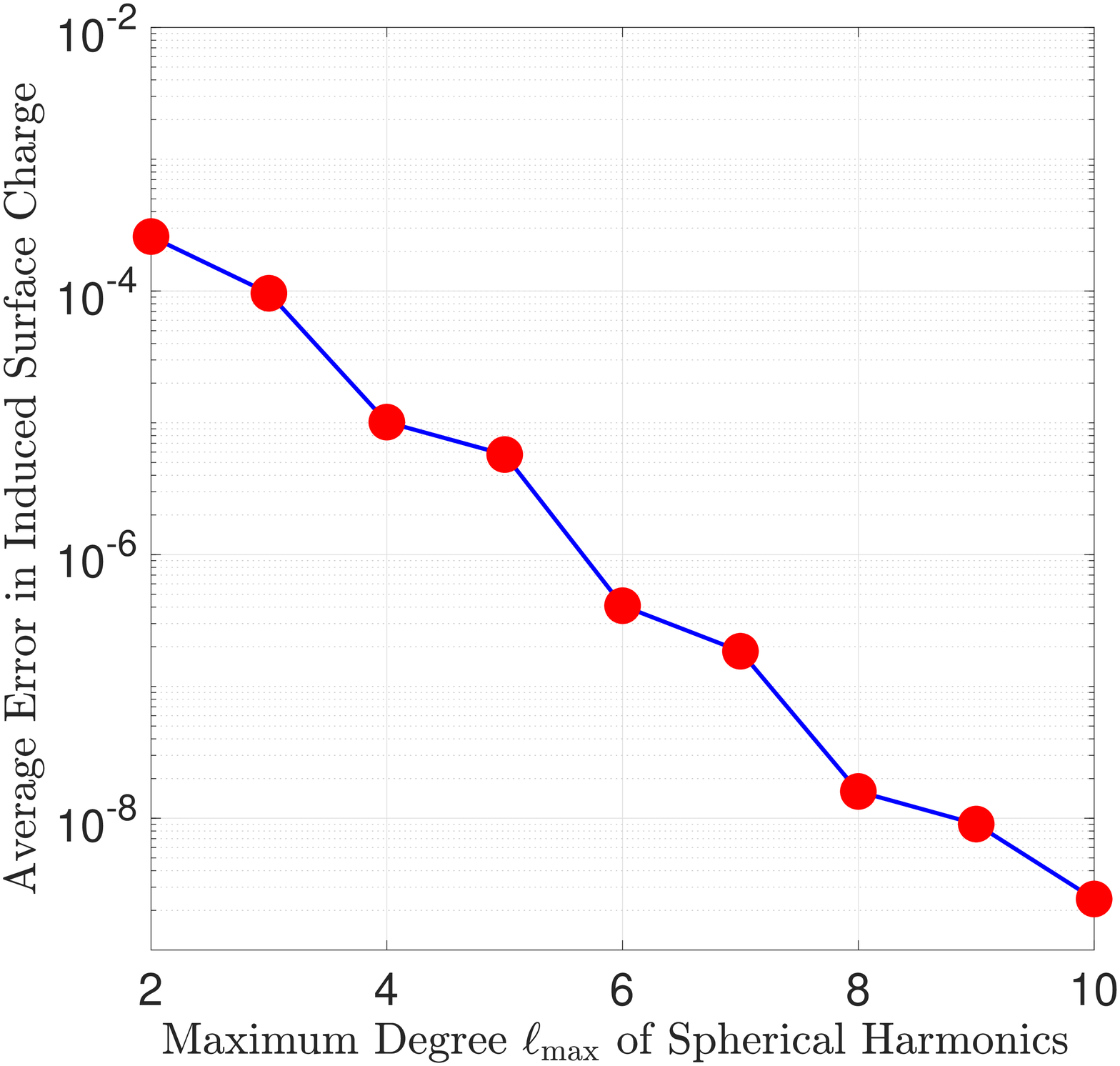} 
			\caption{Results for the cubic lattice with edge length 10.}
			\label{fig:5}
		\end{subfigure}\hfill
		\begin{subfigure}{0.46\textwidth}
			\centering
			\includegraphics[width=0.95\textwidth]{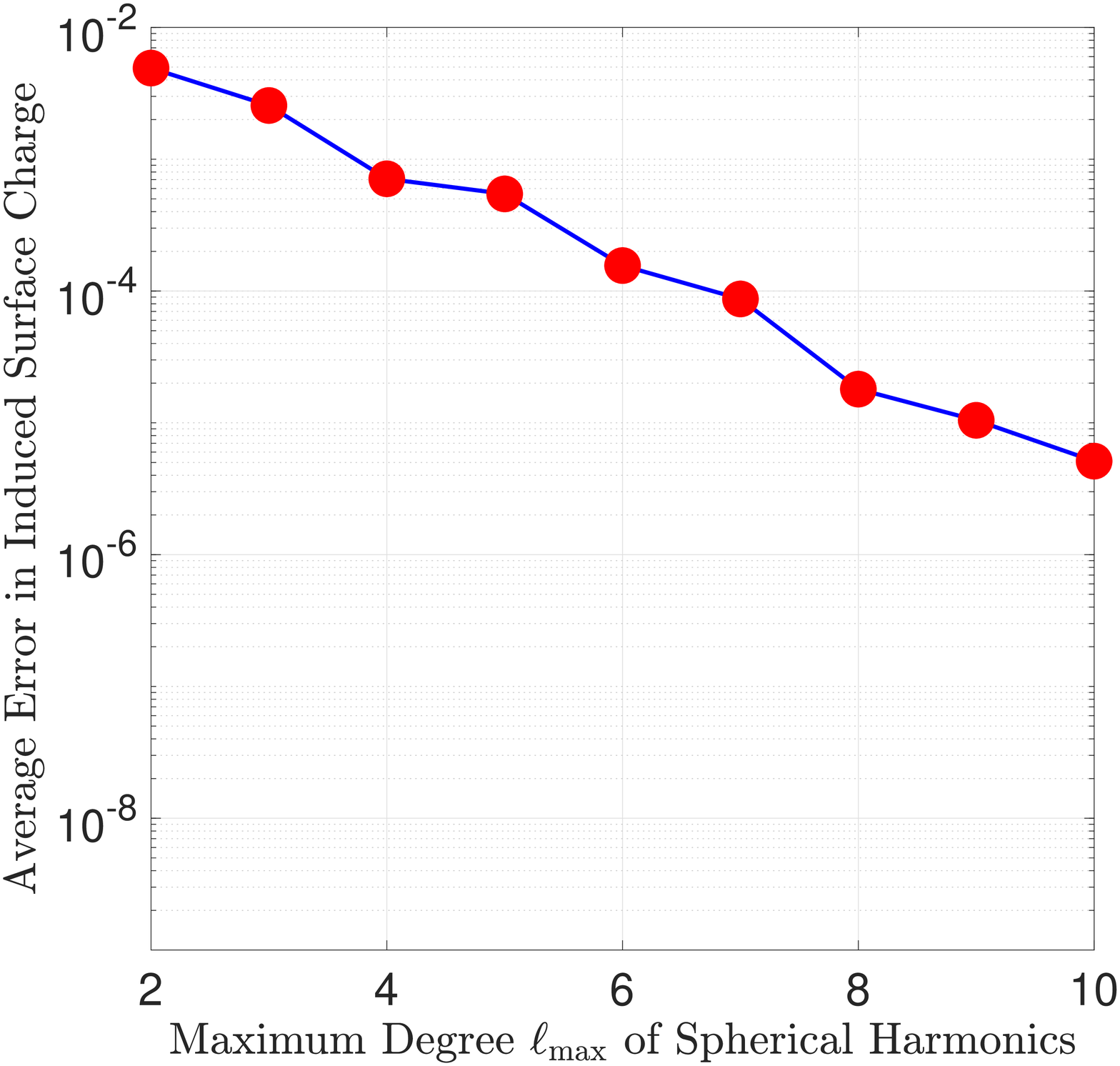} 
			\caption{Results for the cubic lattice with edge length 5.}
			\label{fig:6}
		\end{subfigure}
		\caption{Log-lin plot of the average error in the induced surface charge versus the maximum degree $\ell_{\max}$ of spherical harmonics in the approximation space on each open sphere. These numerical results support the conclusions of Theorem \ref{thm:2}.}
	\end{figure}
		
		{Figures \ref{fig:3} and \ref{fig:4} display the average error in the induced surface charge as the number of dielectric spheres $N$ is increased for the two types of lattices. The reference solution in both cases was constructed by setting the maximum degree of spherical harmonics in the approximation space on each sphere as $\ell_{\max}=20$. The approximate solutions were all constructed using $\ell_{\max}=6$.}

			\begin{figure}[h!]
			\centering
			\begin{subfigure}{0.46\textwidth}
				\centering
				\includegraphics[width=0.95\textwidth]{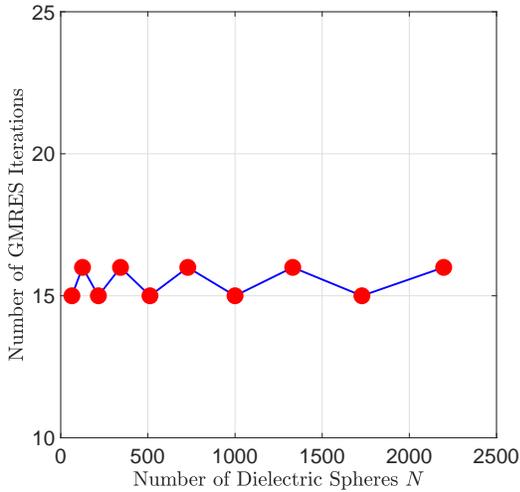} 
				\caption{Results for the cubic lattice with edge length 10.}
				\label{fig:7}
			\end{subfigure}\hfill
			\begin{subfigure}{0.46\textwidth}
				\centering
				\includegraphics[width=0.95\textwidth]{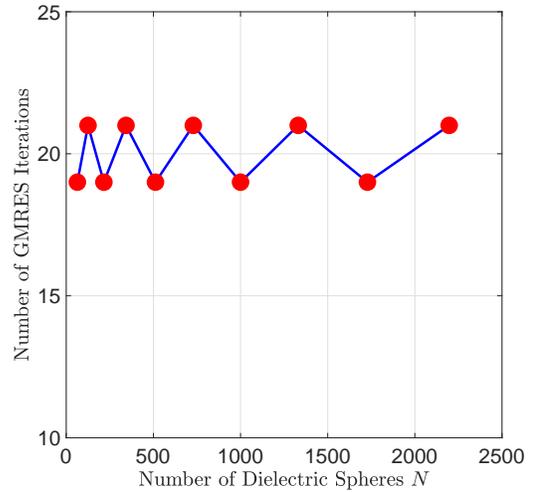} 
				\caption{Results for the cubic lattice with edge length 5.}
				\label{fig:8}
			\end{subfigure}
			\caption{The number of GMRES iterations required to solve the linear system arising from the Galerkin discretisation of the integral equation \eqref{eq:3.3a}.}
		\end{figure}

		{Figures \ref{fig:5} and \ref{fig:6} display the average error in the induced surface charge as the maximum degree of spherical harmonics $\ell_{\max}$ in the approximation space on each sphere is increased. The number of dielectric spheres was chosen as $N=215$. Once again, the reference solution in both cases was constructed by setting the maximum degree of spherical harmonics as $\ell_{\max}=20$.}

		{Finally, Figures \ref{fig:7} and \ref{fig:8} display the number of GMRES iterations required to solve the linear system arising from the Galerkin discretisation of the integral equation \eqref{eq:3.3a} for the two types of lattices. The maximum degree of spherical harmonics in the approximation space on each sphere was chosen as $\ell_{\max}=6$.}

		It is readily seen that these numerical results are in agreement with the conclusions of our main results Theorems \ref{thm:1} and Theorem \ref{thm:2}. {Furthermore, we observe that the average error and the number of GMRES iterations required to solve the linear system both increase as the minimum distance between two balls decreases.}
		
		\section{Proofs}\label{sec:Proofs}
		
		Assume the setting of Section \ref{sec:2a}. As mentioned in Section \ref{sec:3}, we need to introduce a new, indirect analysis in order to prove our main results Theorems \ref{thm:1} and \ref{thm:2}. To this end, we begin by observing that the single layer boundary operator $\mathcal{V} \colon H^{-\frac{1}{2}}(\partial \Omega) \rightarrow H^{\frac{1}{2}}(\partial \Omega)$ is a bijection. Therefore, the integral equation \eqref{eq:3.3a} can in fact be reformulated in terms of an unknown surface electrostatic potential $\lambda:=\mathcal{V}\nu \in H^{\frac{1}{2}}(\partial \Omega)$.\\ \\
		
		\noindent {\textbf{Integral Equation Formulation for the Electrostatic Potential}}~
		
		Let $\sigma_f \in H^{-\frac{1}{2}}(\partial \Omega)$. Find $\lambda \in H^{\frac{1}{2}}(\partial \Omega)$ with the property that
		\begin{align}\label{eq:3.3}
		\mathcal{A}\lambda = \lambda - \mathcal{V}\text{DtN}\Big(\frac{\kappa_0-\kappa}{\kappa_0} \lambda\Big)= \frac{4\pi}{\kappa_0}\mathcal{V}\sigma_f.
		\end{align}
		
		Naturally, the integral equation \eqref{eq:3.3} has a straightforward weak formulation. \\ \\
		
		\noindent {\textbf{Weak Formulation of the Integral Equation \eqref{eq:3.3}}}~
		
		Let $\sigma_f \in H^{-\frac{1}{2}}(\partial \Omega)$ and let $\mathcal{A} \colon H^{\frac{1}{2}}(\partial \Omega) \rightarrow H^{\frac{1}{2}}(\partial \Omega)$ be the operator defined through Definition \ref{def:A}. Find $\lambda \in H^{\frac{1}{2}}(\partial \Omega)$ such that for all $\sigma \in H^{-\frac{1}{2}}(\partial \Omega)$ it holds that
		\begin{align}\label{eq:weak1}
		\left\langle\sigma, \mathcal{A} \lambda\right\rangle_{\partial \Omega}= \frac{4\pi}{\kappa_0}\left \langle\sigma, \mathcal{V}\sigma_f \right \rangle_{\partial \Omega}.
		\end{align}
		
		The integral equation formulation \eqref{eq:3.3} now leads to a corresponding Galerkin discretisation for an unknown approximate surface electrostatic potential $\lambda_{\ell_{\max}} \in W^{\ell_{\max}}$.\\
		
		\noindent {\textbf{Galerkin Discretisation of the Integral Equation \eqref{eq:3.3}}}~
		
		Let $\sigma_f \in {H}^{-\frac{1}{2}}(\partial \Omega)$ and let $\ell_{\max} \in \mathbb{N}$. Find $\widehat{\lambda}_{\ell_{\max}} \in W^{\ell_{\max}}$ such that for all $\widehat{\psi}_{\ell_{\max}} \in W^{\ell_{\max}}$ it holds that
		\begin{align}\label{eq:Galerkin}
		\left(\widehat{\psi}_{\ell_{\max}}, \mathcal{A}\widehat{\lambda}_{\ell_{\max}}\right)_{L^2(\partial \Omega)}= \frac{4\pi}{\kappa_0}\left(\widehat{\psi}_{\ell_{\max}}, \mathcal{V}\sigma_f\right)_{L^2(\partial \Omega)}.
		\end{align}
		
		We emphasise that for the purpose of applications, one is typically interested in calculating either the induced surface charge $\nu \in H^{-\frac{1}{2}}(\partial \Omega)$ or the total electrostatic energy $\mathcal{E}$, which itself can be obtained directly from the induced surface charge $\nu$, and this is precisely why our main results Theorems \ref{thm:1} and \ref{thm:2} have been formulated in terms of the induced surface charge $\nu$ rather than the surface electrostatic potential $\lambda \in H^{\frac{1}{2}}(\partial \Omega)$. One may therefore wonder why we need introduce the weak formulation \eqref{eq:weak1} for the surface electrostatic potential $\lambda$ and its Galerkin discretisation \eqref{eq:Galerkin} at all.
		
		The key difficulty in our analysis is that the continuity constant of the relevant boundary integral operator and the discrete inf-sup constant both appear as pre-factors in the quasi-optimality bound and hence also the error estimates appearing in Theorems \ref{thm:1} and \ref{thm:2}. It therefore becomes essential to obtain both a continuity constant and an inf-sup constant that is independent of the number of balls $N$ in the $N$-body problem. Unfortunately, we have been unable to obtain such $N$-independent continuity and stability constants if we adopt a direct analysis of the weak formulation \eqref{eq:weak1a} for $\nu$ and its Galerkin discretisation \eqref{eq:Galerkina}.
		
		The weak formulation \eqref{eq:weak1} and the Galerkin discretisation \eqref{eq:Galerkin} have thus been introduced as \emph{analytical} tools that will aid our numerical analysis. As we will later show, the difficulties highlighted above can be avoided if we analyse first the weak formulation \eqref{eq:weak1} and its Galerkin discretisation \eqref{eq:Galerkin} involving the exact and approximate surface electrostatic potential and then obtain as a corollary, analogous results for the weak formulation \eqref{eq:weak1a} and the Galerkin discretisation \eqref{eq:Galerkina} and also proofs for Theorems \ref{thm:1} and \ref{thm:2}.
		
		We divide the remainder of this section into three parts. We first prove that the weak formulation \eqref{eq:weak1} and the Galerkin discretisation \eqref{eq:Galerkin} are well-posed, and obtain a \emph{partial} quasi-optimality result for the approximate surface electrostatic potential. Next, we prove that the weak formulation \eqref{eq:weak1a} and the Galerkin discretisation \eqref{eq:Galerkina} are also well-posed, and obtain an approximation result for the induced surface charge. Finally, we provide proofs for Theorems \ref{thm:1} and \ref{thm:2}.

		\subsection{Well-Posedness Analysis for the Surface Electrostatic Potential} \label{sec:4.1}
		\subsubsection{\textbf{{The Classical Analysis of the Infinite-Dimensional Problem and its Limitations}}}~
		The first step in the well-posedness analysis of the weak formulation \eqref{eq:weak1} of the boundary integral equation \eqref{eq:3.3} is to prove the continuity of the underlying linear boundary integral operator $\mathcal{A} \colon H^{\frac{1}{2}}(\partial \Omega) \rightarrow H^{\frac{1}{2}}(\partial \Omega)$ defined through Definition \ref{def:A}. 
		
		\begin{lemma}\label{lem:contin}
			Let the constants $c_{\mathcal{V}}$ and $c_{\mathcal{K}}$ be defined as in Properties 1 and 3 respectively of Section \ref{sec:2a}, let $\Vert \mathcal{K} \Vert_{L^2(\partial \Omega)}$ denote the $L^2$ operator norm of the double layer boundary operator $\mathcal{K} \colon H^{\frac{1}{2}}(\partial \Omega) \rightarrow H^{\frac{1}{2}}(\partial \Omega)$, and let the constant $C_{\mathcal{A}}$ be defined as
			\begin{align*}
			C_{\mathcal{A}} :=1+\max \Big\vert \frac{\kappa-\kappa_0}{\kappa_0}\Big \vert\sqrt{\Big(\frac{1}{2}+ \Vert \mathcal{K}\Vert_{L^2(\partial \Omega)}\Big)^2(1+\max r_i) + \frac{c_{\rm equiv}^2c_{\mathcal{K}}^3}{c_{\mathcal{V}}} }.
			\end{align*}
			
			Then the linear operator $\mathcal{A} \colon {H}^{\frac{1}{2}}(\partial \Omega) \rightarrow {H}^{\frac{1}{2}}(\partial \Omega)$ defined in Definition \ref{def:A} satisfies
			\begin{align*}
			\Vert \mathcal{A} \Vert_{\text{OP}}:= \sup_{0 \neq \lambda \in {H}^{\frac{1}{2}}(\partial \Omega)} \frac{ |||\mathcal{A}\lambda |||}{||| \lambda|||} \leq C_{\mathcal{A}}.
			\end{align*}
		\end{lemma}
		\begin{proof}
			Let $\lambda \in {H}^{\frac{1}{2}}(\partial \Omega)$. Then it holds that
			\begin{align*}
			|||\mathcal{A} \lambda ||| = \Big|\Big|\Big|\lambda - \mathcal{V} \text{DtN}\Big(\frac{\kappa_0-\kappa}{\kappa_0}\lambda\Big)\Big|\Big|\Big| &\leq |||\lambda|||+ \Big|\Big|\Big|\mathcal{V} \text{DtN}\Big(\frac{\kappa_0-\kappa}{\kappa_0}\lambda\Big)\Big|\Big|\Big|.
			\end{align*}
			
			Let $\lambda_{\kappa}:= \frac{\kappa_0-\kappa}{\kappa_0}\lambda$. Using Definition \ref{def:NewNorm} of the $||| \cdot |||$ norm we obtain
			\begin{align*}
			|||\mathcal{V} \text{DtN}\lambda_{\kappa}|||^2&= \left\Vert\mathbb{P}_0 \mathcal{V} \text{DtN} \lambda_{\kappa}\right\Vert^2_{L^2(\partial \Omega)}+\left\langle \text{DtN}\mathcal{V} \text{DtN}\lambda_{\kappa}, \mathcal{V} \text{DtN}\lambda_{\kappa}\right\rangle_{\partial \Omega}.
			\end{align*}

			Let us first focus on the second term. Using standard results on boundary integral operators (see, e.g., \cite[Section 3.7, Section 3.8, Theorem 3.5.3 and Theorem 3.8.7]{Schwab}), we obtain
			\begin{align*}
			\left\langle \text{DtN}\mathcal{V} \text{DtN}\lambda_{\kappa}, \mathcal{V} \text{DtN}\lambda_{\kappa}\right\rangle_{\partial \Omega}&=\left\langle \mathcal{V}^{-1}\mathcal{V}\text{DtN}\mathcal{V} \text{DtN}\lambda_{\kappa}, \mathcal{V} \text{DtN}\lambda_{\kappa}\right\rangle_{\partial \Omega}\\
			&=\left((\mathcal{V}\text{DtN})^2\lambda_{\kappa}, \mathcal{V} \text{DtN}\lambda_{\kappa}\right)_{\mathcal{V}^{-1}}\\
			&\leq \big\Vert(\mathcal{V}\text{DtN})^2\lambda_{\kappa}\big\Vert_{\mathcal{V}^{-1}}\; \big\Vert  \mathcal{V} \text{DtN}\lambda_{\kappa}\big\Vert_{\mathcal{V}^{-1}} \\
			&\leq c_{\mathcal{K}}^2\left\Vert\lambda_{\kappa}\right\Vert_{\mathcal{V}^{-1}}\; c_{\mathcal{K}}\left\Vert\lambda_{\kappa}\right\Vert_{\mathcal{V}^{-1}}\\
			& \leq\frac{c_{\mathcal{K}}^3c_{\rm equiv}^2}{c_{\mathcal{V}}}|||\lambda_{\kappa}|||^2\leq \frac{c_{\mathcal{K}}^3c_{\rm equiv}^2}{c_{\mathcal{V}}} \max \Big\vert \frac{\kappa-\kappa_0}{\kappa_0}\Big \vert^2 |||\lambda|||^2.
			\end{align*}

			Next, we consider the first term. The Calder\'on identities (see, e.g., \cite[Theorem 3.8.7]{Schwab}) imply that
			\begin{align*}
			\left\Vert\mathbb{P}_0 \mathcal{V} \text{DtN} \lambda_{\kappa}\right\Vert_{L^2(\partial \Omega)}^2=&\left\Vert\mathbb{P}_0 \Big(\frac{1}{2}I + \mathcal{K} \Big)\lambda_{\kappa}\right\Vert^2_{L^2(\partial \Omega)}
			\leq\Big(\frac{1}{2}+ \Vert \mathcal{K}\Vert_{L^2(\partial \Omega)}\Big)^2\left\Vert\lambda_{\kappa}\right\Vert^2_{L^2(\partial \Omega)}\\
			\leq& \Big(\frac{1}{2}+ \Vert \mathcal{K}\Vert_{L^2(\partial \Omega)}\Big)^2\max \Big\vert \frac{\kappa-\kappa_0}{\kappa_0}\Big \vert^2 \Vert\lambda\Vert^2_{L^2(\partial \Omega)}.
			\end{align*}
			
			Next, we observe that 
			\begin{align*}
			\Vert\lambda\Vert^2_{L^2(\partial \Omega)}&=\Vert \mathbb{P}_0\lambda\Vert^2_{L^2(\partial \Omega)}+ \Vert \mathbb{P}_0^{\perp}\lambda\Vert^2_{L^2(\partial \Omega)} =||| \mathbb{P}_0 \lambda|||^2 + \Vert \mathbb{P}_0^{\perp}\lambda\Vert^2_{L^2(\partial \Omega)}\\
			&\leq ||| \mathbb{P}_0 \lambda|||^2 + \max r_i ||| \mathbb{P}^{\perp}_0 \lambda|||^2 \leq (1+\max r_i) ||| \lambda|||^2.
			\end{align*}
			
			We conclude that
			\begin{align*}
			\Big|\Big|\Big|\mathcal{V}\text{DtN}\Big(\frac{\kappa_0-\kappa}{\kappa_0}\lambda\Big)\Big|\Big|\Big| \leq \max \Big\vert \frac{\kappa-\kappa_0}{\kappa_0}\Big \vert\sqrt{\Big(\frac{1}{2}+ \Vert \mathcal{K}\Vert_{L^2(\partial \Omega)}\Big)^2(1+\max r_i)+ \frac{c_{\mathcal{K}}^3c_{\rm equiv}^2}{c_{\mathcal{V}}} } ||| \lambda|||.
			\end{align*}
			
			The proof now follows.
		\end{proof}
		
		\begin{remark}\label{rem:Degrade}
			Consider the setting of Lemma \ref{lem:contin}. The continuity constant $C_{\mathcal{A}}$ of the operator $\mathcal{A}$ as determined in Lemma \ref{lem:contin} depends on the operator norm of the double layer boundary operator $\mathcal{K}$. Standard bounds for this operator norm depend on the diameter of the domain $\Omega^-$ (see, e.g., \cite[Chapter 7]{Folland} or \cite[Chapter 3]{Schwab}), which implies that the continuity constant $C_{\mathcal{A}}$ could potentially increase as the number of open balls $N$ increases. 
			
			Notice that the dependence of the continuity constant $C_{\mathcal{A}}$ on the operator norm $\Vert \mathcal{K}\Vert_{L^2(\partial \Omega)}$ appears only when evaluating the operator norm $\Vert \mathcal{V}\text{DtN} \Vert_{L^2(\partial \Omega)}$. In principle, it is possible to refine the estimate for the operator norm $\Vert \mathcal{V}\text{DtN} \Vert_{L^2(\partial \Omega)}$ using the addition theorem for spherical harmonics and the so-called Multipole-to-Local operators introduced by Greengard and Rokhlin \cite{greengard1}. Unfortunately, it turns out that for a completely arbitrary geometry $\Omega^-= \cup_{i=1}^N \Omega_i$, it is \emph{not possible to eliminate the dependence of the continuity constant $C_{\mathcal{A}}$ on the number of open balls $N$}. Indeed, an explicit counter-example can be constructed.
			
			Obviously, this degradation of the continuity constant poses a serious problem if wish to obtain error estimates independent of $N$. Fortunately, as we will now show, it is possible to circumvent this issue by taking advantage of the particular structure of the BIEs \eqref{eq:3.3a} and \eqref{eq:3.3}.	
		\end{remark}
		
		\subsubsection{\textbf{The New Analysis of the Infinite-Dimensional Problem}}~
		In principle, the next step in our analysis would be to prove that the weak formulation \eqref{eq:weak1} is well-posed. In view of Remark \ref{rem:Degrade} however, we cannot obtain $N$-independent stability and continuity constants using a straightforward analysis of the boundary integral operator $\mathcal{A}$, and we must therefore adopt a smarter, indirect approach. To this end, we will appeal to the complementary decompositions of the spaces $H^{\frac{1}{2}}(\partial \Omega)$ and $H^{-\frac{1}{2}}(\partial \Omega)$ introduced in Lemma \ref{lem:decomp}. This complementary decomposition, {together with Remark \ref{rem:Review_3}}, allows us to rewrite the weak formulation \eqref{eq:weak1} in terms of trial and test functions that belong to the spaces $\mathcal{C}(\partial \Omega)$, $\breve{H}^{\frac{1}{2}}(\partial \Omega)$, and $\breve{H}^{-\frac{1}{2}}(\partial \Omega)$.\\
		
		\noindent {\textbf{Modified Weak Formulation of the Integral Equation \eqref{eq:3.3}}}~
		
		Let $\sigma_f \in H^{-\frac{1}{2}}(\partial \Omega)$. Find functions $(\lambda_0, \tilde{\lambda}) \in \mathcal{C}(\partial \Omega) \times \breve{H}^{\frac{1}{2}}(\partial \Omega)$ such that for all test functions $(\sigma_0, \tilde{\sigma}) \in \mathcal{C}(\partial \Omega) \times \breve{H}^{-\frac{1}{2}}(\partial \Omega)$ it holds that
		\begin{align}\label{eq:weak2a}
		\left\langle \sigma_0, \lambda_0\right\rangle_{\partial \Omega}- \left\langle \sigma_0, \mathcal{V}\text{DtN}\left(\frac{\kappa_0-\kappa}{\kappa_0}\tilde{\lambda}\right)\right\rangle_{\partial \Omega}&= \frac{4\pi}{\kappa_0}\left \langle \sigma_0,\mathcal{V}\sigma_f \right \rangle_{\partial \Omega},\\[0.5em]
		\left\langle \tilde{\sigma}, \tilde{\lambda}\right\rangle_{\partial \Omega}- \left\langle \tilde{\sigma}, \mathcal{V}\text{DtN}\left(\frac{\kappa_0-\kappa}{\kappa_0}\tilde{\lambda}\right)\right\rangle_{\partial \Omega}&= \frac{4\pi}{\kappa_0}\left \langle\tilde{\sigma},\mathcal{V}\sigma_f \right \rangle_{\partial \Omega}. \label{eq:weak2b}
		\end{align}
		
		It is a simple exercise to prove that the modified weak formulation \eqref{eq:weak2a}-\eqref{eq:weak2b} is indeed equivalent to the weak formulation \eqref{eq:weak1}.\\
		
		Consider now Equations \eqref{eq:weak2a} and \eqref{eq:weak2b}. We observe that Equation \eqref{eq:weak2b} involves only the unknown function $\tilde{\lambda} \in  \breve{H}^{\frac{1}{2}}(\partial \Omega)$. It is therefore clear that if Equation \eqref{eq:weak2b} is uniquely solvable, then Equation \eqref{eq:weak2a} is also uniquely solvable, and hence the weak formulation \eqref{eq:3.3} is well-posed. Following standard practice in functional analysis, we prove unique solvability of Equation \eqref{eq:weak2b}  by establishing that the underlying reduced bilinear form is bounded and satisfies the inf-sup condition.
		
		\begin{remark}
			In principle, one could use the same complementary decomposition to split the weak formulation \eqref{eq:weak1a} for the induced surface charge $\nu$. In this case however, we do not obtain the useful ``upper-triangular'' structure highlighted above, and consequently our subsequent analysis cannot be applied. 
		\end{remark}
		
		\begin{definition}\label{def:Atilde}
			We define the ``reduced'' bilinear form $\tilde{a} \colon \breve{H}^{\frac{1}{2}}(\partial \Omega) \times \breve{H}^{-\frac{1}{2}}(\partial \Omega)\rightarrow \mathbb{R}$ as the mapping with the property that for all $\tilde{\lambda} \in \breve{H}^{\frac{1}{2}}(\partial \Omega)$ and all $\tilde{\sigma} \in \breve{H}^{-\frac{1}{2}}(\partial \Omega)$ it holds that
			\begin{align*}
			\tilde{a}(\tilde{\lambda}, \tilde{\sigma}) := \left\langle \tilde{\sigma}, \tilde{\lambda}\right\rangle_{\partial \Omega}- \left\langle  \tilde{\sigma}, \mathcal{V}\text{DtN}\left(\frac{\kappa_0-\kappa}{\kappa_0}\tilde{\lambda}\right)\right\rangle_{\partial \Omega}.
			\end{align*}
		\end{definition}

		We first prove that the reduced bilinear form $\tilde{a}$ is bounded.
		
		\begin{lemma}\label{lem:contin2}
			Let the constant $C_{\tilde{\mathcal{A}}}$ be defined as
			\begin{align}\label{eq:contin}
			C_{\tilde{\mathcal{A}}} :=1+\max\left \vert \frac{\kappa - \kappa_0}{\kappa_0}\right \vert\cdot \left(\frac{c^{\frac{3}{2}}_{\mathcal{K}}c_{\rm equiv}}{\sqrt{c_{\mathcal{V}}}} \right),
			\end{align}
			and let the bilinear form $\tilde{a} \colon \breve{H}^{\frac{1}{2}}(\partial \Omega) \times \breve{H}^{-\frac{1}{2}}(\partial \Omega)\rightarrow \mathbb{R}$ be defined as in Definition \ref{def:Atilde}. Then for all $\tilde{\lambda} \in \breve{H}^{\frac{1}{2}}(\partial \Omega)$ and all $\tilde{\sigma} \in \breve{H}^{-\frac{1}{2}}(\partial \Omega)$ it holds that
			\begin{align*}
			\vert \tilde{a}(\tilde{\lambda}, \tilde{\sigma}) \vert  \leq C_{\tilde{\mathcal{A}}} ||| \lambda|||\, ||| \sigma|||^*.
			\end{align*}
		\end{lemma}
		\begin{proof}
			Let the linear operator $\tilde{\mathcal{A}} \colon \breve{H}^{\frac{1}{2}}(\partial \Omega)  \rightarrow \breve{H}^{\frac{1}{2}}(\partial \Omega)$ be defined as $\tilde{A}:= \mathbb{P}_0^{\perp} \mathcal{A}\mathbb{P}_0^{\perp}$. Then $\tilde{\mathcal{A}}$ is the linear operator associated with the reduced bilinear form $\tilde{a}$. {Since $||| \mathbb{P}_0^{\perp} \lambda ||| \leq ||| \lambda |||$ for all $\lambda \in H^{\frac{1}{2}}(\partial \Omega)$, the proof becomes identical to the first part of the proof of Lemma \ref{lem:contin} with one minor modification.}
		\end{proof}
		
		\begin{remark}\label{rem:HassanV}
			Consider the setting of Lemma \ref{lem:contin2} and the continuity constant $C_{\tilde{\mathcal{A}}}$ of the modified boundary integral operator $\tilde{\mathcal{A}}$. We observe that the constant $c_{\mathcal{K}}$ is bounded by one, and therefore the only non quantified constant appearing in the expression of $C_{\tilde{\mathcal{A}}}$ is the coercivity constant $c_{\mathcal{V}}$. A priori, it is not clear how this coercivity constant depends on the geometrical setting of our problem including the number of open balls $N$ in our system. The next step in our analysis therefore, is to obtain a closed form expression for this coercivity constant and to show in particular that it does not explicitly depend on $N$.
		\end{remark}
		
		{
			We first require the following lemma:
			\begin{lemma}\label{lem:Single}
				There exist constants $c_{\rm{int}}, c_{\rm{ext}} >0$ that are independent of the number $N$ of open balls such that for all harmonic functions $v \in H^1(\Omega^-)$ and $w \in H^1(\Omega^+)$ it holds that
				\begin{align*}
				\Vert \gamma^-_N v\Vert_{H^{-\frac{1}{2}}(\partial \Omega)} \leq c_{\rm{int}} \Vert  \nabla v\Vert_{L^2(\Omega^-)},\\
				\intertext{and}
				\Vert \gamma^+_N w\Vert_{H^{-\frac{1}{2}}(\partial \Omega)} \leq c_{\rm{ext}} \Vert  \nabla w\Vert_{L^2(\Omega^+)}.
				\end{align*}
				
				{Additionally, the constant $c_{\rm int}$ depends only on the radii $\{r_j\}_{j=1}^N$ of the open balls while the constant $c_{\rm ext}$ depends on both the radii of the open balls as well as the minimum inter-sphere separation distance, i.e., $\min_{\substack{i, j \in \{1, \ldots, N\}\\ i \neq j}} \big(\vert \bold{x}_i -\bold{x}_j\vert - r_i -r_j\big)$.}
			\end{lemma}
			\begin{proof}
				The first bound is straightforward to prove. Indeed, let $\mathcal{E}_{\mathcal{H}}^{\rm int} \colon H^{\frac{1}{2}}(\partial \Omega) \rightarrow H^1(\Omega^-)$ be defined as the interior harmonic extension operator on $\Omega^-$. A direct calculation yields
				\begin{align*}
				\Vert \gamma^-_N v\Vert_{H^{-\frac{1}{2}}(\partial \Omega)}= \sup_{0\neq \lambda \in H^{\frac{1}{2}}(\partial \Omega)} \dfrac{\langle \gamma^-_N v, \lambda \rangle_{\partial \Omega} }{\Vert \lambda\Vert_{H^{\frac{1}{2}}(\partial \Omega) }}&= \sup_{0\neq \lambda \in H^{\frac{1}{2}}(\partial \Omega)} \dfrac{ \int_{\Omega^-} \nabla v(x) \cdot \nabla \mathcal{E}^{\rm int}_{\mathcal{H}}\lambda(x)\, dx}{\Vert \lambda\Vert_{H^{\frac{1}{2}}(\partial \Omega) }}\\[0.5em]
				&\leq \Vert  \nabla v\Vert_{L^2(\Omega^-)} \sup_{0\neq \lambda \in H^{\frac{1}{2}}(\partial \Omega)} \dfrac{\Vert  \nabla \mathcal{E}^{\rm int}_{\mathcal{H}}\lambda\Vert_{L^2(\Omega^-)}}{\Vert \lambda\Vert_{H^{\frac{1}{2}}(\partial \Omega) }}\\[0.5em]
				&\leq\Vert  \nabla v\Vert_{L^2(\Omega^-)} \sup_{0 \neq \lambda \in H^{\frac{1}{2}}(\partial \Omega)} \dfrac{\Vert \mathcal{E}^{\rm int}_{\mathcal{H}}\lambda\Vert_{H^1(\Omega^-)}}{\Vert \lambda\Vert_{H^{\frac{1}{2}}(\partial \Omega) }}\\[0.5em]
				&\leq c_{\rm equiv}\Vert \mathcal{E}^{\rm int}_{\mathcal{H}}\Vert_{\text{OP}}\Vert \nabla v\Vert_{L^2(\Omega^-)},
				\end{align*}
				where the $N$-independent norm equivalence constant $c_{\rm equiv}$ arises due to the fact that by our convention,  $H^{\frac{1}{2}}(\partial \Omega)$ is equipped with the new $||| \cdot |||$ given by Definition \ref{def:NewNorm} rather than the Sobolev-Slobodeckij norm $\Vert \cdot \Vert_{H^{\frac{1}{2}}(\partial \Omega)}$. Since $\Omega^-$ is simply the union of non-intersecting open balls, i.e., $\Omega^- = \cup_{j=1}^N \Omega_j$, it is easy to see that the operator norm $\Vert \mathcal{E}^{\rm int}_{\mathcal{H}}\Vert_{\text{OP}}$ depends only on the radii $\{r_j\}_{j=1}^N$ of the open balls $\{\Omega_j\}_{j=1}^N$ and is independent of the number $N$ of open balls. This completes the proof for the first bound. \vspace{3mm}
				
				In order to compute the second bound, we require more work. The essential idea is to mimic the proof for the first bound but this requires us to first define an extension operator $\mathcal{E}_{\rm external} \colon H^{\frac{1}{2}}(\partial \Omega) \rightarrow H^1(\Omega^+)$ whose operator norm is also independent of $N$. We proceed in four steps. \vspace{3mm}
				
				\begin{enumerate}
					\item[Step 1)] We first define a family of one-dimensional continuously differentiable cutoff functions. To this end, let $r> 0$ and $\epsilon>0$ be real numbers. We define the cubic polynomial $p_{r, \epsilon} \colon \mathbb{R} \rightarrow \mathbb{R}$ as
					\begin{align*}
					\forall x \in \mathbb{R} \colon \quad p_{r, \epsilon}(x)= \frac{1}{\epsilon^3}\Big(2x^3 - 3(2r+\epsilon)x^2 + 6r(r+\epsilon)x-(r + \epsilon)^2 (2r-\epsilon) \Big).
					\end{align*}
					
					Then for any $r>0$ and $\epsilon > 0$ we define the cutoff function $\phi_{r, \epsilon} \colon \mathbb{R} \rightarrow [0, 1]$ as the mapping with the mapping with the property that for all $x \in \mathbb{R}$ it holds that
					\begin{align*}
					\phi_{r, \epsilon}(x) := \begin{cases}
					1 &\quad \text{if } x \leq r,\\
					p_{r, \epsilon}(x) &\quad  \text{if } x \in (r, r+\epsilon),\\
					0 &\quad \text{if } x \geq r+\epsilon.
					\end{cases}
					\end{align*}
					
					Let $r > 0$ and $\epsilon > 0$ be fixed. It can readily be verified that the cutoff function $\phi_{r, \epsilon} \in C^1(\mathbb{R})$, $\Vert \phi_{r, \epsilon}\Vert_{L^{\infty}(\mathbb{R})}=1$, and furthermore that $\Vert \phi'_{r, \epsilon}\Vert_{L^{\infty}(\mathbb{R})}=\frac{3}{2\epsilon}$. \\
					
					\item[Step 2)] Let $i \in \{1, \ldots, N\}$. We define the (exterior) harmonic extension operator $\mathcal{E}_{i, \mathcal{H}}^{\rm ext} \colon H^{\frac{1}{2}}(\partial \Omega_i) \rightarrow H^1(\mathbb{R}^3 \setminus \Omega_i)$ as follows: Given any $\lambda_i \in H^{\frac{1}{2}}(\partial \Omega_i)$, there exist coefficients $[\lambda_i]_{\ell}^m, ~ \ell \in \mathbb{N}_0, ~-\ell \leq m \leq \ell$ such that for all $\bold{x} \in \partial \Omega_i$ it holds that
					\begin{align*}
					\lambda_i(\bold{x}) = \sum_{{\ell}=0}\sum_{m=-\ell}^{m=\ell}[\lambda_i]_{\ell}^m \mathcal{Y}_{\ell}^m\left(\frac{\bold{x}-\bold{x}_i}{\vert \bold{x}-\bold{x}_i\vert}\right).
					\end{align*}
					
					We therefore define 
					\begin{align}\label{eq:Hassan1}
					\big(\mathcal{E}_{i, \mathcal{H}}^{\rm ext}\lambda_i\big)(\bold{x}) := \sum_{{\ell}=0}^{\infty}\sum_{m=-\ell}^{m=\ell}[\lambda_i]_{\ell}^m \left(\frac{r_i}{\vert \bold{x} -\bold{x}_i \vert}\right)^{\ell+1}\mathcal{Y}_{\ell}^m\left(\frac{\bold{x}-\bold{x}_i}{\vert \bold{x}-\bold{x}_i\vert}\right),
					\end{align}
					for all $\bold{x} \in \mathbb{R}^3$ such that $\vert \bold{x} -\bold{x}_i \vert \geq r_i$. The boundedness of this operator can be deduced from the well-posedness and regularity results on the exterior Dirichlet problem for the Laplace equation. \\
					
					\item[Step 3)] We now recall that we have by assumption that the minimum separation distance of the open balls $\{\Omega_i\}_{i=1}^N$ is uniformly bounded below with respect to $N$. Let $\widetilde{\epsilon}> 0$ be a lower bound for this separation distance and define $\epsilon:= \frac{\widetilde{\epsilon}}{4}$. Moreover, let once again $i \in \{1, \ldots, N\}$. We now define the \emph{local} extension operator $\mathcal{E}^i_{\rm external} \colon H^{\frac{1}{2}}(\partial \Omega_i) \rightarrow H^1(\Omega^+)$ as the mapping with the property that for all $\lambda_i \in H^{\frac{1}{2}}(\partial \Omega_i)$ and all $\bold{x} \in \Omega^+$ it holds that
					\begin{align*}
					\big(\mathcal{E}^i_{\rm external}\lambda_i\big)(\bold{x}):= \big(\mathcal{E}_{i, \mathcal{H}}^{\rm ext}\lambda_i\big)(\bold{x}) \phi_{r_i, \epsilon}\big(\vert \bold{x} -\bold{x}_i\vert\big).
					\end{align*}
					
					Intuitively, this local extension operator $\mathcal{E}^i_{\rm external}$ takes as input Dirichlet data on $\partial \Omega_i$, constructs the exterior harmonic extension according to Equation \eqref{eq:Hassan1}, and then multiplies this extension with a smooth cut-off function. The following properties of this local extension operator can easily be deduced: \\ \\
					
					\begin{itemize}
						
						\item[Property 1:] For all $\bold{x} \in \Omega^+$ it holds that $\big(\mathcal{E}^i_{\rm external}\lambda_i\big)(\bold{x}) \leq \big(\mathcal{E}_{i, \mathcal{H}}^{\rm ext}\lambda_i\big)(\bold{x})$. 
						
						\item[Property 2:] For all $\bold{x} \in \Omega^+$ such that $\vert \bold{x} -\bold{x}_i \vert \geq r_i+ \epsilon$, it holds that $\big(\mathcal{E}^i_{\rm external}\lambda_i\big)(\bold{x}) =0$. In other words, the local extension operator $\mathcal{E}^i_{\rm external}$ is zero outside a ball of radius $ r_i+ \epsilon$ centred at $\bold{x}_i$, i.e., the centre of the open ball $\Omega_i$. This implies in particular that the local extension operator $\mathcal{E}^i_{\rm external}$ is zero on all closed balls $\overline{\Omega_j}, j \in \{1, \ldots, N\}$ such that $j\neq i$.
						
						\item[Property 3:] For all $\bold{x} \in \Omega^+$ such that $\vert \bold{x} -\bold{x}_i \vert < r_i+ \epsilon$, the gradient $ \nabla_{\bold{x}}\big(\mathcal{E}^i_{\rm external}\lambda_i\big)(\bold{x})$ in cartesian coordinates satisfies:
						\begin{align*}
						\vert \nabla_{\bold{x}}\big(\mathcal{E}^i_{\rm external}\lambda_i\big)(\bold{x})\vert &= \Big \vert \phi_{r_i, \epsilon}\big(\vert \bold{x} -\bold{x}_i\vert\big)\nabla_{\bold{x}} \big(\mathcal{E}_{i, \mathcal{H}}^{\rm ext}\lambda_i\big)(\bold{x}) +  \big(\mathcal{E}_{i, \mathcal{H}}^{\rm ext}\lambda_i\big)(\bold{x}) \nabla_{\bold{x}} \phi_{r_i, \epsilon}\big(\vert \bold{x} -\bold{x}_i\vert\big)\Big \vert\\
						&\leq \Big \vert \nabla_{\bold{x}} \big(\mathcal{E}_{i, \mathcal{H}}^{\rm ext}\lambda_i\big)(\bold{x})\Big\vert  +  \Big\vert \big(\mathcal{E}_{i, \mathcal{H}}^{\rm ext}\lambda_i\big)(\bold{x})\phi_{r_i, \epsilon}'\big(\vert \bold{x} -\bold{x}_i\vert\big)\Big \vert\\
						&= \Big \vert \nabla_{\bold{x}} \big(\mathcal{E}_{i, \mathcal{H}}^{\rm ext}\lambda_i\big)(\bold{x})\Big\vert  + \frac{3}{2\epsilon} \Big\vert \big(\mathcal{E}_{i, \mathcal{H}}^{\rm ext}\lambda_i\big)(\bold{x})\Big \vert.
						\end{align*}
					\end{itemize}
					Of course, we have not yet shown that the mapping $\mathcal{E}^i_{\rm external} \colon H^{\frac{1}{2}}(\partial \Omega_i) \rightarrow H^1(\Omega^+)$ is bounded as claimed. In order to show this, let us denote by $B_{r_i + \epsilon}(\bold{x}_i)$ the open ball of radius $r_i + \epsilon$ with centre at~$\bold{x}_i$. Then combining properties 1 and 3 yields that
					\begin{align*}
					\Vert \mathcal{E}^i_{\rm external}\lambda_i\Vert^2_{H^1(\Omega^+)}  &= \int_{\Omega^+} \frac{\vert\mathcal{E}^i_{\rm external}\lambda_i(\bold{x})\vert^2}{1 + \vert \bold{x}\vert^2}\, d\bold{x} + \int_{\Omega^+} \vert \nabla_{\bold{x}}\big(\mathcal{E}^i_{\rm external}\lambda_i\big)(\bold{x})\vert^2\, d \bold{x}\\
					&\leq \left(1+\frac{9}{2\epsilon^2}\right)\int_{\Omega^+ \cap B_{r_i + \epsilon}(\bold{x}_i)} \vert\mathcal{E}_{i, \mathcal{H}}^{\rm ext}\lambda_i(\bold{x})\vert^2\, d\bold{x}+2\int_{\Omega^+ \cap B_{r_i + \epsilon}(\bold{x}_i)} \left\vert \nabla_{\bold{x}}\big(\mathcal{E}_{i, \mathcal{H}}^{\rm ext}\lambda_i\big)(\bold{x})\right\vert^2\, d \bold{x}\\[0.5em]
					&\leq \max\left\{2, 1+\frac{9}{2\epsilon^2}\right\} \Vert \mathcal{E}_{i, \mathcal{H}}^{\rm ext}\lambda_i\Vert^2_{H^1\big(\Omega^+ \cap B_{r_i + \epsilon}(\bold{x}_i)\big)}.
					\end{align*}
					
					In order to simplify the final expression we first use Equation \eqref{eq:Hassan1} to simplify the $L^2\big(\Omega^+ \cap B_{r_i + \epsilon}(\bold{x}_i)\big)$ norm. For reasons that will subsequently become clear, we adopt the convention that the space $H^{\frac{1}{2}}(\partial \Omega_i)$ is equipped with the norm $||| \cdot |||_{H^{\frac{1}{2}}(\partial \Omega_i)}$ defined through Definition \ref{def:7.1}. A direct calculation yields
					\begin{align*}
					\Vert \mathcal{E}_{i, \mathcal{H}}^{\rm ext}\lambda_i\Vert^2_{L^2\big(\Omega^+ \cap B_{r_i + \epsilon}(\bold{x}_i)\big)} &\leq \frac{1}{3} \left((r_i+\epsilon)^3 - r_i^3\right)\sum_{{\ell}=0}^{\infty}\sum_{m=-\ell}^{\ell}\left([\lambda_i]_{\ell}^m\right)^2\\
					&= \left(\epsilon r_i^2 +  \epsilon^2 r_i + \frac{\epsilon^3}{3}\right)\sum_{{\ell}=0}^{\infty}\sum_{m=-\ell}^{\ell}\left([\lambda_i]_{\ell}^m\right)^2\\
					&\leq \epsilon\max\left\{\frac{1}{r_i}, \frac{1}{r_i^2}\right\} \left(r_i^2 + \epsilon r_i + \frac{\epsilon^2}{3}\right) ||| \lambda_i |||^2_{H^{\frac{1}{2}}(\partial \Omega_i)}.
					\end{align*}
					
					Next, we use the fact that the local extension $\mathcal{E}_{i, \mathcal{H}}^{\rm ext}\lambda_i$ is a harmonic function so that Green's identity applies in the domain $\Omega^+ \cap B_{r_i + \epsilon}(\bold{x}_i)$. Simple calculus then yields that
					
					\begin{align*}
					\Vert \nabla \mathcal{E}_{i, \mathcal{H}}^{\rm ext}\lambda_i\Vert^2_{L^2\big(\Omega^+ \cap B_{r_i + \epsilon}(\bold{x}_i)\big)}&= r^2_i \sum_{{\ell}=0}^{\infty}\sum_{m=-\ell}^{\ell}\frac{\ell+1}{r_i}\left([\lambda_i]_{\ell}^m\right)^2 - (r_i+\epsilon)^2 \sum_{{\ell}=0}^{\infty}\sum_{m=-\ell}^{\ell}\frac{\ell+1}{r_i+\epsilon}\left([\lambda_i]_{\ell}^m\right)^2\left(\frac{r_i}{r_i+\epsilon}\right)^{2\ell+2}\\
					&= r^2_i \sum_{{\ell}=0}^{\infty}\sum_{m=-\ell}^{\ell}(\ell+1)\left([\lambda_i]_{\ell}^m\right)^2\left(\frac{1}{r_i}- \frac{1}{r_i+\epsilon}\left(\frac{r_i}{r_i+\epsilon}\right)^{2\ell}\right)\\
					&= r^2_i \sum_{{\ell}=0}^{\infty}\sum_{m=-\ell}^{\ell}(\ell+1)\left([\lambda_i]_{\ell}^m\right)^2\frac{(r_i+\epsilon)^{2\ell+1}-r_i^{2\ell+1}}{r_i (r_i+\epsilon)^{2\ell+1}}.
					\end{align*}
					
					This last expression can be further simplified by observing that for all $\ell \geq 0$ it holds that
					\begin{align*}
					\frac{(r_i+\epsilon)^{2\ell+1}-r_i^{2\ell+1}}{r_i (r_i+\epsilon)^{2\ell+1}} = \frac{(1+\frac{\epsilon}{r_i})^{2\ell+1}-1}{r_i (1+ \frac{\epsilon}{r_i})^{2\ell+1}}=\frac{\frac{\epsilon}{r_i}(1+\frac{\epsilon}{r_i})^{2\ell} +(1+\frac{\epsilon}{r_i})^{2\ell} -1}{r_i (1+ \frac{\epsilon}{r_i})^{2\ell+1}}\leq \frac{\epsilon}{r_i^2} \frac{(1+\frac{\epsilon}{r_i})^{2\ell}}{(1+\frac{\epsilon}{r_i})^{2\ell+1}}\leq \frac{\epsilon}{r_i^2}.
					\end{align*}
					
					We conclude that
					\begin{align*}
					\Vert \nabla \mathcal{E}_{i, \mathcal{H}}^{\rm ext}\lambda_i\Vert^2_{L^2\big(\Omega^+ \cap B_{r_i + \epsilon}(\bold{x}_i)\big)}&\leq \epsilon \sum_{{\ell}=0}^{\infty}\sum_{m=-\ell}^{\ell}(\ell+1)\left([\lambda_i]_{\ell}^m\right)^2 \leq 2\epsilon\max\left\{\frac{1}{r_i}, \frac{1}{r_i^2}\right\} ||| \lambda_i |||^2_{H^{\frac{1}{2}}(\partial \Omega_i)}.
					\end{align*}
					
					Consequently, we can define a constant $C_{r_i, \epsilon}>0$ depending only on $\epsilon$ and $r_i$ as
					\begin{align}\label{eq:constant}
					C_{r_i, \epsilon}:= \epsilon \max \left\{\frac{1}{r_i}, \frac{1}{r_i^2}\right\} \max\{r_i^2+\epsilon r_i+ \frac{\epsilon^2}{3}, 2\},
					\end{align}
					
					and we obtain that
					\begin{align}\nonumber
					\Vert \mathcal{E}^i_{\rm external}\lambda_i\Vert^2_{H^1(\Omega^+)}  &\leq \max\left\{2, 1+\frac{9}{2\epsilon^2}\right\} \Vert \mathcal{E}_{i, \mathcal{H}}^{\rm ext}\lambda_i\Vert^2_{H^1\big(\Omega^+ \cap B_{r_i + \epsilon}(\bold{x}_i)\big)}\\ \label{eq:Hassan2}
					&\leq  \max\left\{2, 1+\frac{9}{2\epsilon^2}\right\} C_{r_i, \epsilon} ||| \lambda_i|||^2_{H^{\frac{1}{2}}(\partial \Omega_{i})}.
					\end{align}
					%
					%
					It follows that the local extension operator $\mathcal{E}^i_{\rm external} \colon H^{\frac{1}{2}}(\partial \Omega_i) \rightarrow H^1(\Omega^+)$ is indeed bounded. \\
					
					\item[Step 4)]	We are now ready to define the extension operator $\mathcal{E}_{\rm external}\colon H^{\frac{1}{2}}(\partial\Omega)  \rightarrow H^1(\Omega^+)$. Indeed, given $\lambda \in H^{\frac{1}{2}}(\partial \Omega)$ and denoting $\lambda_i:= \lambda \vert_{\partial \Omega_i}$ for each $i\in\{1, \ldots, N\}$, we define:
					\begin{align*}
					\mathcal{E}_{\rm external}(\lambda):= \sum_{i=1}^N \mathcal{E}^i_{\rm external}\lambda_i.
					\end{align*}
					
					Property 2 of the local extension operators $\mathcal{E}^i_{\rm external}, ~i=1, \ldots, N$ now yields that $\gamma^+ \big(\mathcal{E}_{\rm external}(\lambda)\big)=~\lambda$. Moreover, from the bound \eqref{eq:Hassan2} we see that
					\begin{align*}
					\Vert \mathcal{E}_{\rm external}(\lambda)\Vert^2_{H^1(\Omega^+)}  &\leq \max\left\{2, 1+\frac{9}{\epsilon^2}\right\}\max_{i=1,\ldots, N}C_{r_i,\epsilon}\sum_{i=1}^N ||| \lambda_i |||^2_{H^{\frac{1}{2}}(\partial \Omega_i)}\\
					&=\max\left\{2, 1+\frac{9}{\epsilon^2}\right\}\max_{i=1,\ldots, N}C_{r_i,\epsilon}||| \lambda |||^2.
					\end{align*}
					
					Thus, the mapping $\mathcal{E}_{\rm external}\colon H^{\frac{1}{2}}(\partial\Omega)  \rightarrow H^1(\Omega^+)$ is indeed a bounded extension operator with operator norm 
					\[\Vert \mathcal{E}_{\rm external}\Vert^2_{\rm OP}:=\max\left\{2, 1+\frac{9}{\epsilon^2}\right\}\max_{i=1,\ldots, N}C_{r_i,\epsilon},\]
					
					Notice that the operator norm is independent of the number $N$ of open balls and depends only the radii of the open balls $\{\Omega_{i}\}_{i=1}^N$ and the minimal inter-sphere separation distance $\epsilon$. Furthermore, it follows from Equation \eqref{eq:constant} that $\max_{i=1,\ldots, N} C_{r_i, \epsilon } = \mathcal{O}(\epsilon)$ as $\epsilon \to 0$. Consequently, we obtain that $\Vert \mathcal{E}_{\rm external}\Vert^2_{\rm OP} = \mathcal{O}\big(\frac{1}{\epsilon}\big)$ as $\epsilon \to 0$.\\
				\end{enumerate}
				
				Using the extension operator $\mathcal{E}_{\rm external}$ we have just defined, we can mimic the calculations performed in the beginning of this proof in order to obtain the second, required bound:
				\begin{align*}
				\Vert \gamma_N^+w\Vert_{H^{-\frac{1}{2}}(\partial \Omega)}\leq c_{\rm equiv}\Vert \mathcal{E}_{\rm external}\Vert_{\text{OP}}\Vert  \nabla w\Vert_{L^2(\Omega^+)}.
				\end{align*}
				Here, the $N$-independent norm equivalence constant $c_{\rm equiv}$ arises once again due to the fact that the canonical dual norm $\Vert \cdot \Vert_{H^{-\frac{1}{2}}(\partial \Omega)}$ is defined with respect to the Sobolev-Slobodeckij norm $\Vert \cdot \Vert_{H^{\frac{1}{2}}(\partial \Omega)}$ rather than the new $||| \cdot |||$ given by Definition \ref{def:NewNorm}. Defining $c_{\rm int}:= c_{\rm equiv}\Vert \mathcal{E}_{\mathcal{H}}^{\rm int}\Vert_{\rm OP}$ and $c_{\rm ext}:= c_{\rm equiv}\Vert \mathcal{E}_{\rm external} \Vert_{\rm OP}$ thus completes the proof.
			\end{proof}
			
			We can now deduce a lower bound for the coercivity constant $c_{\mathcal{V}}$ of the single layer boundary operator. 
			
			\begin{lemma}\label{lem:Hassan2}
				Let the constants $c_{\rm int}>0$ and $c_{\rm ext}>0$ be defined as in Lemma \ref{lem:Single} and let $c_{\mathcal{V}}>0$ denote the coercivity constant of the single layer boundary operator $\mathcal{V}\colon H^{\frac{1}{2}}(\partial \Omega) \rightarrow H^{\frac{1}{2}}(\partial \Omega)$. Then it holds that
				\begin{align*}
				c_{\mathcal{V}} \geq \frac{1}{2} \min \left\{\frac{1}{c^2_{\rm int}}, \frac{1}{c^2_{\rm ext}}\right\}.
				\end{align*}
			\end{lemma}
			\begin{proof}
				Let $\sigma \in H^{-\frac{1}{2}}(\partial \Omega)$ and let $u = \mathcal{S}\sigma \in H^1\left(\Omega^- \cup \Omega^+\right)$. It follows from the jump properties of the single layer potential operator that
				\begin{align*}
				\langle \sigma, \mathcal{V}\sigma\rangle_{\partial \Omega}&= \int_{\Omega^-} \vert \nabla u(x) \vert^2\, dx + \int_{\Omega^{\text{+}}} \vert \nabla u(x) \vert^2\, dx.
				\end{align*}

				Lemma \ref{lem:Single} therefore yields that
				\begin{align*}
				\langle \sigma, \mathcal{V}\sigma\rangle_{\partial \Omega}&\geq \frac{1}{c_{\text{int}}^2}\Vert \gamma^-_N u\Vert^2_{H^{-\frac{1}{2}}(\partial \Omega)}+ \frac{1}{c_{\text{ext}}^2}\Vert \gamma^+_N u\Vert^2_{H^{-\frac{1}{2}}(\partial \Omega)}\\
				&\geq \min\Big\{\frac{1}{c^2_{\text{int}}}, \frac{1}{c_{\text{ext}}^2}\Big\} \left(\Vert \gamma^-_N u\Vert^2_{H^{-\frac{1}{2}}(\partial \Omega)}+ \Vert \gamma^+_N u\Vert^2_{H^{-\frac{1}{2}}(\partial \Omega)}\right)\\
				&\geq \min\Big\{\frac{1}{c^2_{\text{int}}}, \frac{1}{c_{\text{ext}}^2}\Big\} \left(\frac{1}{2} \Vert \gamma^-_Nu - \gamma^+_N u\Vert^2_{H^{-\frac{1}{2}}(\partial \Omega)}\right)\\
				&=\frac{1}{2}\min\Big\{\frac{1}{c^2_{\text{int}}}, \frac{1}{c_{\text{ext}}^2}\Big\} \Vert \sigma\Vert^2_{H^{-\frac{1}{2}}(\partial \Omega)}.
				\end{align*}
			\end{proof}
			
			\begin{remark}\label{rem:Hassan}
				Consider the settings of Lemma \ref{lem:Single} and Lemma \ref{lem:Hassan2}. {Two facts can be deduced from the proofs of these results. First, that the coercivity constant $c_{\mathcal{V}}$ of the single layer boundary operator depends only on the radii $\{r_j\}_{j=1}^N$ of the open balls $\{\Omega_j\}_{j=1}^N$ and the minimal inter-sphere separation distance. As a consequence, the continuity constant $C_{\widetilde{\mathcal{A}}}$ of the reduced bilinear form $\tilde{a}$ (see Lemma~\ref{lem:contin2}) depends only on the radii of the open balls, the minimal inter-sphere separation distance, and the dielectric constants $\{\kappa_j\}_{j=1}^N$.} {Second, we have also obtained significant insight into the behaviour of the coercivity constant $c_{\mathcal{V}}$ for small minimal inter-sphere separation distance. Indeed, let $\epsilon := \min_{\substack{i, j \in \{1, \ldots, N\}\\ i \neq j}} \big(\vert \bold{x}_i - \bold{x}_j \vert -r_i-r_j\big)$. Then $c_{\mathcal{V}}=\mathcal{O}(\epsilon)$ for $\epsilon\to 0$. This result implies that the continuity constant $C_{\widetilde{\mathcal{A}}}$ grows with rate at most $\mathcal{O}(\frac{1}{\sqrt{\epsilon } })$ as $\epsilon \to 0$.}
			\end{remark}	
		}

		Now that we have analysed the continuity constant $C_{\widetilde{\mathcal{A}}}$ of the reduced bilinear form $\tilde{a} \colon \breve{H}^{\frac{1}{2}}(\partial \Omega) \times \breve{H}^{-\frac{1}{2}}(\partial \Omega)\rightarrow~\mathbb{R}$ in detail, the next step in our analysis is to prove that this bilinear form satisfies the inf-sup condition.
		
		\begin{lemma}\label{lem:inf-sup1}
			Let the bilinear form $\tilde{a} \colon \breve{H}^{\frac{1}{2}}(\partial \Omega) \times \breve{H}^{-\frac{1}{2}}(\partial \Omega)\rightarrow \mathbb{R}$ be defined as in Definition \ref{def:Atilde}. Then there exists a constant $\beta_{\tilde{\mathcal{A}}}> 0$ that depends only on the function $\kappa$ and the dielectric constant $\kappa_0>0$ of the external medium such that 
			
			\begin{enumerate}
				\item[(i)] It holds that
				\begin{align*}
				\inf_{0 \neq\tilde{\lambda} \in \breve{H}^{\frac{1}{2}}(\partial \Omega)} \sup_{0\neq \tilde{\sigma} \in \breve{H}^{-\frac{1}{2}}(\partial \Omega)} \frac{\vert\tilde{a}(\tilde{\lambda}, \tilde{\sigma})\vert}{ ||| \tilde{\lambda}||| \, ||| \tilde{\sigma}|||^* } \geq \beta_{\tilde{\mathcal{A}}} > 0; \qquad \text{\emph{(Bounded Below)}}
				\end{align*}
				
				\item[(ii)] For all $0\neq \tilde{\sigma} \in \breve{H}^{-\frac{1}{2}}(\partial \Omega)$ it holds that
				\begin{align*}
				\hspace{20mm}\sup_{0\neq \tilde{\lambda} \in \breve{H}^{\frac{1}{2}}(\partial \Omega)} \vert \tilde{a}(\tilde{\lambda}, \tilde{\sigma})\vert > 0. \hspace{10mm}\qquad \text{ \emph{(Dense Range)}}
				\end{align*}
			\end{enumerate}
		\end{lemma}
		
		\begin{proof}
			The proof relies on the fact that the Dirichlet-to-Neumann map $\text{DtN} \colon \breve{H}^{\frac{1}{2}}(\partial \Omega)  \rightarrow \breve{H}^{-\frac{1}{2}}(\partial \Omega)$ is an isomorphism. We first prove Property \textit{(i)}. To this end, let $\widehat{\lambda} \in \breve{H}^{\frac{1}{2}}(\partial \Omega)$ be arbitrary. We decompose $\widehat{\lambda}$ as the sum of two functions as follows:
			\begin{align*}
			\widehat{\lambda}=\widehat{\lambda}_{+} + \widehat{\lambda}_{-}.
			\end{align*}
			
			Here, $\widehat{\lambda}_{+} \in \breve{H}^{\frac{1}{2}}(\partial \Omega)$ is a function equal to $\widehat{\lambda}$ on all spheres $\partial \Omega_{i}, ~i \in \{1, \ldots, N\}$ such that $\kappa_i - \kappa_0 > 0$ and zero otherwise. Similarly, $\widehat{\lambda}_{-}\in \breve{H}^{\frac{1}{2}}(\partial \Omega)$ is a function equal to $\widehat{\lambda}$ on all spheres $\partial \Omega_{i}, ~i \in \{1, \ldots, N\}$ such that $\kappa_i - \kappa_0 < 0$ and zero otherwise. We recall that we have assumed that $\kappa \neq \kappa_0$ as mentioned in Remark \ref{rem:trivial}.

			We now define a corresponding test function $\widehat{\sigma} \in \breve{H}^{\frac{1}{2}}(\partial \Omega)$ by setting
			\begin{align*}
			\widehat{\sigma}:=\frac{\kappa-\kappa_0}{\kappa_0}\text{DtN}\widehat{\lambda}_{+} - \frac{\kappa-\kappa_0}{\kappa_0}\text{DtN}\widehat{\lambda}_{-}.
			\end{align*}
			
			For notational convenience, we define sets of indices $N_+\subset \mathbb{N}$ and $N_-\subset \mathbb{N}$ such that $i \in N_+ \iff \kappa_i - \kappa_0 > 0$ and $i \in N_- \iff \kappa_i - \kappa_0< 0$. { Moreover, for all $j =1, \ldots, N$ we define
				\begin{align*}
				\widehat{\lambda}_j&:=\begin{cases}
				\widehat{\lambda} \quad &\text{on } \partial \Omega_j,\\
				0 \quad &\text{otherwise, }
				\end{cases}\\
				\widehat{\sigma}_j&:=\begin{cases}
				\widehat{\sigma} \quad &\text{on } \partial \Omega_j,\\
				0 \quad &\text{otherwise, }
				\end{cases}
				\end{align*}}
			
			It follows that the reduced bilinear form $\tilde{a}$ satisfies
			\begin{align*}
			\tilde{a}(\widehat{\lambda}, \widehat{\sigma})&= \sum_{j \in N_+}\frac{\kappa_j-\kappa_0}{\kappa_0} ||| \widehat{\lambda}_j|||^2 + \sum_{j \in N_-}\frac{\kappa_0-\kappa_j}{\kappa_0} ||| \widehat{\lambda}_j|||^2  + \underbrace{\left\langle \widehat{\sigma}, \mathcal{V}\text{DtN}\left(\frac{\kappa-\kappa_0}{\kappa_0}\widehat{\lambda}\right)\right\rangle_{\partial \Omega }}_{:=J}.
			\end{align*}
			
			Note that due to our choice of test function $\widehat{\sigma}$, the coefficients of all terms in the above two sums are positive. Therefore, let us focus on analysing the term $J$. Using the decomposition we have introduced, we obtain that
			\begin{align*}
			J=\left\langle \widehat{\sigma}, \mathcal{V}\text{DtN}\left(\frac{\kappa-\kappa_0}{\kappa_0}\widehat{\lambda}\right)\right\rangle_{\partial \Omega}&= \left\langle \text{DtN}\left(\frac{\kappa-\kappa_0}{\kappa_0}\widehat{\lambda}_+\right), \mathcal{V}\text{DtN}\left(\frac{\kappa-\kappa_0}{\kappa_0}\widehat{\lambda}_+\right)\right\rangle_{\partial \Omega}\\
			&-  \left\langle \text{DtN}\left(\frac{\kappa-\kappa_0}{\kappa_0}\widehat{\lambda}_-\right), \mathcal{V}\text{DtN}\left(\frac{\kappa-\kappa_0}{\kappa_0}\widehat{\lambda}_-\right)\right\rangle_{\partial \Omega}.
			\end{align*}
			
			Using the Calderon identities (see, e.g., \cite[Theorem 3.8.7]{Schwab}), we further obtain that 
			\begin{align*}
			-  \left\langle \text{DtN}\left(\frac{\kappa-\kappa_0}{\kappa_0}\widehat{\lambda}_-\right), \mathcal{V}\text{DtN}\left(\frac{\kappa-\kappa_0}{\kappa_0}\widehat{\lambda}_-\right)\right\rangle_{\partial \Omega} =& -  \left\langle \text{DtN}\left(\frac{\kappa-\kappa_0}{\kappa_0}\widehat{\lambda}_-\right), \left(\frac{\kappa-\kappa_0}{\kappa_0}\widehat{\lambda}_-\right)\right\rangle_{\partial \Omega}\\
			&+ \left\langle  \mathcal{W}\left(\frac{\kappa-\kappa_0}{\kappa_0}\widehat{\lambda}_-\right), \left(\frac{\kappa-\kappa_0}{\kappa_0}\widehat{\lambda}_-\right)\right\rangle_{\partial \Omega}.
			\end{align*}
			
			The non-negativity of the hypersingular operator $\mathcal{W} \colon H^{\frac{1}{2}}(\partial \Omega) \rightarrow H^{-\frac{1}{2}}(\partial \Omega)$ (see Property 2 of Section \ref{sec:2a}) thus implies that
			\begin{align*}
			J \geq -  \left\langle \text{DtN}\left(\frac{\kappa-\kappa_0}{\kappa_0}\widehat{\lambda}_-\right), \left(\frac{\kappa-\kappa_0}{\kappa_0}\widehat{\lambda}_-\right)\right\rangle_{\partial \Omega} = -\sum_{j \in N_{-}} \Big( \frac{\kappa_j-\kappa_0}{\kappa_0}\Big)^2 ||| \widehat{\lambda}_j |||^2.
			\end{align*}

			Consequently, we obtain that
			\begin{align*}
			\tilde{a}(\widehat{\lambda}, \widehat{\sigma})&\geq \sum_{j \in N_+}\frac{\kappa_j-\kappa_0}{\kappa_0} ||| \widehat{\lambda}_j|||^2 + \sum_{j \in N_-}\frac{\kappa_0-\kappa_j}{\kappa_0}||| \widehat{\lambda}_j|||^2 -\sum_{j \in N_-}\Big(\frac{\kappa_0-\kappa_j}{\kappa_0}\Big)^2||| \widehat{\lambda}_j|||^2\\
			&= \sum_{j \in N_+}\frac{\kappa_j-\kappa_0}{\kappa_0} ||| \widehat{\lambda}_j|||^2 - \sum_{j \in N_-}\frac{\kappa_j}{\kappa_0}\frac{\kappa_j-\kappa_0}{\kappa_0} ||| \widehat{\lambda}_j|||^2\\
			&\geq \min\left\{ \min_{j \in N_+}\frac{\kappa_j-\kappa_0}{\kappa_0},~ \min_{j \in N_-} \frac{\kappa_j}{\kappa_0}\frac{\kappa_0-\kappa_j}{\kappa_0}\right\} ||| \widehat{\lambda}|||^2.
			\end{align*}
			
			Furthermore, using Remark \ref{rem:isometry} we obtain that the norm of the test function $\widehat{\sigma}$ is given by
			\begin{align*}
			{||| \widehat{\sigma}|||^*} &= {\Big|\Big|\Big| \frac{\kappa-\kappa_0}{\kappa_0}\text{DtN}\widehat{\lambda}_{+} - \frac{\kappa-\kappa_0}{\kappa_0}\text{DtN}\widehat{\lambda}_{-} \Big|\Big|\Big|^*}\\
			&=\Big|\Big|\Big| \frac{\kappa-\kappa_0}{\kappa_0}\widehat{\lambda}_{+} - \frac{\kappa-\kappa_0}{\kappa_0}\widehat{\lambda}_{-} \Big|\Big|\Big|\\
			&\leq \max_{j=1, \ldots, N}  \Big \vert \frac{\kappa_j-\kappa_0}{\kappa_0}\Big \vert||| \widehat{\lambda}|||.
			\end{align*}
			
			We therefore define the constant $\beta_{\tilde{\mathcal{A}}} > 0$ as
			\begin{equation}\label{eq:coercive}
			\beta_{\tilde{\mathcal{A}}}:= \frac{\min\left\{ \min_{j \in N_+}\frac{\kappa_j-\kappa_0}{\kappa_0},~ \min_{j \in N_-} \frac{\kappa_j}{\kappa_0}\frac{\kappa_0-\kappa_j}{\kappa_0} \right\}}{ \max_{j=1, \ldots, N}  \Big \vert \frac{\kappa_j-\kappa_0}{\kappa_0}\Big \vert}.
			\end{equation}
			
			We then obtain that
			\begin{align*}
			\inf_{0 \neq\tilde{\lambda} \in \breve{H}^{\frac{1}{2}}(\partial \Omega)} \sup_{0\neq \tilde{\sigma} \in \breve{H}^{-\frac{1}{2}}(\partial \Omega)} \frac{\vert\tilde{a}(\tilde{\lambda}, \tilde{\sigma})\vert}{ ||| \tilde{\lambda}||| \, ||| \tilde{\sigma}|||^* } \geq  \beta_{\tilde{\mathcal{A}}},
			\end{align*}
			which completes the proof of Property \textit{(i)}.
			
			Let us now turn to the proof of Property \textit{(ii)}. Let $0\neq \widehat{\sigma} \in \breve{H}^{-\frac{1}{2}}(\partial \Omega)$ be arbitrary and let $\text{NtD} \colon \breve{H}^{-\frac{1}{2}}(\partial \Omega) \rightarrow \breve{H}^{\frac{1}{2}}(\partial \Omega)$ be the inverse of the Dirichlet-to-Neumann map. Using the decomposition and notation developed above, it is possible to define a corresponding function $\widehat{\lambda} \in \breve{H}^{\frac{1}{2}}(\partial \Omega)$ as
			\begin{align*}
			\widehat{\lambda}:=\sum_{j \in N_+}\frac{\kappa_0}{\kappa_j-\kappa_0} \text{NtD}\widehat{\sigma}_j - \sum_{j \in N_-}\frac{\kappa_0}{\kappa_j-\kappa_0} \text{NtD}\widehat{\sigma}_j.
			\end{align*}
			
			With this choice of $\widehat{\lambda}$, we immediately obtain that
			\begin{align*}
			\widehat{\sigma}=\sum_{j \in N_+}\frac{\kappa_j-\kappa_0}{\kappa_0} \text{DtN}\widehat{\lambda}_j - \sum_{j \in N_-}\frac{\kappa_j-\kappa_0}{\kappa_0} \text{DtN}\widehat{\lambda}_j.
			\end{align*}
			
			Therefore, a similar calculation to the one used to prove Property \textit{(i)} reveals that
			\begin{align*}
			\vert \tilde{a}(\widehat{\lambda}, \widehat{\sigma})\vert &\geq \min\left\{ \min_{j \in N_+}\frac{\kappa_j-\kappa_0}{\kappa_0},~ \min_{j \in N_-} \frac{\kappa_j}{\kappa_0}\frac{\kappa_0-\kappa_j}{\kappa_0}\right\} ||| \widehat{\lambda}|||^2\\
			&\geq  \frac{\beta_{\tilde{\mathcal{A}}}}{\max_{j=1, \ldots, N}  \Big \vert \frac{\kappa_j-\kappa_0}{\kappa_0}\Big \vert} \big(||| \widehat{\sigma} |||^*\big)^2.
			\end{align*}
			
			We conclude that for all $0\neq \tilde{\sigma} \in \breve{H}^{-\frac{1}{2}}(\partial \Omega)$ it holds that
			\begin{align*}
			\sup_{0\neq \tilde{\lambda} \in \breve{H}^{\frac{1}{2}}(\partial \Omega)} \vert \tilde{a}(\tilde{\lambda}, \tilde{\sigma})\vert > 0.
			\end{align*}
		\end{proof}

		An immediate consequence of Lemma \ref{lem:inf-sup1} is that both the modified weak formulation \eqref{eq:weak2a}-\eqref{eq:weak2b} and the weak formulation \eqref{eq:weak1} are well-posed. \vspace{2mm}
		
		\subsubsection{\textbf{The New Analysis of the Discrete Problem}}~
		Our next goal is to prove that the Galerkin discretisation \eqref{eq:Galerkin} is also well-posed with a stability constant that is independent of the number of open balls $N$. Similar to the infinite-dimensional case, we adopt an indirect approach, and reformulate Equation \eqref{eq:Galerkin} as a modified Galerkin discretisation using the projection operators $\mathbb{P}_0$ and $\mathbb{P}_0^\perp$ introduced through Lemma \ref{lem:decomp}. We first define the relevant approximation space.

		\begin{definition}[Reduced Global Approximation Space]\label{def:Appromxation2}
			Let $\ell_{\max} \in \mathbb{N}$. We define the finite-dimensional Hilbert space $W_0^{\ell_{\max}} \subset \breve{H}^{\frac{1}{2}}(\partial \Omega)$ as the set
			\begin{align*}
			W_0^{\ell_{\max}} := \Big\{u \in W^{\ell_{\max}}(\partial \Omega)\colon \mathbb{P}_0 u = 0\Big\},
			\end{align*}
			equipped with the $(\cdot, \cdot)_{W^{\ell_{\max}}}$ inner product.
		\end{definition}
		
		\begin{remark}
			Using the fact that the spherical harmonics functions are smooth, we can immediately infer that the finite-dimensional Hilbert spaces $W_0^{\ell_{\max}} \subset W^{\ell_{\max}} \subset \breve{H}^{\frac{1}{2}}(\partial \Omega)$ also satisfy
			\begin{align*}
			W_0^{\ell_{\max}} \subset W^{\ell_{\max}} \subset \breve{H}^{-\frac{1}{2}}(\partial \Omega) \quad \text{ and } \quad 
			\forall \lambda_{\ell_{\max}} \in W_0^{\ell_{\max}} \colon ~\Vert \lambda_{\ell_{\max}}\Vert^2_{W^{\ell_{\max}}}= ||| \lambda_{\ell_{\max}}|||^2.
			\end{align*}
			
			Note that if one wishes to view $W^{\ell_{\max}}$ and $W^{\ell_{\max}}_0$ as subspaces of $H^{-\frac{1}{2}}(\partial \Omega)$, then the definition of the equipped norms would have to be modified accordingly. \vspace{1cm}
		\end{remark}

		\noindent {\textbf{Modified Galerkin Discretisation of the Integral Equation \eqref{eq:3.3}}~
			
			Let $\sigma_f \in H^{-\frac{1}{2}}(\partial \Omega)$. Find functions $(\lambda_0, \lambda_{\ell_{\max}}) \in \mathcal{C}(\partial \Omega) \times W_0^{\ell_{\max}}$ such that for all test functions $(\sigma_0, \sigma_{\ell_{\max}}) \in \mathcal{C}(\partial \Omega) \times W_0^{\ell_{\max}}$ it holds that
			\begin{align}\label{eq:Galerkin2a}
			\left( \sigma_0, \lambda_0\right)_{L^2(\partial \Omega)}- \left(\sigma_0, \mathcal{V}\text{DtN}\Big(\frac{\kappa_0-\kappa}{\kappa_0}{\lambda_{\ell_{\max}}}\Big)\right)_{L^2(\partial \Omega)}&= \frac{4\pi}{\kappa_0}\left ( \sigma_0, \mathcal{V}\sigma_f \right) _{L^2(\partial \Omega)},\\[0.5em]
			\left({\sigma}_{\ell_{\max}}, {\lambda}_{\ell_{\max}}\right)_{L^2(\partial \Omega)}- \left( {\sigma}_{\ell_{\max}}, \mathcal{V}\text{DtN}\Big(\frac{\kappa_0-\kappa}{\kappa_0}{\lambda}_{\ell_{\max}}\Big)\right)_{L^2(\partial \Omega)}&= \frac{4\pi}{\kappa_0}\left ({\sigma}_{\ell_{\max}}, \mathcal{V}\sigma_f \right) _{L^2(\partial \Omega)}. \label{eq:Galerkin2b}
			\end{align}
			
			It is a simple exercise to prove that the modified Galerkin discretisation \eqref{eq:Galerkin2a}-\eqref{eq:Galerkin2b} is indeed equivalent to the Galerkin discretisation \eqref{eq:Galerkin}.
			
			The structure of the Galerkin discretisation \eqref{eq:Galerkin2a}-\eqref{eq:Galerkin2b} is very similar to the structure of the infinite-dimensional modified weak formulation \eqref{eq:weak2a}-\eqref{eq:weak2b}. Indeed, we observe once again that Equation \eqref{eq:Galerkin2b} involves only the unknown function $\lambda_{\ell_{\max}} \in W_0^{\ell_{\max}}$. It is therefore clear that if Equation \eqref{eq:Galerkin2b} is uniquely solvable, then Equation \eqref{eq:Galerkin2a} is also uniquely solvable, and hence the Galerkin discretisation \eqref{eq:Galerkin} is well-posed. Moreover, thanks to the analysis carried out for the infinite-dimensional Equation \eqref{eq:weak2b}, well-posedness of the finite-dimensional equation \eqref{eq:Galerkin2b} follows almost immediately. Indeed, we have the following result.
			
			\begin{lemma}\label{lem:inf-sup2}
				Let the bilinear form $\tilde{a} \colon \breve{H}^{\frac{1}{2}}(\partial \Omega) \times \breve{H}^{-\frac{1}{2}}(\partial \Omega)\rightarrow \mathbb{R}$ be defined as in Definition \ref{def:Atilde}, and let the constant $\beta_{\tilde{\mathcal{A}}}> 0$ be defined through Equation \eqref{eq:coercive} as in the proof of Lemma \ref{lem:inf-sup1}. Then it holds that
				\begin{align*}
				\inf_{0 \neq {\lambda}_{\ell_{\max}} \in W_0^{\ell_{\max}} } \sup_{0\neq {\sigma}_{\ell_{\max}} \in W_0^{\ell_{\max}}} \frac{\vert\tilde{a}({\lambda}_{\ell_{\max}}, {\sigma}_{\ell_{\max}})\vert}{ |||{\lambda}_{\ell_{\max}}||| \, |||{\sigma}_{\ell_{\max}}|||^* } \geq \beta_{\tilde{\mathcal{A}}} > 0. \qquad \text{\emph{(Discrete inf-sup Condition)}}
				\end{align*}
			\end{lemma}
			\begin{proof}
				The proof uses the fact that the Dirichlet-to-Neumann operator $\text{DtN} \colon W_0^{\ell_{\max}}  \rightarrow W_0^{\ell_{\max}} $ is an isomorphism. Indeed, consider $\lambda_j \in W_0^{\ell_{\max}}(\partial \Omega_j)$ given by
				\begin{align*}
				\lambda_j(\bold{x})= \sum_{{\ell}=1}^{\ell_{\max}} \sum_{m=-\ell}^{m=+\ell} [\lambda_j]_{\ell}^m \mathcal{Y}_{\ell}^m\left(\frac{\bold{x}-\bold{x}_j}{\vert \bold{x}-\bold{x}_j\vert}\right).
				\end{align*}
				
				Then the function $\text{DtN} \lambda_j  \in W_0^{\ell_{\max}}$ is given by
				\begin{align*}
				\text{DtN}\lambda_j(\bold{x})= \sum_{{\ell}=1}^{\ell_{\max}} \sum_{m=-\ell}^{m=+\ell} \frac{\ell}{r_j}[\lambda_j]_{\ell}^{m} \mathcal{Y}_{\ell}^m\left(\frac{\bold{x}-\bold{x}_j}{\vert \bold{x}-\bold{x}_j\vert}\right).
				\end{align*}

				Consequently given any arbitrary function $\widehat{\lambda} \in W_0^{\ell_{\max}} \subset \breve{H}^{\frac{1}{2}}(\partial \Omega)$, we may pick as the test function $\widehat{\sigma} \in W_0^{\ell_{\max}} \subset \breve{H}^{-\frac{1}{2}}(\partial \Omega)$ given by
				\begin{align*}
				\widehat{\sigma}=\frac{\kappa-\kappa_0}{\kappa_0}\text{DtN}\widehat{\lambda}_{+} - \frac{\kappa-\kappa_0}{\kappa_0}\text{DtN}\widehat{\lambda}_{-},
				\end{align*}
				where we have used the decomposition $\widehat{\lambda}= \widehat{\lambda}_{+} + \widehat{\lambda}_{-}$ introduced in the proof of Lemma \ref{lem:inf-sup1}. The remainder of the proof is now identical to the proof of Lemma \ref{lem:inf-sup1} and yields the discrete inf-sup constant $\beta_{\tilde{\mathcal{A}}}$ defined through Equation \eqref{eq:coercive}.
			\end{proof}
			
			Lemma \ref{lem:inf-sup2} now has several important consequences: \vspace{3mm}
			
			\begin{enumerate}
				\item Both the modified Galerkin discretisation \eqref{eq:Galerkin2a}-\eqref{eq:Galerkin2b} and the Galerkin discretisation \eqref{eq:Galerkin} are well-posed. \vspace{0.2mm}
				
				\item For every choice of the approximation parameter $\ell_{\max} \in \mathbb{N}$, the finite-dimensional solution to the Galerkin discretisation \eqref{eq:Galerkin} satisfies a standard quasi-optimality result. \vspace{2mm}
				
				\item Since the discrete inf-sup constant $\beta_{\tilde{\mathcal{A}}}$ is independent of the approximation space, we obtain stability and convergence to the exact solution of the approximate solutions as the approximation parameter $\ell_{\max} \to \infty$. \vspace{2mm}
			\end{enumerate}
			
			All of the above results can be proven using text-book functional analysis techniques. We state one particular quasi-optimality result concerning solutions to the finite-dimensional equation \eqref{eq:Galerkin2b} which will be of use in the next subsection. 
			
			\begin{lemma}[Partial Quasi-Optimality]\label{lem:quasi1}
				Let $C_{\tilde{\mathcal{A}}}> 0$ be the continuity constant defined through Equation \eqref{eq:contin} in Lemma \ref{lem:contin2}, let $\beta_{\tilde{\mathcal{A}}}>0$ be the inf-sup constant defined through Equation \eqref{eq:coercive} in Lemma \ref{lem:inf-sup1}, let $\sigma_f \in {H}^{-\frac{1}{2}}(\partial \Omega)$, let $\ell_{\max} \in \mathbb{N}$, let $\lambda_{\ell_{\max}} \in W_0^{\ell_{\max}}$ be the unique solution to the finite-dimensional Equation \eqref{eq:Galerkin2b} with right hand side given by $\sigma_f$, and let $\tilde{\lambda} \in \breve{H}^{\frac{1}{2}}(\partial \Omega)$ be the unique solution to infinite-dimensional Equation \eqref{eq:weak2b} with right hand side given by $\sigma_f$. Then it holds that
				\begin{align}\label{eq:quasi1}
				|||\tilde{\lambda}-\lambda_{\ell_{\max}}||| \leq \left(1 + \frac{C_{\tilde{\mathcal{A}}}}{\beta_{\tilde{\mathcal{A}}}}\right) \inf_{\psi \in W_0^{\ell_{\max}}} ||| \tilde{\lambda}-\psi|||.
				\end{align}
			\end{lemma}
			\begin{proof}
				The proof is also text-book functional analysis.
			\end{proof}
			
			Notice that thus far we have only proved well-posedness of the infinite-dimensional weak formulation \eqref{eq:weak1} and the Galerkin discretisation \eqref{eq:Galerkin} involving the \emph{surface electrostatic potential}. However, the main results in Section \ref{sec:2} have been formulated for the \emph{induced surface charge}. Therefore, the next step in our analysis will be to transfer our existing results to the infinite-dimensional weak formulation \eqref{eq:weak1a} and the Galerkin discretisation \eqref{eq:Galerkina} involving the exact and approximate induced surface charge.
			
			\subsection{Well-Posedness Analysis for the Induced Surface Charge}~
			As the astute reader may already have realised, the well-posedness analysis for the infinite-dimensional weak formulation \eqref{eq:weak1a} and the Galerkin discretisation \eqref{eq:Galerkina} is exceedingly simple because the underlying boundary integral operator is simply $\mathcal{A}^*$, i.e., the adjoint of the boundary integral operator $\mathcal{A}$, which has already been completely analysed in both the infinite-dimensional and finite dimensional setting. To facilitate the subsequent exposition, we introduce some additional notation.\vspace{3mm}
			
			{	\begin{definition}[Finite-Dimensional Projection Operators]\label{def:fin_proj}~
					Let $\ell_{\max} \in \mathbb{N}$. We define the projection operator $\mathbb{P}_{\ell_{\max}}\colon H^{\frac{1}{2}}(\partial \Omega) \rightarrow W^{\ell_{\max}}$ as the mapping with the property that for any $\psi \in H^{\frac{1}{2}}(\partial \Omega)$,  $\mathbb{P}_{\ell_{\max}}\psi $ is the unique element of $W^{\ell_{\max}}$ satisfying
					\begin{align*}
					\left(\phi_{\ell_{\max}}, \mathbb{P}_{\ell_{\max}}\psi\right)_{L^2(\partial \Omega)}=\left\langle \phi_{\ell_{\max}}, \psi\right\rangle_{\partial \Omega} \qquad \forall \phi_{\ell_{\max}} \in W^{\ell_{\max}},
					\end{align*}
					
					Moreover, we define the projection operator $\mathbb{Q}_{\ell_{\max}}\colon H^{-\frac{1}{2}}(\partial \Omega) \rightarrow W^{\ell_{\max}}$ as the mapping with the property that for any $\sigma \in H^{-\frac{1}{2}}(\partial \Omega)$,  $\mathbb{Q}_{\ell_{\max}}\sigma $ is the unique element of $W^{\ell_{\max}}$ satisfying
					\begin{align*}
					\left(\mathbb{Q}_{\ell_{\max}}\sigma, \phi_{\ell_{\max}}\right)_{L^2(\partial \Omega)}&=\left\langle \sigma, \phi_{\ell_{\max}}\right\rangle_{\partial \Omega} \qquad \forall \phi_{\ell_{\max}} \in W^{\ell_{\max}}.
					\end{align*}
			\end{definition}}
			
			{	\begin{remark}
					Consider the setting of Definition \ref{def:fin_proj}. It is possible to show that the projection operators $\mathbb{P}_{\ell_{\max}}$ and $\mathbb{Q}_{\ell_{\max}}$ are stable, i.e., for all $\psi \in H^{\frac{1}{2}}(\partial \Omega)$ and all $\sigma \in H^{-\frac{1}{2}}(\partial \Omega)$ it holds that
					\begin{align*}
					||| \mathbb{P}_{\ell_{\max}} \psi ||| \leq ||| \psi ||| \quad \text{and} \quad ||| \mathbb{Q}_{\ell_{\max}}\sigma|||^* \leq ||| \sigma |||^*.
					\end{align*}
				\end{remark}
			}
			
			We now have the following simple result.
			
			\begin{theorem}[Infinite-Dimensional Well-Posedness]
				The infinite-dimensional weak formulation \eqref{eq:weak1a} of the boundary integral equation \eqref{eq:3.3a} is well-posed.
			\end{theorem}
			\begin{proof}
				The well-posedness of the infinite-dimensional weak formulation \eqref{eq:weak1} implies that the boundary integral operator $\mathcal{A} \colon H^{\frac{1}{2}}(\partial \Omega) \rightarrow H^{\frac{1}{2}}(\partial \Omega)$ defined through Definition \ref{def:A} is a continuous bijection. Consequently the adjoint operator $\mathcal{A}^* \colon H^{-\frac{1}{2}}(\partial \Omega) \rightarrow H^{-\frac{1}{2}}(\partial \Omega)$ is also a continuous bijection. 
			\end{proof}
			
			A similar result holds for the Galerkin discretisation of the integral equation \eqref{eq:3.3a} for the induced surface~charge.
			
			\begin{theorem}[Finite-Dimensional Well-Posedness]
				The finite-dimensional Galerkin discretisation \eqref{eq:Galerkina} of the weak formulation \eqref{eq:weak1a} is well-posed.
			\end{theorem}{
				\begin{proof}
					Let $\mathbb{P}_{{\ell_{\max}}} \colon H^{\frac{1}{2}}(\partial \Omega) \rightarrow W^{\ell_{\max}}$ and $\mathbb{Q}_{{\ell_{\max}}} \colon H^{-\frac{1}{2}}(\partial \Omega) \rightarrow W^{\ell_{\max}}$ denote the projection operators defined through Definition \ref{def:fin_proj}. The well-posedness of the finite-dimensional Galerkin discretisation \eqref{eq:Galerkin} implies that the boundary integral operator $\mathbb{P}_{{\ell_{\max}}} \mathcal{A} \mathbb{P}_{{\ell_{\max}}} \colon W^{\ell_{\max}} \rightarrow W^{\ell_{\max}}$ is a continuous bijection. Consequently, the adjoint operator $\mathbb{Q}_{{\ell_{\max}}} \mathcal{A}^* \mathbb{Q}_{{\ell_{\max}}} \colon W^{\ell_{\max}} \rightarrow W^{\ell_{\max}}$ is also a continuous bijection. 
				\end{proof}
			}
			We conclude this subsection by stating a first approximation result for the  solution $\nu_{\ell_{\max}} \in W^{\ell_{\max}}$ to the Galerkin discretisation \eqref{eq:Galerkina}. 
			{
				\begin{theorem}[First Approximability Result]\label{lem:quasi2}~
					\noindent Let $\ell_{\max} \in \mathbb{N}$, let $\mathbb{Q}_{\ell_{\max}} \colon H^{-\frac{1}{2}}(\partial \Omega) \rightarrow W^{\ell_{\max}}$ denote the projection operator defined through Definition \ref{def:fin_proj}, let $\mathbb{Q}_{\ell_{\max}}^{\perp}:= I -\mathbb{Q}_{\ell_{\max}}$ where $I$ is the identity map on $H^{-\frac{1}{2}}(\partial \Omega)$, let $C_{\tilde{\mathcal{A}}}> 0$ be the continuity constant defined through Equation \eqref{eq:contin} in Lemma \ref{lem:contin2}, let $\beta_{\tilde{\mathcal{A}}}>0$ be the inf-sup constant defined through Equation \eqref{eq:coercive} in Lemma \ref{lem:inf-sup1}, let $\sigma_f \in H^{-\frac{1}{2}}(\partial \Omega)$, let $\nu \in {H}^{-\frac{1}{2}}(\partial \Omega)$ be the unique solution to infinite-dimensional weak formulation \eqref{eq:weak1a} with right hand side given by $\sigma_f$ and let $\nu_{\ell_{\max}} \in W^{\ell_{\max}}$ be the unique solution to the finite-dimensional Galerkin discretisation \eqref{eq:Galerkina} with right hand side given by $\sigma_f$. Then it holds~that
					\begin{align}\label{eq:quasi2}
					||| \nu - \nu_{\ell_{\max}}|||^* &\leq \frac{\max \Big\vert \frac{\kappa_0-\kappa}{\kappa_0}\Big\vert}{\min\Big\vert \frac{\kappa-\kappa_0}{\kappa_0}\Big \vert} \Big(1+\frac{C_{\tilde{\mathcal{A}}}}{\beta_{\tilde{\mathcal{A}}}}\Big) \left( \Big|\Big|\Big|\mathbb{Q}_{\ell_{\max}}^{\perp}\nu \Big|\Big|\Big|^* + \frac{8\pi}{\kappa_0} \Big|\Big|\Big|\mathbb{Q}_{\ell_{\max}}^{\perp}\sigma_f\Big|\Big|\Big|^*\right).
					\end{align}
				\end{theorem}
				\begin{proof}
					
					Let $\tilde{\lambda} \in \breve{H}^{\frac{1}{2}}(\partial \Omega)$ be the solution of Equation \eqref{eq:weak2b} in the modified weak formulation.  It is straightforward to show that
					\begin{align}\label{eq:Ben1}
					\nu = \frac{\kappa_0 -\kappa}{\kappa_0} \text{DtN} \tilde{\lambda} + \frac{4 \pi}{\kappa_0} \sigma_f.
					\end{align}
					
					Next, let $\mathcal{A}$ be the integral operator defined through Definition \ref{def:A}, let $\widehat{\lambda}_{\ell_{\max}} \in W^{\ell_{\max}}$ be the solution to the Galerkin discretisation \eqref{eq:Galerkin}, and let $\mathbb{P}_{{\ell_{\max}}} \colon H^{\frac{1}{2}}(\partial \Omega) \rightarrow W^{\ell_{\max}}$ denote the projection operator defined through Definition \ref{def:fin_proj}. We then define the mappings
					\begin{align*}
					\mathcal{V}_{\ell_{\max}}&:=\mathbb{P}_{{\ell_{\max}}} \mathcal{V}\mathbb{Q}_{{\ell_{\max}}}, \qquad
					\mathcal{A}_{\ell_{\max}}:=\mathbb{P}_{{\ell_{\max}}} \mathcal{A}\mathbb{P}_{{\ell_{\max}}}, \quad \text{ and } \quad \mathcal{A}^*_{\ell_{\max}}:= \mathbb{Q}_{{\ell_{\max}}} \mathcal{A}^*\mathbb{Q}_{{\ell_{\max}}},
					\end{align*}
					
					and we define the function $\psi_{\ell_{\max}}:= \mathcal{V}_{\ell_{\max}}\nu_{\ell_{\max}} \in W^{\ell_{\max}}$. We first claim that $\psi_{\ell_{\max}}$ satisfies the equation
					\begin{align}\label{eq:star}
					\mathcal{A}_{\ell_{\max}} \psi_{\ell_{\max}}=\frac{4\pi}{\kappa_0} \mathcal{V}_{\ell_{\max}}\sigma_f.
					\end{align}

					Indeed, since $\nu_{\ell_{\max}}$ satisfies the Galerkin discretisation \eqref{eq:Galerkina} we obviously have
					\begin{align*}
					\mathcal{A}^*_{\ell_{\max}}\nu_{\ell_{\max}}= \frac{4\pi}{\kappa_0}\mathbb{Q}_{\ell_{\max}}\sigma_f \qquad \text{which implies that} \qquad 
					\mathcal{V}_{\ell_{\max}}\mathcal{A}_{\ell_{\max}}^* \nu_{\ell_{\max}}= \frac{4\pi}{\kappa_0} \mathcal{V}_{\ell_{\max}}\sigma_f.
					\end{align*}
					
					Using the fact that $\mathcal{V}_{\ell_{\max}}\mathcal{A}_{\ell_{\max}}^*= \mathcal{A}_{\ell_{\max}} \mathcal{V}_{\ell_{\max}}$ yields 
					\begin{align*}
					\mathcal{A}_{\ell_{\max}} \psi_{\ell_{\max}}=\mathcal{A}_{\ell_{\max}}\mathcal{V}_{\ell_{\max}} \nu_{\ell_{\max}}=\mathcal{V}_{\ell_{\max}}\mathcal{A}_{\ell_{\max}}^* \nu_{\ell_{\max}}=\frac{4\pi}{\kappa_0} \mathcal{V}_{\ell_{\max}}\sigma_f,
					\end{align*}
					which gives the intermediary result.
					
					We now consider again the Galerkin discretisation \eqref{eq:Galerkina}. Using the definition of $\psi_{\ell_{\max}}$ and the fact that $\mathbb{Q}_{\ell_{\max}}\text{DtN}=\text{DtN}\mathbb{P}_{\ell_{\max}}$ we obtain that
					\begin{align}\label{eq:Ben2}
					\nu_{\ell_{\max}}= \frac{\kappa_0 -\kappa}{\kappa_0} \mathbb{Q}_{\ell_{\max}}\text{DtN} \mathcal{V}\nu_{\ell_{\max}} + \frac{4 \pi}{\kappa_0}  \mathbb{Q}_{\ell_{\max}}\sigma_f=  \frac{\kappa_0 -\kappa}{\kappa_0} \text{DtN}\psi_{\ell_{\max}}+ \frac{4 \pi}{\kappa_0}  \mathbb{Q}_{\ell_{\max}}\sigma_f.
					\end{align}

					Let $\mathbb{P}_0^{\perp} \colon H^{\frac{1}{2}}(\partial \Omega) \rightarrow \breve{H}^{\frac{1}{2}}(\partial \Omega)$ be the projection operator defined through Lemma \ref{lem:decomp}. Subtracting Equation \eqref{eq:Ben2} from Equation \eqref{eq:Ben1} then gives
					\begin{align}\nonumber
					||| \nu - \nu_{\ell_{\max}}|||^* &= \Big|\Big|\Big| \frac{\kappa_0-\kappa}{\kappa_0} \text{DtN} \tilde{\lambda}- \frac{\kappa_0-\kappa}{\kappa_0} \text{DtN} {\psi}_{\ell_{\max}} + \frac{4\pi}{\kappa_0} (\sigma_f- \mathbb{Q}_{\ell_{\max}}\sigma_f) \Big|\Big|\Big|^*\\ \nonumber
					&\leq \Big|\Big|\Big| \frac{\kappa_0-\kappa}{\kappa_0} \text{DtN} \tilde{\lambda}- \frac{\kappa_0-\kappa}{\kappa_0} \text{DtN} {\psi}_{\ell_{\max}}\Big|\Big|\Big|^* + \frac{4\pi}{\kappa_0}\big|\big|\big|\mathbb{Q}_{\ell_{\max}}^{\perp}\sigma_f \big|\big|\big|^*\\ \nonumber
					&= \Big|\Big|\Big|\frac{\kappa_0-\kappa}{\kappa_0} \Big(\tilde{\lambda} - \mathbb{P}_0^{\perp}{\psi}_{\ell_{\max}}\Big) \Big|\Big|\Big|+\frac{4\pi}{\kappa_0}\big|\big|\big|\mathbb{Q}_{\ell_{\max}}^{\perp}\sigma_f \big|\big|\big|^*\\
					&\leq \max \Big\vert \frac{\kappa_0-\kappa}{\kappa_0}\Big\vert ||| \tilde{\lambda} - \mathbb{P}_0^{\perp}{\psi}_{\ell_{\max}}||| + \frac{4\pi}{\kappa_0} \big|\big|\big|\mathbb{Q}_{\ell_{\max}}^{\perp}\sigma_f \big|\big|\big|^*. \label{eq:NewNew}
					\end{align}
					
					Let $\lambda_{\ell_{\max}} \in W_0^{\ell_{\max}}$ denote the solution to Equation \eqref{eq:Galerkin2b} of the modified Galerkin discretisation. The first term in the bound \eqref{eq:NewNew} can then be written as
					\begin{align}\label{eq:NewNew2}
					||| \tilde{\lambda} - \mathbb{P}_0^{\perp}{\psi}_{\ell_{\max}}||| \leq ||| \lambda_{\ell_{\max}} -\mathbb{P}_0^{\perp}\psi_{\ell_{\max}}|||+ ||| \tilde{\lambda} -\lambda_{\ell_{\max}} |||.
					\end{align}
					
					The first term in Inequality \eqref{eq:NewNew2} can be simplified as follows: We first define the mapping $\widetilde{\mathcal{A}}_{\ell_{\max}} \colon W_0^{\ell_{\max}} \rightarrow W_0^{\ell_{\max}}$ as $\widetilde{\mathcal{A}}_{\ell_{\max}} := \mathbb{P}_0^{\perp} \mathcal{A}_{\ell_{\max}}\mathbb{P}_0^{\perp}$. Thus, $\widetilde{\mathcal{A}}_{\ell_{\max}}$ is the operator associated with the Galerkin discretisation of the ``reduced'' bilinear form defined through Definition \ref{def:Atilde}. We therefore obtain from Lemma \ref{lem:inf-sup2} that 
					\begin{align*}
					||| \lambda_{\ell_{\max}} -\mathbb{P}_0^{\perp}\psi_{\ell_{\max}}||| &= \big|\big|\big| \widetilde{\mathcal{A}}_{\ell_{\max}}^{-1} \widetilde{\mathcal{A}}_{\ell_{\max}}\big(\lambda_{\ell_{\max}} -\mathbb{P}_0^{\perp}\psi_{\ell_{\max}}\big)\big|\big|\big| \leq \frac{1}{\beta_{\tilde{\mathcal{A}}}}\big|\big|\big| \widetilde{\mathcal{A}}_{\ell_{\max}}\big(\lambda_{\ell_{\max}} -\mathbb{P}_0^{\perp}\psi_{\ell_{\max}}\big)\big|\big|\big|.
					\end{align*}
					
					In order to simplify this last bound, we first use Equation \eqref{eq:star}, the definitions of the operators $\widetilde{\mathcal{A}}_{\ell_{\max}}$ and ${\mathcal{A}}_{\ell_{\max}}$ together with a simple calculation to deduce that
					\begin{align*}
					\widetilde{\mathcal{A}}_{\ell_{\max}}\mathbb{P}_0^{\perp}\psi_{\ell_{\max}} = \mathbb{P}_0^{\perp} \mathcal{A}_{\ell_{\max}}\mathbb{P}_0^{\perp}\psi_{\ell_{\max}} = \frac{4\pi}{\kappa_0}\mathbb{P}_0^{\perp} \mathcal{V}_{\ell_{\max}}\sigma_f=\frac{4\pi}{\kappa_0}\mathbb{P}_0^{\perp} \mathbb{P}_{\ell_{\max}} \mathcal{V}\mathbb{Q}_{\ell_{\max}}\sigma_f.
					\end{align*}
					 \indent A similar calculation using the definition of $\lambda_{\ell_{\max}}$ (see Equation \eqref{eq:Galerkin2b}) yields
					\begin{align*}
					\widetilde{\mathcal{A}}_{\ell_{\max}}\lambda_{\ell_{\max}} = \frac{4\pi}{\kappa_0}\mathbb{P}_0^{\perp} \mathbb{P}_{\ell_{\max}} \mathcal{V}\sigma_f.
					\end{align*}
					
					We can therefore deduce that
					\begin{align*}
					||| \lambda_{\ell_{\max}} -\mathbb{P}_0^{\perp}\psi_{\ell_{\max}}||| \leq \frac{1}{\beta_{\tilde{\mathcal{A}}}}\big|\big|\big| \widetilde{\mathcal{A}}\big(\lambda_{\ell_{\max}} -\mathbb{P}_0^{\perp}\psi_{\ell_{\max}}\big)\big|\big|\big|\leq \frac{4\pi}{\kappa_0\beta_{\tilde{\mathcal{A}}}}\Big|\Big|\Big|\mathbb{P}_0^{\perp} \mathbb{P}_{\ell_{\max}}\mathcal{V}\mathbb{Q}_{\ell_{\max}}^{\perp}\sigma_f\big)\Big|\Big|\Big|.
					\end{align*}
					
					Since $\text{DtN} \colon \breve{H}^{\frac{1}{2}}(\partial \Omega)\rightarrow \breve{H}^{-\frac{1}{2}}(\partial \Omega)$ is an isomorphism and thus invertible, we can define $\Phi_{\ell_{\max}}:=\text{DtN}^{-1}\mathbb{Q}_{\ell_{\max}}^{\perp}\sigma_f$. We then obtain
					\begin{align*}
					\frac{4\pi}{\kappa_0\beta_{\tilde{\mathcal{A}}}}\Big|\Big|\Big|\mathbb{P}_0^{\perp} \mathbb{P}_{\ell_{\max}}\mathcal{V}\mathbb{Q}_{\ell_{\max}}^{\perp}\sigma_f\big)\Big|\Big|\Big| &=\frac{4\pi}{\kappa_0\beta_{\tilde{\mathcal{A}}}}\Big|\Big|\Big| \mathbb{P}_{\ell_{\max}}\mathbb{P}_0^{\perp}\mathcal{V}\mathbb{Q}_{\ell_{\max}}^{\perp}\sigma_f\big)\Big|\Big|\Big| \leq \frac{4\pi}{\kappa_0\beta_{\tilde{\mathcal{A}}}}\Big|\Big|\Big|\mathbb{P}_0^{\perp} \mathcal{V}\text{DtN}\Phi_{\ell_{\max}}\Big|\Big|\Big|\\
					&\leq \frac{4\pi}{\kappa_0\beta_{\tilde{\mathcal{A}}}} \frac{c^{\frac{3}{2}}_{\mathcal{K}} c_{\rm equiv}}{\sqrt{c_{\mathcal{V}}}}||| \Phi_{\ell_{\max}}||| = \frac{4\pi}{\kappa_0\beta_{\tilde{\mathcal{A}}}} \frac{c^{\frac{3}{2}}_{\mathcal{K}} c_{\rm equiv}}{\sqrt{c_{\mathcal{V}}}}||| \mathbb{Q}_{\ell_{\max}}^{\perp}\sigma_f|||^*,
					\end{align*}
					where the first step in the second line follows from the arguments used in the proof of Lemma \ref{lem:contin}.
					
					In order to simplify the second term in the Inequality \eqref{eq:NewNew2} we use the quasi-optimality result Lemma \ref{lem:quasi1}:
					\begin{align*}
					||| \tilde{\lambda} -\lambda_{\ell_{\max}} ||| &\leq  \Big(1+\frac{C_{\tilde{\mathcal{A}}}}{\beta_{\tilde{\mathcal{A}}}}\Big) \inf_{\psi \in W_0^{\ell_{\max}}}||| \tilde{\lambda} - \psi|||.
					\end{align*}
					
					Using again the fact that $\text{DtN} \colon \breve{H}^{\frac{1}{2}}(\partial \Omega)\rightarrow \breve{H}^{-\frac{1}{2}}(\partial \Omega)$ is invertible, we deduce from Equation \eqref{eq:Ben1} that
					\begin{align*}
					\tilde{\lambda}= \frac{\kappa_0}{\kappa_0-\kappa}\text{DtN}^{-1} \mathbb{Q}_0^{\perp}\nu -\frac{\kappa_0}{\kappa_0-\kappa}\frac{4\pi}{\kappa_0}\text{DtN}^{-1}\mathbb{Q}_0^{\perp}\sigma_f.
					\end{align*}
					
					Since the Dirichlet-to-Neumann mapping is bijective on $W^{\ell_{\max}}_0$, we can therefore write
					\begin{align*}
					\inf_{\psi \in W_0^{\ell_{\max}}}||| \tilde{\lambda} - \psi||| &= \inf_{ \text{DtN}^{-1}\sigma_0 \in W_0^{\ell_{\max}}}\Big|\Big|\Big| \frac{\kappa_0}{\kappa_0-\kappa}\text{DtN}^{-1} \mathbb{Q}_0^{\perp}\Big(\nu-\frac{4\pi}{\kappa_0}\sigma_f\Big) - \text{DtN}^{-1}\sigma_0\Big|\Big|\Big|\\
					&=\inf_{\sigma_0 \in W_0^{\ell_{\max}}}\Big|\Big|\Big| \frac{\kappa_0}{\kappa_0-\kappa}\text{DtN}^{-1}\Big(\mathbb{Q}_0^{\perp}\Big(\nu-\frac{4\pi}{\kappa_0}\sigma_f\Big) - \sigma_0\Big)\Big|\Big|\Big|\\
					&\leq \frac{1}{\min \Big\vert \frac{\kappa-\kappa_0}{\kappa_0}\Big \vert}\inf_{\sigma_0 \in W_0^{\ell_{\max}}} \Big|\Big|\Big|\mathbb{Q}_0^{\perp}\Big(\nu -\frac{4\pi}{\kappa_0}\sigma_f\Big) -\sigma_0 \Big|\Big|\Big|^*.
					\end{align*}
					
					In the above infimum, we may pick $\sigma_0=\mathbb{Q}_0^{\perp}\mathbb{Q}_{\ell_{\max}}\Big(\nu -\frac{4\pi}{\kappa_0}\sigma_f\Big)$ and use the triangle inequality to obtain
					\begin{align*}
					\inf_{\psi \in W_0^{\ell_{\max}}}||| \tilde{\lambda} - \psi||| \leq \frac{1}{\min\Big\vert \frac{\kappa-\kappa_0}{\kappa_0}\Big \vert} \Big|\Big|\Big|\mathbb{Q}_{\ell_{\max}}^{\perp}\nu \Big|\Big|\Big|^* + \frac{1}{\min\Big\vert \frac{\kappa-\kappa_0}{\kappa_0}\Big \vert}\frac{4\pi}{\kappa_0} \Big|\Big|\Big|\mathbb{Q}_{\ell_{\max}}^{\perp}\sigma_f \Big|\Big|\Big|^*.
					\end{align*}
					
					Using the above calculations, we can finally bound the original Inequality \eqref{eq:NewNew} as 
					\begin{align*}
					||| \nu - \nu_{\ell_{\max}}|||^* &\leq \frac{\max \Big\vert \frac{\kappa_0-\kappa}{\kappa_0}\Big\vert}{\min\Big\vert \frac{\kappa-\kappa_0}{\kappa_0}\Big \vert} \Big(1+\frac{C_{\tilde{\mathcal{A}}}}{\beta_{\tilde{\mathcal{A}}}}\Big) \left( \Big|\Big|\Big|\mathbb{Q}_{\ell_{\max}}^{\perp}\nu \Big|\Big|\Big|^* + \frac{4\pi}{\kappa_0} \Big|\Big|\Big|\mathbb{Q}_{\ell_{\max}}^{\perp}\sigma_f\Big|\Big|\Big|^*\right)\\
					&+ \frac{4\pi}{\kappa_0} \frac{1}{\beta_{\tilde{\mathcal{A}}}} \max \Big\vert \frac{\kappa_0-\kappa}{\kappa_0}\Big\vert \frac{c^{\frac{3}{2}}_{\mathcal{K}} c_{\rm equiv}}{\sqrt{c_{\mathcal{V}}}}  \big|\big|\big|\mathbb{Q}_{\ell_{\max}}^{\perp}\sigma_f \big|\big|\big|^*+ \frac{4\pi}{\kappa_0} \big|\big|\big|\mathbb{Q}_{\ell_{\max}}^{\perp}\sigma_f \big|\big|\big|^*.
					\end{align*}
					
					Using the fact that $\max \Big\vert \frac{\kappa_0-\kappa}{\kappa_0}\Big\vert \frac{c^{\frac{3}{2}}_{\mathcal{K}} c_{\rm equiv}}{\sqrt{c_{\mathcal{V}}}}  \leq C_{\widetilde{\mathcal{A}}}$, we therefore obtain
					\begin{align*}
					||| \nu - \nu_{\ell_{\max}}|||^* &\leq \frac{\max \Big\vert \frac{\kappa_0-\kappa}{\kappa_0}\Big\vert}{\min\Big\vert \frac{\kappa-\kappa_0}{\kappa_0}\Big \vert} \Big(1+\frac{C_{\tilde{\mathcal{A}}}}{\beta_{\tilde{\mathcal{A}}}}\Big) \left( \Big|\Big|\Big|\mathbb{Q}_{\ell_{\max}}^{\perp}\nu \Big|\Big|\Big|^* + \frac{4\pi}{\kappa_0} \Big|\Big|\Big|\mathbb{Q}_{\ell_{\max}}^{\perp}\sigma_f\Big|\Big|\Big|^*\right)+\Big(1+\frac{C_{\widetilde{\mathcal{A}}}}{\beta_{\tilde{\mathcal{A}}}}\Big)\frac{4\pi}{\kappa_0} \big|\big|\big|\mathbb{Q}_{\ell_{\max}}^{\perp}\sigma_f \big|\big|\big|^*\\
					&\leq \frac{\max \Big\vert \frac{\kappa_0-\kappa}{\kappa_0}\Big\vert}{\min\Big\vert \frac{\kappa-\kappa_0}{\kappa_0}\Big \vert} \Big(1+\frac{C_{\tilde{\mathcal{A}}}}{\beta_{\tilde{\mathcal{A}}}}\Big) \left( \Big|\Big|\Big|\mathbb{Q}_{\ell_{\max}}^{\perp}\nu \Big|\Big|\Big|^* + \frac{8\pi}{\kappa_0} \Big|\Big|\Big|\mathbb{Q}_{\ell_{\max}}^{\perp}\sigma_f\Big|\Big|\Big|^*\right),
					\end{align*}
					as claimed.				
				\end{proof}
			} \vspace{2mm}
			
			\subsection{Proofs of the Main Results}~
			
			We begin with the proof of Theorem \ref{thm:1}, which involves a priori error estimates and convergence rates. \vspace{3mm}
			
			{\begin{proof}[\large{\textbf{Proof of Theorem \ref{thm:1}:}}]~
					
					Consider the setting of Theorem \ref{thm:1}. We first observe that for all $s \geq 0$, $\sigma_f \in H^s(\partial \Omega)$ implies that $\nu \in H^s(\partial \Omega)$ (see, e.g., \cite[Section 9.1.4]{2ndkind1}). \vspace{2mm}
					
					Next, let $j \in \{1, \ldots, N\}$ and let $\nu_j,~ \sigma_{f, j} \in {H}^{s}(\partial \Omega_j)$ be defined as $\nu_j:=\nu\vert_{\partial \Omega_j}$ and $\sigma_{f,j}:=\sigma_f\vert_{\partial \Omega_j}$. It follows that there exist coefficients $[\nu_j]_\ell^m, [\sigma_{f,j}]_{\ell}^m$, ~ $\ell \in \mathbb{N}_0, ~-\ell \leq m \leq +\ell$ such that for all $\bold{x} \in \partial \Omega_j$ it holds that
					\begin{align*}
					\nu_j (\bold{x}) = \sum_{\ell=0}^\infty \sum_{m=-\ell}^{m=+\ell} [\nu_j]_\ell^m \mathcal{Y}_\ell^m\left(\frac{\bold{x}-\bold{x}_j}{\vert \bold{x}-\bold{x}_j\vert}\right),\qquad \text{and} \qquad
					\sigma_{f, j} (\bold{x}) = \sum_{\ell=0}^\infty \sum_{m=-\ell}^{m=+\ell} [\sigma_{f, j}]_\ell^m \mathcal{Y}_\ell^m\left(\frac{\bold{x}-\bold{x}_j}{\vert \bold{x}-\bold{x}_j\vert}\right).
					\end{align*}
					
					Using Definition \ref{def:NewNorm} of the $||| \cdot |||^*$ norm and Definition \ref{def:fin_proj} of the projection operator $\mathbb{Q}_{\ell_{\max}}^{\perp}$ we obtain that
					\begin{align*}
					{\Big|\Big|\Big|\mathbb{Q}_{\ell_{\max}}^{\perp}\nu\Big|\Big|\Big|^*}^2 &\leq \sum_{j=1}^N r_j^2\sum_{\ell=\ell_{\max}+1}^\infty \sum_{m=-\ell}^{m=+\ell} \left(\frac{\ell}{r_j}\right)^{-1}\left([\nu_j]_\ell^m\right)^2,\\
					\intertext{and}
					{\big|\big|\big|\mathbb{Q}_{\ell_{\max}}^{\perp}\sigma_f \big|\big|\big|^*}^2 &\leq \sum_{j=1}^N r_j^2\sum_{\ell=\ell_{\max}+1}^\infty \sum_{m=-\ell}^{m=+\ell} \left(\frac{\ell}{r_j}\right)^{-1}\left([\sigma_{f, j}]_\ell^m\right)^2.
					\end{align*}
					
					Using the definition of the $||| \cdot |||_{{H}^s(\partial \Omega)}$ from Equation \eqref{eq:frac20} and standard arguments from the error analysis of spectral methods then yields that
					\begin{align*}
					{||| \nu-\nu_{\ell_{\max}}|||^*} \leq \frac{\max \Big\vert \frac{\kappa-\kappa_0}{\kappa_0}\Big\vert}{\min \Big\vert \frac{\kappa-\kappa_0}{\kappa_0}\Big\vert}\left(1 + \frac{C_{\tilde{\mathcal{A}}}}{\beta_{\tilde{\mathcal{A}}}}\right)\left(\frac{\max r_j}{\ell_{\max}+1}\right)^{s+\frac{1}{2}} \left(\Big|\Big|\Big| \mathbb{Q}_{0}^{\perp}\nu\Big|\Big|\Big|_{{H}^{s}(\partial \Omega)} + \frac{8\pi}{\kappa_0} \Big|\Big|\Big| \mathbb{Q}_{0}^{\perp}\sigma_f\Big|\Big|\Big|_{H^s(\partial \Omega)}\right).
					\end{align*}

					The convergence rates for the total electrostatic energy follow by observing that the Cauchy-Schwarz inequality yields
					\begin{align*}
					\vert \mathcal{E}_{\sigma_f}(\nu) -\mathcal{E}_{\sigma_f}(\nu_{\ell_{\max}}) \vert= \left \langle \nu - \nu_{\ell_{\max}}, \mathcal{V}\sigma_f\right\rangle_{\partial \Omega}
					\leq ||| \nu - \nu_{\ell_{\max}} |||^* |||\mathcal{V}\sigma_f |||.
					\end{align*}
			\end{proof}}

			{\begin{proof}[\large{\textbf{Proof of Theorem \ref{thm:2}:}}]~
					
					Consider the setting of Theorem \ref{thm:2}. We first observe that since $\sigma_f \in C^{\infty}(\partial \Omega)$, the regularity theory for boundary integral equations (see, e.g., \cite[Section 9.1.4]{2ndkind1}) implies that $\nu \in C^{\infty}(\partial \Omega)$. Next, let us focus on obtaining an expression for the norm of the induced surface charge $\nu$. To this end, let $j \in \{1, \ldots, N\}$ and let $\nu_j \in C^{\infty}(\partial \Omega_j)$ be defined as $\nu_j:=\nu\vert_{\partial \Omega_j}$. It follows that there exist coefficients $[\nu_j]_{\ell}^m, ~ \ell \in \mathbb{N}_0, -\ell \leq m \leq \ell$ such that for all $x \in {\partial \Omega_j}$ it holds that
					\begin{align*}
					\nu_j(\bold{x})= \sum_{\ell=0}^{\infty}\sum_{m=-\ell}^{m=+\ell} [\nu_j]_{\ell}^m \mathcal{Y}_{\ell}^m \Big(\frac{\bold{x}-\bold{x}_j}{\vert \bold{x}-\bold{x}_j \vert}\Big).
					\end{align*}
					
					Let $\mathcal{E}_{\mathcal{H}}\nu_j \in C^{\infty}(\overline{\Omega_j})$ be the harmonic extension of $\nu_j$ inside the ball $\Omega_j$. Then for all $x \in \overline{\Omega_j}$ it holds that
					\begin{align}\label{eq:added1}
					\big(\mathcal{E}_{\mathcal{H}}\nu_j\big)(\bold{x})= \sum_{\ell=0}^{\infty}\sum_{m=-\ell}^{m=+\ell} [\nu_j]_{\ell}^m \Big(\frac{\vert \bold{x}-\bold{x}_j\vert}{r_j}\Big)^{\ell}\mathcal{Y}_{\ell}^m \Big(\frac{\bold{x}-\bold{x}_j}{\vert \bold{x}-\bold{x}_j \vert}\Big).
					\end{align}
					
					Using Equation \eqref{eq:added1}, it is straightforward to verify that for all integers $k \in \mathbb{N}_0$ it holds that
					\begin{align} \label{eq:thm21}
					||| \mathbb{Q}_0^{\perp}\nu_j |||^2_{{H}^{k}(\partial \Omega_j)}&= \int_{\partial \Omega_j}  \big(\mathcal{E}_{\mathcal{H}}\nu_j\big)(\bold{x}) \frac{\partial^{2k} \big(\mathcal{E}_{\mathcal{H}}\nu_j\big)(\bold{x})}{\partial \eta^{2k}}\, d\bold{x},
					\end{align}
					where $\mathbb{Q}^{\perp}_0\colon H^{-\frac{1}{2}}(\partial \Omega) \rightarrow \breve{H}^{-\frac{1}{2}}(\partial \Omega)$ is the projection operator defined through Lemma \ref{lem:decomp} and $\eta \colon \partial \Omega_j \rightarrow \mathbb{R}^3$ is the unit outward-pointing normal vector. On the other hand, we have by assumption that $\mathcal{E}_{\mathcal{H}}\nu_j$ is analytic on $\overline{\Omega_j}$. Therefore, there exists some constant $C_{\nu_j} >1$ that depends on the function $\nu_j$ such that for all $k \in \mathbb{N}_0$  and $\bold{x} \in \partial \Omega_j$ it holds that
					\begin{align*}
					\Big \vert \frac{\partial^k \big(\mathcal{E}_{\mathcal{H}}\nu_j\big)(\bold{x})}{\partial \eta^k}\Big \vert \leq C^{k+1}_{\nu_j} k!.
					\end{align*}
					
					Defining the constant $C_{\nu}:= \max_{j} C_{\nu_j}$, we therefore obtain from Equation \eqref{eq:thm21} that
					\begin{align*}
					||| \mathbb{Q}_0^{\perp}\nu_j |||^2_{{H}^{k}(\partial \Omega_j)}\leq 4\pi r^2_j C^{2k+2}_{\nu_j} (2k)!, ~~\text{so that}~~
					\frac{1}{N}||| \mathbb{Q}_0^{\perp}\nu |||^2_{{H}^{k}(\partial \Omega)}\leq 4\pi \max_j r^2_jC^{2k+2}_{\nu} (2k)!,
					\end{align*}
					
					A similar calculation which uses the fact that the harmonic extension of $\sigma_f$ is analytic on $\overline{\Omega^-}$ yields that there exist some constant $C_{\sigma_f}$ depending on $\sigma_f$ such that
					\begin{align*}
					\frac{1}{N}||| \mathbb{Q}_{0}^{\perp}\sigma_f|||^2_{{H}^{k}(\partial \Omega)}\leq 4\pi \max_j r^2_jC^{2k+2}_{\sigma_f} (2k)!.
					\end{align*}
					
					The remainder of the proof is standard. Indeed, we define $C_{\nu, \sigma_f}:= \max\left\{ C_{\nu}, \left(\frac{8\pi}{\kappa_0}\right)^{\frac{1}{2k+2}}C_{\sigma_f}\right\}$ and we use the error estimate from Theorem \ref{thm:1} to obtain 
\begin{align*}
					\frac{1}{N}{||| \nu-\nu_{\ell_{\max}}|||^*}^2 &\leq 8\pi \max_j r^2_j\frac{\max \Big\vert \frac{\kappa-\kappa_0}{\kappa_0}\Big\vert^2}{\min \Big\vert \frac{\kappa-\kappa_0}{\kappa_0}\Big\vert^2}\left(1 + \frac{C_{\tilde{\mathcal{A}}}}{\beta_{\tilde{\mathcal{A}}}}\right)^2 \left(C_{\nu}^{2k+2}(2k)!+ \frac{8\pi}{\kappa_0} C_{\sigma_f}^{2k+2}(2k)!\right) \left(\frac{\max r_j}{\ell_{\max}+1}\right)^{1+2k}\\
						&\leq 8\pi \max_j r^2_j\frac{\max \Big\vert \frac{\kappa-\kappa_0}{\kappa_0}\Big\vert^2}{\min \Big\vert \frac{\kappa-\kappa_0}{\kappa_0}\Big\vert^2}\left(1 + \frac{C_{\tilde{\mathcal{A}}}}{\beta_{\tilde{\mathcal{A}}}}\right)^2 \left(\frac{\max r_j}{\ell_{\max}+1}\right)^{1+2k}C_{\nu, \sigma_f}^{2k+2}(2k)!.
					\end{align*}

					
					Stirling's formula then yields that
					\begin{align*}
					\left(\frac{\max r_j}{\ell_{\max}+1}\right)^{1+2k} C^{2k+2}_{\nu, \sigma_f}(2k)! \leq \left(\frac{\max r_j}{\ell_{\max}+1}\right)^{1+2k} C^{2k+2}_{\nu, \sigma_f}e^{-2k+1}(2k)^{2k+\frac{1}{2}}.
					\end{align*}
					
					In particular, for $\ell_{\max}$ sufficiently large, we can choose $\alpha \in \big[\frac{1}{4C_{\nu, \sigma_f}}, \frac{1}{2C_{\nu, \sigma_f}}\big]$ such that $k= \alpha\frac{\ell_{\max}+1}{\max r_j} \in \mathbb{N}$. We then see that
					
					\begin{align*}
					\left(\frac{\max r_j}{\ell_{\max}+1}\right)^{1+2k} C^{2k+2}_{\nu, \sigma_f}e^{-2k+1}(2k)^{2k+\frac{1}{2}}&=\Big(\frac{\alpha}{k}\Big)^{1+2k} C^{2k+2}_{\nu, \sigma_f}e^{-2k+1}(2k)^{2k+\frac{1}{2}}\\ 
					&=\alpha^{1+2k} C^{2k+2}_{\nu, \sigma_f}e^{-2k+1}2^{2k+\frac{1}{2}}k^{-\frac{1}{2}}\\ 
					&= \frac{\alpha C^2_{\nu, \sigma_f}e\sqrt{2}}{\sqrt{k}} \left({4\alpha^2 C^2_{\nu, \sigma_f}}\frac{1}{e^2}\right)^k\\
					&\leq \frac{\sqrt{\alpha}}{\sqrt{\frac{\ell_{\max}+1}{\max r_j}}}{C^2_{\nu, \sigma_f}\sqrt{2}} e^{-2k+1}\\
					&\leq \sqrt{2\max r_j}C_{\nu, \sigma_f}^2\exp\left(-2\alpha\frac{\ell_{\max}+1}{\max r_j} +1\right)\\
					&\leq \sqrt{2\max r_j}C_{\nu, \sigma_f}^2\exp\left(-\frac{1}{2C_{\nu, \sigma_f}}\frac{\ell_{\max}+1}{\max r_j} +1\right).
					\end{align*}
					
					We conclude that 
					\begin{align*}
					\frac{1}{\sqrt{N}}{||| \nu-\nu_{\ell_{\max}}|||^*} \leq \sqrt{8 \pi \max r^2_j}(2\max r_j)^{\frac{1}{4}}C_{\nu, \sigma_f}\frac{\max \Big\vert \frac{\kappa-\kappa_0}{\kappa_0}\Big\vert}{\min \Big\vert \frac{\kappa-\kappa_0}{\kappa_0}\Big\vert}\left(1 + \frac{C_{\tilde{\mathcal{A}}}}{\beta_{\tilde{\mathcal{A}}}}\right)\exp\left(-\frac{1}{4C_{\nu, \sigma_f}}\frac{\ell_{\max}+1}{\max r_j} +\frac{1}{2}\right).
					\end{align*}
					
					This completes the proof for the exponential convergence of the approximate induced surface charge. The proof for the exponential convergence of the approximate total electrostatic energy is essentially identical.
					
			\end{proof}}

			\section{Conclusion and Future Work}
			
			In this work, we presented a detailed numerical analysis of an integral equation formulation of the second kind for the induced surface charges resulting on a large number of dielectric spheres of varying radii and dielectric constants, embedded in a homogenous dielectric medium and undergoing mutual polarisation. We derived a priori error estimates and convergence rates that \emph{do not have any explicit dependence on the number of dielectric spheres $N$ in the system}. In order to achieve this, we introduced a new analysis of second kind boundary integral equations posed on spherical domains.
			
			In order to complete a scalability analysis of the numerical algorithm under consideration, it is also necessary to analyse computational aspects of the algorithm such as the conditioning of the linear system that arises from the Galerkin discretisation \eqref{eq:Galerkina}. This topic, as well as related computational considerations, is the subject of the contribution \cite{Hassan2}.
			
			From the point of view of further numerical analysis, we emphasise that the differential operator which generated all layer potentials and boundary operators in the current work was the Laplace operator. Future theoretical work could therefore involve the analysis of $N$-body systems involving more complicated differential operators. Such operators arise, for instance, in the study of wave propagation in non-homogenous media or electrostatic interactions between dielectric spheres in an ionic solvent.

			\bibliographystyle{plain}
			\bibliography{refs.bib}
			
			\newpage
			\begin{appendix}

				\section{Justification for the Equivalence of the $||| \cdot |||$ Norm} \label{sec:Appendix_A}
				
				\textbf{Notation:} We write $\breve{\mathbb{H}}:=\{u \in \mathbb{H}(\Omega^-) \colon \gamma^- u \in \breve{H}^{\frac{1}{2}}(\partial \Omega)\}.$\\
				
				Intuitively, $\breve{\mathbb{H}}$ consists of harmonic functions in $H^1(\Omega^-)$ such that the interior Dirichlet trace of these functions is of average zero. Consequently, it holds that $\breve{\mathbb{H}}$ is a Hilbert space with respect to the $H^1$ semi-norm. Henceforth, we will equip the space $\breve{\mathbb{H}}$ with the inner product given by
				\begin{align*}
				(u, v)_{\breve{\mathbb{H}}}:= \sum_{i=1}^N\int_{\Omega_i} \nabla u(\bold{x}) \cdot \nabla v(\bold{x})\, d\bold{x},
				\end{align*}
				and we observe that the associated norm $\Vert \cdot \Vert_{\breve{\mathbb{H}}}$ is equivalent to the $\Vert \cdot \Vert_{H^1(\Omega^-)}$ norm defined in Section \ref{sec:2}.
				
				\begin{lemma}\label{lem:Biject}
					The interior Dirichlet trace mapping $\gamma^- \colon\breve{\mathbb{H}} \rightarrow \breve{H}^{\frac{1}{2}}(\partial \Omega)$ and the interior Neumann trace operator $\gamma_N^- \colon \breve{\mathbb{H}} \rightarrow \breve{H}^{-\frac{1}{2}}(\partial \Omega)$ are both bijective, continuous linear operators.
				\end{lemma}
				\begin{proof}
					The proof follows from the well-posedness of the interior Dirichlet and Neumann problems for the Laplace equation on Lipschitz domains.
				\end{proof}
				
				\noindent	\textbf{Notation:} We define $\mathcal{E} \colon \breve{H}^{\frac{1}{2}}(\partial \Omega) \rightarrow \breve{\mathbb{H}}$ as the inverse of the interior Dirichlet trace operator $\gamma^- \colon\breve{\mathbb{H}} \rightarrow \breve{H}^{\frac{1}{2}}(\partial \Omega)$.

				\begin{corollary}\label{cor:norm}
					Lemma \ref{lem:Biject} implies in particular that the interior trace operator $\gamma^- \colon\breve{\mathbb{H}} \rightarrow \breve{H}^{\frac{1}{2}}(\partial \Omega)$ is an isomorphism. It follows that we can define a new norm $\Vert \cdot \Vert_{\breve{H}^{\frac{1}{2}}(\partial \Omega)}$ on the space $\breve{H}^{\frac{1}{2}}(\partial \Omega)$ that is equivalent to the Sobolev-Slobodeckij norm defined in Section \ref{sec:2} by setting for all $\lambda \in \breve{H}^{\frac{1}{2}}(\partial \Omega)$
					\begin{align*}
					\Vert \lambda \Vert_{\breve{H}^{\frac{1}{2}}(\partial \Omega)} = \Vert \mathcal{E}\lambda \Vert_{\breve{\mathbb{H}}}.
					\end{align*}
				\end{corollary}

				Lemma \ref{lem:Biject} also yields the following corollary.
				
				\begin{corollary}\label{cor:DtN}
					The Dirichlet-to-Neumann map $\text{\emph{DtN}} \colon \breve{H}^{\frac{1}{2}}(\partial \Omega) \rightarrow \breve{H}^{-\frac{1}{2}}(\partial \Omega)$ is a bijective operator.
				\end{corollary}
				
				\begin{remark}\label{rem:DtN}
					The Dirichlet-to-Neumann map $\text{DtN} \colon \breve{H}^{\frac{1}{2}}(\partial \Omega) \rightarrow \breve{H}^{-\frac{1}{2}}(\partial \Omega)$ yields an alternative characterisation of the norm $\Vert \cdot \Vert_{\breve{H}^{\frac{1}{2}}(\partial \Omega)}$. Indeed, let $u \in \breve{H}^{\frac{1}{2}}(\partial \Omega)$. Then Green's identity implies that
					\begin{align*}
					\Vert u\Vert^2_{\breve{H}^{\frac{1}{2}}(\partial \Omega)}&= \Vert \mathcal{E}u \Vert^2_{\breve{\mathbb{H}}}= \sum_{i=1}^N\int_{\Omega_i} \nabla \mathcal{E}u(\bold{x})\cdot \nabla \mathcal{E}u(\bold{x})\, d\bold{x} \\
					&= \sum_{i=1}^N\langle \text{DtN}u\vert_{\partial \Omega_i}, u\vert_{\partial \Omega_i} \rangle_{H^{-\frac{1}{2}}(\partial \Omega_i) \times H^{\frac{1}{2}}(\partial \Omega_i) }\\
					&=\langle \text{DtN}u, u \rangle_{H^{-\frac{1}{2}}(\partial \Omega) \times H^{\frac{1}{2}}(\partial \Omega) }.
					\end{align*}
				\end{remark}

				\begin{corollary}
					Combining Corollary \ref{cor:norm} and Remark \ref{rem:DtN}  yields that the norm $||| \cdot ||| \colon H^{\frac{1}{2}}(\partial \Omega) \rightarrow \mathbb{R}$ defined through Definition \ref{def:NewNorm} is indeed equivalent to the $\Vert \cdot \Vert_{H^{\frac{1}{2}}(\partial \Omega)}$ norm introduced in Section \ref{sec:2}.
				\end{corollary}
				
				\newpage
				\section{Proof of Lemma \ref{lem:equivalence}}\label{sec:Appendix_C}
				\begin{proof}
					Let $\bold{\Phi}:= (\Phi^-, \Phi^+) \in \mathbb{H}(\Omega^-) \times \mathbb{H}(\Omega^+)$ be a solution to the transmission problem \eqref{eq:3.2}. It follows from Green's representation theorem (see, e.g., \cite[Theorem 3.1.6]{Schwab}) that for each $s \in \{+, -\}$ it holds that
					\begin{align*}
					\Phi^s = \mathcal{S} \left(\gamma_N^- \Phi^--\gamma_N^+ \Phi^+ \right) \vert_{\Omega^s}.
					\end{align*}
					
					It follows from the hypothesis of the transmission problem \eqref{eq:3.2} that
					\begin{align*}
					-\gamma_N^+ \Phi^+=\frac{4\pi}{\kappa_0}\sigma_f-\frac{\kappa}{\kappa_0}\gamma_N^- \Phi^-,
					\end{align*}
					so that
					\begin{align*}
					\gamma^-\Phi^- &= \gamma^-\mathcal{S} \left(\gamma_N^- \Phi^-+\frac{4\pi}{\kappa_0}\sigma_f-\frac{\kappa}{\kappa_0}\gamma_N^- \Phi^-\right)\\
					&=\mathcal{V} \left(\frac{1}{\kappa_0}\left(\kappa_0\gamma_N^- \Phi^- - \kappa\gamma_N^- \Phi^-\right) + \frac{4\pi}{\kappa_0}\sigma_f\right) \\
					&=\mathcal{V} \left(\frac{\kappa_0 -\kappa}{\kappa_0}\gamma_N^- \Phi^-\right) + \frac{4\pi}{\kappa_0}\mathcal{V}\sigma_f\\
					&=\mathcal{V} \left(\frac{\kappa_0 -\kappa}{\kappa_0}\text{DtN} \gamma^-\Phi^-\right) + \frac{4\pi}{\kappa_0}\mathcal{V}\sigma_f.
					\end{align*}
					
					Define $\nu:=\mathcal{V}^{-1}\gamma^- \Phi^-$ and use the fact that $\mathcal{V}^{-1} \colon H^{\frac{1}{2}}(\partial \Omega) \rightarrow H^{-\frac{1}{2}}(\partial \Omega)$ is a bijection to obtain that
					\begin{align*}
					\nu = \Big(\frac{\kappa_0-\kappa}{\kappa_0}\text{DtN}\mathcal{V}\nu\Big)+ \frac{4\pi}{\kappa_0}\sigma_f.
					\end{align*}
					This completes the first part of the proof. 
					
					For the converse, let $\nu \in H^{-\frac{1}{2}}(\partial \Omega)$ be a solution to the BIE \eqref{eq:3.3a}. It follows from the jump properties of the single layer potential (see, e.g., \cite[Theorem 3.3.1]{Schwab}) that
					\begin{align*}
					\nu=\gamma_N^- \mathcal{S}\nu -\gamma_N^+ \mathcal{S}\nu.
					\end{align*}
					
					Define $(\Phi^-, \Phi^+) = \big(\mathcal{S}\nu\vert_{\Omega^-}, \mathcal{S}\nu\vert_{\Omega^+}\big)$. The definition of the single layer potential implies that we need only check the jump condition for the normal derivative. We observe that
					\begin{align*}
					\kappa\gamma_N^-\Phi^- - \kappa_0 \gamma_N^+ \Phi^+&= \kappa\text{DtN}\gamma^-\Phi^- + \kappa_0 \nu - \kappa_0 \text{DtN}\gamma^-\Phi^-\\
					&=(\kappa-\kappa_0)\text{DtN}\gamma^-\Phi^- + \kappa_0 \nu.
					\end{align*}
					
					It follows from the hypothesis of BIE \eqref{eq:3.3a} that
					\begin{align*}
					\kappa\gamma_N^-\Phi^- - \kappa_0 \gamma_N^+ \Phi^+=(\kappa-\kappa_0)\text{DtN}\gamma^-\Phi^- + \kappa_0 \nu={4\pi}\sigma_f.
					\end{align*}
					
					This completes the proof.
				\end{proof}
			\end{appendix}
			
		\end{document}